\documentclass[11pt]{amsart}
\usepackage{amstext,amsmath,amssymb,amsfonts,orcidlink}%
\usepackage{amsthm}
\usepackage{xcolor}
\usepackage{manyfoot}
\usepackage[utf8]{inputenc}
\usepackage[english]{babel}
\usepackage[T1]{fontenc} 
\usepackage{braket,mathtools}
\usepackage{hyperref,csquotes}
\usepackage[english,capitalize]{cleveref}
\usepackage{float}
\usepackage[a4paper,twoside,top=0.8in, bottom=0.7in, left=0.7in, right=0.7in]{geometry}
\usepackage{tikz}
\usetikzlibrary{patterns}

\definecolor{lava}{rgb}{0.81, 0.06, 0.13}
\definecolor{newblue}{rgb}{0.2, 0.3, 0.85}
\definecolor{mygreen}{rgb}{0.0, 0.5, 0.0}
\hypersetup{colorlinks=true, linkcolor=newblue, citecolor=newblue, urlcolor  = newblue}

\usepackage[maxbibnames=99,backend=biber, sorting=nyt, url=false, isbn=false, style=alphabetic]{biblatex}
\numberwithin{equation}{section}
\newcommand{\N}{\mathbb{N}}

\newcommand{\R}{\mathbb{R}}

\renewcommand{\S}{\mathbb{S}}
\renewcommand{\epsilon}{\varepsilon}
\renewcommand{\theta}{\vartheta}
\renewcommand{\rho}{\varrho}
\renewcommand{\phi}{\varphi}
\newcommand{\de}{\,{\rm d}}
\renewcommand{\d}{{\rm d}}
\newcommand{\st}{\ensuremath{\ :\ }} 
\newcommand{\mres}{\mathbin{\vrule height 1.6ex depth 0pt width
0.13ex\vrule height 0.13ex depth 0pt width 1.3ex}}
\newcommand{\grad}{\nabla}
\newcommand{\hausdorff}{\mathcal{H}}
\newcommand{\symmdiff}{\Delta}
\newcommand{\eps}{\epsilon}

\theoremstyle{plain}
\newtheorem{theorem}{Theorem}[section]
\newtheorem{proposition}[theorem]{Proposition}
\newtheorem{lemma}[theorem]{Lemma}
\newtheorem{corollary}[theorem]{Corollary}
\theoremstyle{remark}

\newtheorem{remark}[theorem]{Remark}
\theoremstyle{definition}
\newtheorem{definition}[theorem]{Definition}
\DeclareMathOperator{\diam}{diam}
\DeclareMathOperator{\divergence}{div}
\DeclareMathOperator{\Id}{Id}
\DeclareMathOperator{\barycenter}{bar}
\DeclareMathOperator{\iinterior}{int}
\newcommand{\interior}[1]{\iinterior\left(#1\right)}

\begin{document}


\title[Quantitative isoperimetric inequalities in capillarity problems and cones]{Quantitative isoperimetric inequalities in capillarity problems and cones in strong and barycentric forms}

\author{Davide Carazzato}
\address{Faculty of Mathematics, University of Vienna, Oskar-Morgenstern-Platz 1, 1090 Vienna, Austria} \email{davide.carazzato@univie.ac.at \orcidlink{0000-0001-7511-6698}}

\author{Giulio Pascale}
\address{Dipartimento di Matematica e Applicazioni ``Renato Caccioppoli'', Universit\'a degli Studi di Napoli ``Federico II'', via Cintia - Monte Sant'Angelo, 80126 Napoli, Italy} \email{giulio.pascale@unina.it \orcidlink{0000-0003-1680-3425}}

\author{Marco Pozzetta}
\address{Politecnico di Milano, Via Bonardi 9, 20133 Milano, Italy} \email{marco.pozzetta@polimi.it \orcidlink{0000-0002-2757-0826}}

\date{\today}

\begin{abstract}
We study quantitative isoperimetric inequalities for two different perimeter-type functionals. We first consider classical capillarity functionals, which measure the perimeter of sets in a Euclidean half-space, assigning a constant weight $\lambda \in (-1, 1)$ to the portion of the boundary that lies on the boundary of the half-space. 
We then consider the relative perimeter of sets contained in some suitable convex cone in the Euclidean space.


In both settings, we establish sharp quantitative isoperimetric inequalities in the so-called strong form.
More precisely, we show that the isoperimetric deficit of a competitor not only controls the Fraenkel asymmetry, but it also controls an oscillation asymmetry that measures how much the unit normals to the boundary of a competitor deviate from those of an isoperimetric set.\\
Our technique is also able to explicitly identify a center that can be employed to compute the asymmetries. In particular, we also derive barycentric versions of these quantitative isoperimetric inequalities (both in the classical and in the strong form).

The proofs are based on the derivation of Fuglede-type estimates for graphs defined on general spherically convex domains, in combination with a new application of the selection principle that directly provides the inequalities in barycentric form.
\end{abstract}

\maketitle

\setcounter{tocdepth}{1} 
\vspace{-1.2cm}
\tableofcontents

\vspace{-1.7cm}
\section{Introduction}

The study of quantitative stability of geometric and functional inequalities, and in particular of isoperimetric inequalities, represents a wide research line in Calculus of Variations in the last thirty years. The sharp quantitative isoperimetric inequality for the classical perimeter in $\R^n$ was first proved in \cite{FuscoMaggiPratelli}, after \cite{Fuglede, HallQuantitative, HallHaymanWeitsman}, and it states that
\begin{equation}\label{eq:zzClassicalQuantIsop}
\inf_{x_0 \in \R^n} |E \Delta B_1(x_0)|^2 \le C(n) \big[ P(E) - P(B_1) \big],
\end{equation}
for any measurable set $E\subset \R^n$ with $|E|=|B_1|$, where $P(\cdot)$ denotes the perimeter functional, and $|\cdot|$ denotes Lebesgue measure. 
The left hand side in \eqref{eq:zzClassicalQuantIsop} is the so-called Fraenkel asymmetry with respect to Euclidean balls. After \cite{FuscoMaggiPratelli}, many other proofs of the quantitative inequality have been found; in fact, quantitative isoperimetric inequalities have been proved for several different notions of perimeter functionals, such as anisotropic, weighted, or non-local perimeters \cite{FigalliMaggiPratelli, CicaleseLeonardi, AcerbiFuscoMorini, FuscoJulin, BarchiesiBrancoliniJulin, FigalliFuscoMaggiMillotMorini, CianchiFuscoMaggiPratelli, FigalliIndrei, CintiGlaudoPratelliRosOtonSerra2022}. We refer the interested reader to the very nice survey \cite{FuscoDispensa} and to references therein.\\
While \eqref{eq:zzClassicalQuantIsop} bounds from above an $L^1$-notion of distance of a competitor from the set of minimizers in terms of the perimeter deficit, it is well-known since \cite{Fuglede} that if a competitor $E$ is a sufficiently small $C^1$-graphical perturbation of a ball in terms of a function $u:\S^{n-1}\to \R$ that parametrizes $\partial E$ over the sphere, then the left hand side in \eqref{eq:zzClassicalQuantIsop} can be replaced by the $H^1$-norm of $u$. More recently, in \cite{FuscoJulin} the authors managed to prove that \eqref{eq:zzClassicalQuantIsop} can be, in fact, improved in order to bound a higher order notion of distance of a competitor from the set of minimizers in terms of its perimeter deficit. Such an inequality represents a so-called \emph{strong form} of the quantitative isoperimetric inequality and reads
\begin{equation}\label{eq:zzQuantFuscoJulin}
    \inf \left\{ \int_{\partial^*E}\left|\nu_E(x) - \frac{x - x_0}{|x - x_0|}\right|^2 \de \hausdorff^{n - 1}(x) \st x_0 \in \R^n \right\} \le C(n) \big[ P(E) - P(B_1) \big],
\end{equation}
for any measurable set $E\subset \R^n$ with $|E|=|B_1|$, where $\partial^*E$ denotes the reduced boundary of $E$, and $\nu_E$ denotes its generalized outward unit normal, see \cref{sec:SetsFinitePerimeter}. 
The left hand side in \eqref{eq:zzQuantFuscoJulin} is an oscillation asymmetry that measures how much the unit normals of a competitor deviate from those of an isoperimetric sets, see \cref{fig:asymmetry}. 
Roughly speaking, it represents an $H^1$-notion of distance of a general competitor from the set of minimizers.\\
Further quantitative inequalities in strong form have been obtained in \cite{Neumayer16} for anisotropic perimeters, in \cite{DeMason24} in the case of crystalline perimeters, and in \cite{BogeleinDuzaarFusco, BogeleinDuzaarScheven} for the problem set on the sphere or on the hyperbolic space. We will compare these works with ours below. We also mention that in \cite{BarchiesiBrancoliniJulin} it is proved a dimension-free inequality in strong form in the Gaussian space.
%

Another possible strengthening of inequalities \eqref{eq:zzClassicalQuantIsop}, \eqref{eq:zzQuantFuscoJulin} is represented by \emph{barycentric forms} of quantitative isoperimetric inequalities. In \cite{GambicchiaPratelli2025}, after \cite{BianchiniCroceHenrot2023}, it is proved that
\begin{equation}\label{eq:GambicchiaPratelli}
|E \Delta B_1(b_E)|^2 \le C(n,d) \big[ P(E) - P(B_1) \big],
\end{equation}
for any measurable set $E\subset \R^n$ with $|E|=|B_1|$ and diameter ${\rm diam}(E)\le d$, where $b_E$ is the barycenter of $E$. The dependence of the constant in \eqref{eq:GambicchiaPratelli} on $d$ cannot be removed, but the left hand side gets a remarkable improvement compared to \eqref{eq:zzClassicalQuantIsop} as the asymmetry is computed with respect to a unique ball, centered at a well-determined point.

\medskip

In this paper, by proposing a unifying approach, we will prove quantitative isoperimetric inequalities in both \emph{strong and barycentric forms} for the isoperimetric problems defined by two perimeter-type functionals: (i) the classical capillarity functional in a half-space, and (ii) the standard (relative) perimeter in a convex cone. To our knowledge, this is the first instance where a \emph{strong} quantitative isoperimetric inequality is also proved in barycentric form.

\medskip

The capillarity functional with parameter $\lambda\in(-1,1)$ in a half-space $H:=\{x_n\ge 0\}\subset \R^n$ is defined by
\[
    P_{\lambda}(E) := P(E,\{x_n>0\})-\lambda P(E,\{x_n=0\}),
\]
for any set $E\subset H$ of finite perimeter.
It is well-known that isoperimetric sets for the capillarity functional are given by (translated rescalings of) the bubble $B^{\lambda}:=(B_1(0)-\lambda e_n)\cap H$, where $e_n:=(0,\ldots,0,1)\in\R^n$ (see \cite[Section 19, Theorem 19.21]{MaggiBook}, \cite{Finn1986} and the classical works \cite{GonzalezMassariTamanini, Giusti1981, Tamanini}). The latter minimality property comes with the sharp and rigid isoperimetric inequality
\begin{equation}\label{eq:zzIsoperimetricaCapillare}
    P_{\lambda}(E)\geq  n|B^\lambda|^{\frac1n} |E|^{\frac{n-1}{n}}\qquad \forall E\subset H, \quad |E|<
    \infty.
\end{equation}
Capillarity-type problems still attract a remarkable attention. We mention here recent works on the fine regularity of almost minimizers and varifolds \cite{DePhilippisMaggi2015, ChodoshEdelenLi, DeMasiChodoshGasparettoLi}, and recent developments on sharp and rigid relative and capillarity isoperimetric inequalities outside convex bodies \cite{ChoeGhomiRitoreInequality, FuscoMorini, LiuWangWeng, Krummel, FuscoJulinMorini}.
A quantitative version of \eqref{eq:zzIsoperimetricaCapillare} for the Fraenkel asymmetry in the spirit of \eqref{eq:zzClassicalQuantIsop} has been recently proved in \cite{PascalePozzettaQuantitative, KreutzSchmidt24} (see \cref{thm:PPquantitativaFraenkelCapillare} below).

The next theorem contains our main results about the capillarity isoperimetric problem, that is the natural analogue of the previously known quantitative isoperimetric inequalities in strong form (see \cref{fig:asymmetry}).

\begin{theorem}[{Strong and barycentric quantitative inequalities in capillarity problems -- cf. \cref{thm:capillarityStrongQuantitative}, \cref{cor:bar-quantitative}, \cref{cor:barycentric-Fraenkel}}]\label{thm:MAINcapillarity}
Let $n \in \N$ with $n\ge 2$, $\lambda \in (-1,1)$. There exists $C_{\ref{thm:MAINcapillarity}}>0$ depending on $n,\lambda$ such that for any set of finite perimeter $E \subset \{x_n\ge 0\}\subset \R^n$ with $|E|=|B^\lambda|$ there holds
\begin{equation*}
    \inf\left\{\int_{\partial^* E \cap \{x_n>0\}} \left| \nu_E(x) - \frac{x-(x',-\lambda)}{|x-(x',-\lambda)|} \right|^2 \de \hausdorff^{n-1} (x) \st x' \in\{x_n=0\}
\right\} \le 
    C_{\ref{thm:MAINcapillarity}} \big[P_\lambda(E)- P_\lambda(B^\lambda)\big].
\end{equation*}
Moreover, for any $d>\diam_H B^{\lambda}$ there exists a constant $C_{\ref{thm:MAINcapillarity}}'>0$ depending on $n$, $\lambda$ and $d$ such that for every set of finite perimeter $E\subset \{x_n\ge 0\}\subset \R^n$ with $|E|=|B^\lambda|$ and $\diam_H E\leq d$ there holds
    \begin{equation*}
        |E\symmdiff (B^{\lambda}+\barycenter_H E)|^2 + \int_{\partial^*E\cap \{x_n>0\}} \left|\nu_E(x)-\frac{x-(\barycenter_H E,-\lambda)}{|x-(\barycenter_H E,-\lambda)|}\right|^2\de\hausdorff^{n-1}(x) \leq C_{\ref{thm:MAINcapillarity}}' \big[P_\lambda(E)- P_\lambda(B^\lambda)\big],
    \end{equation*}
where $\diam_H E$ is the diameter of the projection of $E$ onto $\{x_n=0\}$, and $\barycenter_H E:= \tfrac{1}{|E|} \int_E (x_1,\ldots, x_{n-1}) $.
\end{theorem}

Inequalities in \cref{thm:MAINcapillarity} are stated here for sets having fixed volume $|E|=|B^\lambda|$, but there clearly hold scaling invariant counterparts of such inequalities by normalizing the quantities involved in terms of the suitable power of $|E|$ (see \cref{thm:capillarityStrongQuantitative}, \cref{cor:bar-quantitative}, \cref{cor:barycentric-Fraenkel}).\\
In fact, we will also prove a barycentric version of the strong quantitative inequality for capillarity problems for possibly \emph{unbounded} sets, replacing $\barycenter_H E$ (whose definition could be ill posed in this case) with a nonlinear notion of barycenter, well-defined for any set of finite measure, thus getting an inequality with a constant independent of the diameter of the competitor. Such a quantitative inequality is stated in \cref{cor:tilde-bar-quantitative}. See also \cref{rem:MoreBarycentricInequlities} for additional comments.\\
We remark that the exponents in the inequalities in \cref{thm:MAINcapillarity} are sharp, as we show in \cref{subsec:sharp}. We mention that one may try to combine \cite{Neumayer16} with \cite{KreutzSchmidt24}, where a connection between capillarity perimeters and anisotropic perimeters is found, to obtain a quantitative isoperimetric inequality in strong form for the capillarity problem. However, since the anisotropic perimeter corresponding to the capillarity one is non-smooth, one would get a non-sharp exponent in the final inequality (see \cite[Theorem~1.1]{Neumayer16}). More importantly, the oscillation asymmetry obtained following such an approach would lack of a clear geometric interpretation. In fact, as we comment below, our oscillation asymmetry is substantially different from that considered in the sharp strong quantitative inequality in \cite[Theorem 1.3]{Neumayer16}.

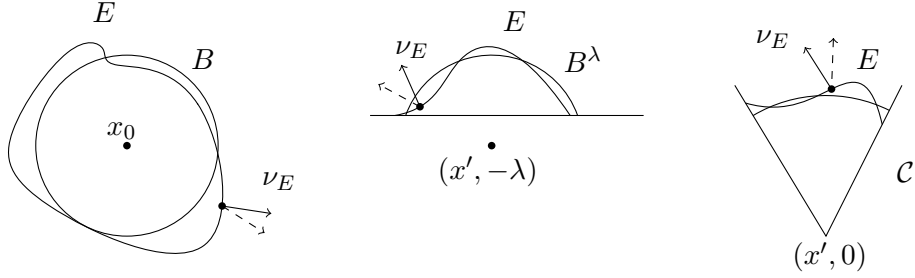
\begin{figure}[H]
    \centering
    \begin{tikzpicture}[scale=0.025]
      \draw (80, 64) circle[radius=48];
      \draw
        (117.3333, 86)
         .. controls (130.6667, 65.3333) and (137.3333, 22.6667) .. (120, 10.6667)
         .. controls (102.6667, -1.3333) and (61.3333, 17.3333) .. (39.3333, 32)
         .. controls (17.3333, 46.6667) and (14.6667, 57.3333) .. (20, 72)
         .. controls (25.3333, 86.6667) and (38.6667, 105.3333) .. (48.6667, 113.3333)
         .. controls (58.6667, 121.3333) and (65.3333, 118.6667) .. (67.3333, 115.3333)
         .. controls (69.3333, 112) and (66.6667, 108) .. (75.3333, 106.6667)
         .. controls (84, 105.3333) and (104, 106.6667) .. cycle;
      \fill (80, 64) circle (2cm);
      \fill (130.0941, 31.9908) circle (2cm);
      \draw[->] (130.0941, 31.9908) -- (156, 28);
      \draw[shift={(130.094, 31.991)}, scale=0.446, dashed, ->] (0, 0) -- (50.094, -32.0092);
      
      \draw (208, 80) -- (352, 80);
      \draw (226.7452, 80)
         arc[start angle=-160.5288, end angle=-19.4712, x radius=48, y radius=-48];
      \draw (221.1323, 80) .. controls (244, 84) and (246, 100) .. (256, 110) .. controls (266, 120) and (284, 124) .. (313.3455, 80);
      \fill (272, 64) circle (2cm);
      
      \fill (234.265, 84.6299) circle (2cm);
      \draw[shift={(234.265, 84.63)}, scale=0.95,->] (0, 0) -- (-10.265, 23.3701);
      \draw[shift={(234.265, 84.63)}, scale=0.58, dashed, ->] (0, 0) -- (-37.735, 20.6299);
      \draw (400, 96) -- (448, 16) -- (488, 96);
      \draw (409.6253, 79.9578) arc[start angle=-120.9638, end angle=-63.4349, x radius=74.587, y radius=-74.587];
      \draw (405.632, 86.6133) .. controls (428, 80) and (442, 90) .. (453, 95) .. controls (464, 100) and (472, 100) .. (477.6018, 75.2037);
      
      \draw[shift={(450.918, 94.017)}, scale=0.7412,->] (0, 0) -- (-18.9178, 29.9826);
      \draw[shift={(451.93, 121.091)}, scale=0.347, dashed,<-] (0, 0) -- (-2.918, -78.0174);
        \fill (450.9178, 94.0174) circle (2cm);

      \draw (77.242, 72) node {$x_0$};
      \draw (270, 50.49) node {$(x',-\lambda)$};
      
      \draw (450, 5) node {$(x',0)$};
      \draw (230, 115) node {$\nu_E$};
      \draw (420, 120.527) node {$\nu_E$};
      \draw (160, 45) node {$\nu_E$};
      \draw (284, 130) node {$E$};
      \draw (470, 110) node {$E$};
      \draw (68.071, 135) node {$E$};
    
      \draw (120,110) node {$B$};
      \draw (320,110) node {$B^{\lambda}$};
      \draw (490,50) node {$\mathcal{C}$};
      
    \end{tikzpicture}
    \caption{Representation of the oscillation asymmetries we consider. Left: setting considered in \cite{FuscoJulin}. Center: capillarity model. Right: section of a subset in a cone ($x'$ lies in an axis that is orthogonal to the picture). In all cases, the dashed lines represent the normal vector to the isoperimetric set centered at $x_0$, $(x',-\lambda)$ and $(x',0)$ respectively.}
    \label{fig:asymmetry}
\end{figure}

\vspace{-0.3cm}

For a measurable set $E$ contained in a closed convex cone $\mathcal{C}\subset\R^n$ with non-empty interior $\interior{\mathcal{C}}$, the relative perimeter of $E$ in $\mathcal{C}$ is simply $P_{\mathcal{C}}(E):= P(E, \interior{\mathcal{C}})$. It is well-known since \cite{LionsPacella1990} that isoperimetric sets for $P_{\mathcal{C}}$ are given by relative balls in $\mathcal{C}$ centered at tips, i.e., if $\mathcal{C}=\R^m\times\widetilde{\mathcal{C}}\subset\R^n$ where $\tilde{\mathcal{C}}\subset\R^{n-m}$ does not contain lines, isoperimetric sets are given by $B_r(x',0) \cap \mathcal{C}$, for any $(x',0) \in \R^m\times\{0\}$ and $r>0$. There exists a corresponding sharp quantitative isoperimetric inequality for the Fraenkel asymmetry in the spirit of \eqref{eq:zzClassicalQuantIsop} for the isoperimetric problem in cones, which was proved in \cite{FigalliIndrei} (see \cref{thm:QuantitativeConesFraenkel} below).
We also mention that isoperimetric problems in cones with respect to weighted notions of volume or perimeter have been largely studied, see e.g. \cite{BrockChiacchioMercaldo12, CabreRosOtonSerra12BallsNonradialWeights, AlvinoBrockChiacchioMercaldoPosteraro19, CabreRosOtonSerra} and references therein. A remarkable quantitative inequality with respect to Fraenkel asymmetry for the problem considered in \cite{CabreRosOtonSerra} has been proved in \cite{CintiGlaudoPratelliRosOtonSerra2022}.

The next statement contains our main results about the isoperimetric problem in convex cones  (see also \cref{fig:asymmetry}).

\begin{theorem}[{Strong and barycentric quantitative inequalities in cones -- cf. \cref{thm:ConesStrongQuantitative}, \cref{cor:barycentric-ineq-cones}}]\label{thm:MAINcones}
Let $\mathcal{C}=\R^m\times\widetilde{\mathcal{C}}\subset\R^n$ be a piecewise $C^2$ closed convex cone (see \cref{def:piecewise-C2}) with non-empty interior $\interior{\mathcal{C}}$, where $\tilde{\mathcal{C}}\subset\R^{n-m}$ does not contain lines.
Then there exists  $ C_{\ref{thm:MAINcones}}>0$ depending on $\mathcal{C}$ such that for any set of finite perimeter $E\subset \mathcal{C}$ with $|E|=|B_1 \cap \mathcal{C}|$ there holds
\begin{equation*}
     \inf\left\{\int_{\partial^* E \cap \interior{\mathcal{C}}} \left| \nu_E(x) - \frac{x-(x',0)}{|x-(x',0)|} \right|^2 \de \hausdorff^{n-1} (x) \st x' \in \R^m \right\}
        \le C_{\ref{thm:MAINcones}} \big[P(E,\interior{\mathcal{C}}) - P(B_1 \cap \mathcal{C}, \interior{\mathcal{C}})\big].
\end{equation*}
Moreover, for any $d>\diam_{\R^m}(B_1\cap \mathcal{C})$ there exists a constant $ C_{\ref{thm:MAINcones}}'>0$ depending on $\mathcal{C}$ and $d$ such that for every set of finite perimeter $E\subset \mathcal{C}$ with $|E|=|B_1\cap \mathcal{C}|$ and $\diam_{\R^m}E\leq d$ there holds
    \[
    \begin{split}
        |E \symmdiff((B_1\cap \mathcal{C}) + (\barycenter_{\R^m}E,0))|^2&+\int_{\partial^*E\cap \interior{\mathcal{C}}}\left|\nu_E(x)-\frac{x-(\barycenter_{\R^m}E,0)}{|x-(\barycenter_{\R^m}E,0)|}\right|^2d\hausdorff^{n-1}(x)\\
        &\qquad\qquad\qquad\qquad \qquad\qquad\leq C_{\ref{thm:MAINcones}}' \big[P(E,\interior{\mathcal{C}}) - P(B_1 \cap \mathcal{C}, \interior{\mathcal{C}})\big],
    \end{split}
    \]
where $\diam_{\R^m} E$ is the diameter of the projection of $E$ onto $\R^m$, and $\barycenter_{\R^m}E:= \tfrac{1}{|E|} \int_E (x_1,\ldots, x_m)$.
\end{theorem}

For \cref{thm:MAINcones} we can draw considerations analogous to those discussed above for \cref{thm:MAINcapillarity}. In particular, scaling invariant versions of the inequalities in \cref{thm:MAINcones} hold true (see \cref{thm:ConesStrongQuantitative}), and a barycentric version of the inequality holds for possibly unbounded sets, up to replacing $\barycenter_{\R^m} E$ with a nonlinear notion of barycenter (see \cref{cor:barycentric-ineq-cones}). Moreover, analogously to the case of \cref{thm:MAINcapillarity}, it can be seen that exponents in the inequalities in \cref{thm:MAINcones} are sharp.\\
We also remark that it is possible to unravel a more precise dependence of constant $C_{\ref{thm:MAINcones}}$ on the cone $\mathcal{C}$. Such a dependence is explicitly pointed out in \cref{thm:ConesStrongQuantitative}, where it is proved that $C_{\ref{thm:MAINcones}}$ depends on the dimension $n$, a lower bound on $|B_1(0) \cap \mathcal{C}|$, the constant in the quantitative isoperimetric inequality for the Fraenkel asymmetry (see \cref{thm:QuantitativeConesFraenkel}), and a threshold parameter under which it is possible to apply an $\eps$-regularity result for perimeter almost minimizers in the cone under consideration (see \cref{thm:almost-regularity-cones}).

In fact, a cone $\mathcal{C}$ in \cref{thm:MAINcones} is assumed to be piecewise $C^2$ precisely in order to apply the $\eps$-regularity result in \cref{thm:almost-regularity-cones}, which is a corollary of the main result from \cite{EdelenLi22}. We believe that whenever one has at disposal an analogous $\eps$-regularity result holding in a more general class of cones, our proof could be repeated verbatim leading to \cref{thm:MAINcones} accordingly stated on the less regular class of cones.

As we shall describe below in \cref{sec:StrategyProof}, both \cref{thm:MAINcapillarity} and \cref{thm:MAINcones} follow by applying a new rather robust strategy that builds upon the classical selection principle first developed in \cite{CicaleseLeonardi, AcerbiFuscoMorini}.

\begin{remark}[More barycentric inequalities]\label{rem:MoreBarycentricInequlities}
Our method can be applied to further isoperimetric-type problems, producing barycentric forms of quantitative inequalities. In fact, in \cref{cor:barycentric-Fraenkel} we derive a barycentric version of the quantitative inequality \eqref{eq:PPquantitativaFraenkelCapillare} for the Fraenkel asymmetry in capillarity problems.
Moreover, with our technique we could obtain the barycentric version of the first strong quantitative isoperimetric inequality \eqref{eq:zzQuantFuscoJulin} from \cite{FuscoJulin}. Similarly, our method would also recover the quantitative isoperimetric inequality with barycentric Fraenkel asymmetry for the standard perimeter $P$, which was obtained in dimension $2$ in \cite[Theorem~3.1]{BianchiniCroceHenrot2023}, and recently generalized to every dimension in \cite[Theorem~A]{GambicchiaPratelli2025}.

We stress that our approach is always able to detect an explicit point $x^*$ representing a ``center'' where the asymmetry can be computed, see \cref{cor:tilde-bar-quantitative} and \eqref{eq:BaricentricaConiConBtilde}. In this case, the constants appearing in the inequalities are \emph{independent} of the diameter of the considered sets. On the other hand, if one whishes to compute the asymmetries with respect to the the actual barycenter, as stated in \cref{thm:MAINcapillarity} and \cref{thm:MAINcones}, our technique provides a constant depending on the diameter $d$ of the set. The dependence on $d$ is, in this case, inevitable; however, we are not able to obtain a precise dependence of the constants $C'_{\ref{thm:MAINcapillarity}}$, $C'_{\ref{thm:MAINcones}}$ on $d$.

Finally, we believe that our approach might be useful to provide a quantitative inequality with barycentric asymmetry for the fractional perimeter as well. There is a very recent result for convex sets in \cite{GambicchiaMerlinoRuffiniTalluri2025}, and our methods could be fruitful to treat the general case.
\end{remark}

\subsection{Strategy of the proof and technical comments}\label{sec:StrategyProof}

We first discuss the content of our work regarding the capillarity perimeter, which contains the main technical novelties.
Next, we comment on the isoperimetric problem in convex cones, that requires similar (and in many cases simpler) arguments.

To prove \cref{thm:MAINcapillarity}, we essentially follow the classical strategy leading to a selection-type argument, as also done in \cite{FuscoJulin,Neumayer16, DeMason24}. The method consists in two very different steps:
\begin{description}
    \item [{(Fuglede inequality)}] prove the quantitative inequality for sets that are regular small perturbations of our model (in this case $B^{\lambda}$). This is the content of \cref{prop:FugledeCapillarity}, building on \cref{lem:OScillationBoundedByH1Norm,prop:EspansionePlambda};
    \item [{(Selection principle)}] reduce to regular sets by finding suitable competitors as minimizers of some well designed auxiliary functionals, exploiting the regularity theory already developed for the perimeter functional (see \cref{thm:almost-regularity}). This is the content of \cref{prop:MinimizationProblemSelection}.
\end{description}
In our setting, both of the steps present nontrivial difficulties with respect to the classical results available in the literature. Concerning the Fuglede inequality, the standard approach consists in a Fourier expansion of the perimeter and asymmetry terms (see \cite[Section~1]{Fuglede} and \cite[Lemma~4.2]{CicaleseLeonardi}). However, for the capillarity problem, this is not feasible for different reasons: even if we could consider only sets that can be parametrized as normal graphs on $B^{\lambda}$ (and this is already too restrictive in order to combine it with the selection principle), we lack of a description of the eigenvalues of the Laplacian on the spherical cap $\partial B^{\lambda}\cap\{x_n\geq0\}$ with Dirichlet data. Instead, our idea is to parametrize both $B^{\lambda}$ and its perturbation over $\S^{n-1}\cap\{x_n\geq 0\}$, and then to do a Taylor expansion of the quantities involved. Still, the parametrization of $B^{\lambda}$ does not seem to have any good properties in terms of spherical harmonics, and thus we cannot perform the expansion of our perturbation with the usual technique. We then follow the path paved by \cite[Proposition~1.9]{Neumayer16}, relying on the quantitative isoperimetric inequality with Fraenkel asymmetry (i.e. \cref{thm:PPquantitativaFraenkelCapillare}). In this part of our work, there are much more computations compared to the aforementioned references because the parametrization of $B^{\lambda}$ is clearly non-constant (for $\lambda\neq0$), yielding many more terms in the expressions that we manipulate.

We now briefly describe the selection principle tailored to a proof by contradiction. In a nutshell, we suppose that there exists a sequence of sets $E_k\subset\{x_n\geq0\}$ with $|E_k|=|B^{\lambda}|$ that contradicts the first quantitative inequality in \cref{thm:MAINcapillarity}. If these sets were small perturbations of $B^{\lambda}$, then the Fuglede inequality would give an immediate contradiction, as it provides exactly the desired inequality. To reduce to this case, we consider a sequence of auxiliary functionals $\mathcal{F}_k$ and we take a sequence of minimizers $F_k\subset \{x_n\geq0\}$ of $\mathcal{F}_k$. If we design $\mathcal{F}_k$ properly, then sequence $\{F_k\}_k$ still contradicts \cref{thm:MAINcapillarity}, and they turn out to be regular enough thanks to the regularity theory available for almost minimizers of the (anisotropic weighted) perimeter. Moreover, the sets $F_k$ will converge to $B^{\lambda}$ in $C^1$-topology. Hence, we get to a contradiction by applying the Fuglede inequality to the sets $F_k$ with $k\gg1$. 

We consider the auxiliary functionals
\begin{equation}\label{eq:DefFkINTRO}
    \mathcal{F}_k(F) := P_{\lambda}(F)+\Lambda||F|-|B^{\lambda}||-\frac{1}{k}\int_{\partial^*F\cap\{x_n>0\}}\left|\nu_F(x)-\frac{x+\lambda e_n}{|x+\lambda e_n|}\right|^2\de\hausdorff^{n-1}_x+|\widetilde b(F)|^2,
\end{equation}
where $\Lambda>0$ is a constant, the integral term coincides with the expression in \cref{thm:MAINcapillarity} with $x'=0$, and $\widetilde b(F)$ is a bounded volume term that coincides with $|F|\barycenter_H(F)$ whenever $F\subset B_{100}$. The functional $\mathcal{F}_k$ is defined for any set $F$ with finite measure, without any confinement restriction, and $\widetilde b(F)$ represents a nonlinear notion of barycenter for $F$ (see \eqref{eq:AlternativeBarycenter}). A few comments are in order:
\begin{itemize}
    \item it is easily checked that $F\mapsto P_{\lambda}(F)+\Lambda||F|-|B^{\lambda}||$ is uniquely minimized by $B^{\lambda}$ for $\Lambda>n$, and since $\widetilde b(B^{\lambda})=0$, then $\mathcal{F}_k$ is a small perturbation of a functional that is minimized by our model $B^{\lambda}$, hinting that the minimizers of $\mathcal{F}_k$ (if they exist) converge to $B^{\lambda}$ in some sense as $k\to\infty$;
    \item the structure of the functional is not completely new, as one can check by looking at all the other papers based on the selection principle. Once we minimize $\mathcal{F}_k$, we find sets $F_k$ with $P_{\lambda}(F_k)\sim P_{\lambda}(B^{\lambda})$, while the minus sign in front of the third term in $\mathcal{F}_k$ favors sets for which the asymmetry is not too small. Notice that the integral term in \eqref{eq:DefFkINTRO} is not equal to the oscillation asymmetry considered in the final main inequality in \cref{thm:MAINcapillarity}, as the reference center here is fixed to be $x'=0$;
    \item adding the nonlinear barycenter term $\widetilde b(F)$ is new, and we introduce it here also because the integral term in \eqref{eq:DefFkINTRO} is not translation invariant. In fact, if $\widetilde b(F)$ were not present, then the functionals $\mathcal{F}_k$ might not possess minimizers for some values of $\lambda$ (this is immediate for $\lambda=0$ using \cref{lem:RiscritturaDiMuLambdaZero}).
\end{itemize} 

Comparing with \cite{FuscoJulin,Neumayer16,DeMason24,BogeleinDuzaarFusco,BogeleinDuzaarScheven}, one immediately notices that their auxiliary functionals contain a term that \emph{exactly} coincides with the oscillation asymmetry in their quantitative inequalities. 
However, doing that same (natural) choice in our situation would lead to serious difficulties in applying the regularity theory for almost minimizers of the perimeter, that is the most restraining point in this approach. These issues are due to the different nature of our oscillation asymmetry compared to those present in \cite{FuscoJulin,Neumayer16,DeMason24,BogeleinDuzaarFusco,BogeleinDuzaarScheven}. Indeed, while the oscillation asymmetries considered in those references can be rewritten as the their perimeter functional plus a volume term, in our case we have (see computations in \cref{lem:RiscritturaDiMuLambdaZero})
\begin{equation*}
\begin{split}
    &\inf_{x'\in\{x_n=0\}}\int_{\partial^* F \cap \{x_n>0\}} \left| \nu_F(x) - \frac{x-(x',-\lambda)}{|x-(x',-\lambda)|} \right|^2 \de \hausdorff^{n-1}_x\\
    &\qquad= \inf_{x'\in\{x_n=0\}}\int_{\partial^*F \cap \{x_n>0\}} 1 - \left\langle \nu_F , \frac{\lambda e_n}{|x- \Braket{x,e_n}e_n -(x',-\lambda)|} \right\rangle\de\hausdorff^{n-1}_x
 - \int_F \frac{n-1}{|x-(x',-\lambda)|}\de x.
\end{split}
\end{equation*}
The first term on the second line is a ``perimeter term'' (so a main term in the regularity theory) that is computed at some point $x'=x'(F)$. So, placing our oscillation asymmetry in place of the integral term in \eqref{eq:DefFkINTRO} would not lead to a higher order perturbation of the capillarity perimeter, but rather to a perturbation of a perimeter functional with a density that \emph{depends itself} on the set $F$ that we are considering.
For this reason we consider the auxiliary functionals $\mathcal{F}_k$ containing a non-translation-invariant integral term. Even with this simplification, in \cref{prop:MinimizationProblemSelection} we need to rely on the regularity theory developed in \cite{DePhilippisMaggi2015} for capillarity almost minimizers of weighted anistropic perimeters (see \cref{AppendixA} for the precise definitions and results).

The existence of minimizers for the auxiliary functionals $\mathcal{F}_k$ (at least for $k\gg1$) is fundamental to pursue this strategy and it is nontrivial in our case. Indeed, typically one is able to prove that it is sufficient to consider sets with uniformly bounded diameter. Hence auxiliary functionals in a selection-type argument are usually defined on sets that are a priori contained in some ball.
%
%
See, e.g., \cite[Lemma~3.2]{FuscoJulin} and \cite[Lemma~3.1]{Neumayer16} for this confinement method. However, this confinement procedure is once again based on the expansion of the oscillation asymmetry as the sum of perimeter plus a volume term, which is not available in our situation. Hence we are forced to minimize $\mathcal{F}_k$ among all sets having finite measure, which means that existence of minimizers for $\mathcal{F}_k$ is not immediate.
%
%
To overcome this difficulty, we use a robust and rather simple argument, that appears to be new in our context. The cornerstone of this approach is Ekeland's variational principle (see \cref{thm:Ekeland}), that ensures the existence of minimizers for some \emph{volume} perturbations of $\mathcal{F}_k$. Indeed, the space $\left\{E\subset \{x_n\geq0\}:|E|<+\infty\right\}$ endowed with the $L^1$ distance is a complete metric space, and we prove that $\mathcal{F}_k$ is lower semicontinuous for $k\gg1$. Hence, for any $j>1$, Ekeland's \cref{thm:Ekeland} guarantees the existence of a set $E_j^k\subset\{x_n\geq0\}$ minimizing
\[
    F\mapsto \mathcal{F}_k(F)+\frac{1}{j}|F\symmdiff E_j^k|.
\]
The use of the $L^1$ distance is particularly effective since it amounts to adding a (very small) volume term, and this is harmless when we apply regularity theory. In fact, it is not difficult to prove that $E_j^k$ are quasi-minimizers of the perimeter with uniform constants (that is weaker than being almost minimizers, see \cref{def:quasiminimizer}). Then, \cref{thm:quasi-regularity} provides density estimates, that automatically lead to a uniform confinement of $E_j^k$ for $k\gg1$ thanks to \cref{thm:PPquantitativaFraenkelCapillare} and to a control on the barycenter. This is basically the starting point of \cref{prop:MinimizationProblemSelection}. We mention that an application of Ekeland's principle was also employed in \cite{BoccardoFeroneFuscoOrsina99} in a very different variational problem to improve the regularity of minimizing sequences.

We finally point out that adding the ``barycentric penalization'' in the auxiliary functionals (see \eqref{eq:DefFkINTRO}) in place of performing a confinement reduction eventually has the advantage of automatically producing also the barycentric version of the target quantitative inequality (see the discussion in \cref{sec:BarycentricCapillarity}).

\medskip

To deal with the problem in convex cones, the strategy is basically the same as the one exploited in the capillarity setting. In fact, given a convex cone $\mathcal{C}=\R^m\times \widetilde{\mathcal{C}}\subset\R^n$, where $\widetilde{\mathcal{C}}$ does not contain lines, the Fuglede inequality was already known in a particular class of cones (see \cite[Lemma~3.1]{BaerFigalli17}); after extending such an inequality to the class of convex cones as we did in the capillarity case, we proceed with the selection principle introducing the sequence of auxiliary functionals
\[
    \mathcal{G}_k(F) = P(F,\interior{\mathcal{C}}) + \Lambda\left||F|-|B_1\cap \mathcal{C}|\right|-\frac{1}{k}\int_{\partial^*F\cap\interior{\mathcal{C}}}\left|\nu_F(x)-\frac{x}{|x|}\right|^2\de\hausdorff^{n-1}_x+|\widetilde{b}_{\R^m}(F)|^2,
\]
where $\widetilde{b}_{\R^m}(F)$ is again a nonlinear barycenter, i.e., a volume term coinciding with $|F|\barycenter_{\R^m}(F)$ whenever $F\subset B_{100}$. Now the situation is even simpler than before, since this setting is analogous to the capillarity problem with $\lambda=0$ from many points of view.
The proof of the final \cref{thm:ConesStrongQuantitative} is more explicit than that of the capillarity case (\cref{thm:capillarityStrongQuantitative}), in the sense that no compactness is used at that point and we can keep track rather explicitly of the dependence of the final constant in the quantitative inequality.

\medskip
\noindent\textbf{Organization.}
In \cref{sec:SetsFinitePerimeter} we collect the basic definitions and facts on sets of finite perimeter, while in \cref{sec:capillarity} we provide the first tools specific for the capillarity problem. In \cref{subsec:fuglede-capillarity} we obtain Fuglede-type results for the capillarity functional.
In \cref{sec:proof-capillarity} and \cref{sec:BarycentricCapillarity} we prove the strong and barycentric quantitative inequalities in capillarity problems. In \cref{subsec:sharp} we show that the exponents in our results are sharp.
In \cref{sec:cones} we prove the strong and barycentric quantitative inequalities in cones.
In \cref{AppendixA} we recall some results in regularity theory for sets satisfying certain weak notions of perimeter minimality and a version of Ekeland's variational principle. 
In \cref{AppendixB} we recall some technical identities and estimates.

\section{Strong quantitative isoperimetric inequalities for capillarity problems}

\begin{center}
\emph{Throughout the whole paper it is assumed that $n\in \N$ with $n\ge 2$, and $\lambda \in (-1,1)$ are fixed.}    
\end{center}

\noindent\textbf{List of symbols.}

\begin{itemize}
    \item $B_r(x)$ = open ball of radius $r$ and center $x$ in $\R^n$. We denote $B_r:= B_r(0)$.

    \item $|E|$ = Lebesgue measure of a set $E\subset \R^n$.

    \item $\Braket{v,w}$ denotes Euclidean scalar product, for $v,w \in \R^n$.

    \item $d_{\hausdorff}$ denotes Hausdorff distance in $\R^n$, $\hausdorff^s$ denotes $s$-dimensional Hausdorff measure in $\R^n$.

    \item $\partial^*E$ = reduced boundary of a set $E\subset \R^n$, $P(E,\cdot):= \hausdorff^{n-1}\mres \partial^*E$ denotes the perimeter measure, $P(E):=P(E,\R^n)$.
    

    \item $B^\lambda = \{x \in B : \Braket{x, e_n} > \lambda\} - \lambda e_n$.
    
    \item $B^\lambda(v) := \frac{v^{\frac{1}{n}}}{|B^\lambda|^{\frac1n}} B^\lambda$, for any $v>0$.
    
    \item $B^\lambda(v,x) := B^\lambda(v)  +x$, for any $x \in \{x_n=0\}$. In particular $B^\lambda(v)=B^\lambda(v,0)$.

    \item $H^+:=\{ x_n >0\}$, $H:=\{x_n\ge0\} = \overline{H^+}$.

    \item $\S^{n-1} := \partial B_1$, $\S^{n-1}_+:=\partial B_1 \cap \{x_n\ge 0\}$.

    \item $P_\lambda(E) := P(E,H^+) -\lambda \hausdorff^{n-1}(\partial^* E \cap \partial H)$.

    \item $\alpha_\lambda, D_\lambda$ denote capillarity Fraenkel asymmetry and deficit, respectively, defined in \cref{def:Asymmetry}.

    \item $w_\lambda$ denotes the function that parametrizes the boundary of $B^\lambda$ as a graph over $\S^{n-1}_+$, defined in \cref{def:DistanzaC1}.

    \item $\mu_{\lambda}, \mu_{\lambda,0}$ denote capillarity oscillation asymmetries, defined in \cref{def:capillarity-oscillation}.
    

    \item $C_{n,\lambda}$ denotes a positive constant depending on $n,\lambda$ only, that may change from line to line.

    \item $O(A)$ denotes a real quantity such that $O(A) \le C_{n,\lambda} A$ for some $C_{n,\lambda}>0$.

    \item ${\rm bar}_H F := \tfrac{1}{|F|} \int_F (x-\Braket{x, e_n}e_n) \de x$, for any set $E\subset H$ with $\int_E |x| <+\infty$.

    \item $\psi(t) := \max\{ -100 , \min\{ t, 100\}\}$, defined in \eqref{eq:DefPSI}.

    \item $\widetilde{b}(F) := \int_F \left( \psi(x_1), \ldots, \psi(x_{n-1}) \right) \de x$ for any $F \subset H$ with $|F|<+\infty$, defined in \eqref{eq:AlternativeBarycenter}.

    \item $k_\lambda \in \N$ is the least integer such that $k_\lambda\ge \max\{10, 16/(1-|\lambda|)^2)\}$, defined in \eqref{eq:DefKlambda}.
\end{itemize}

\subsection{Sets of finite perimeter}\label{sec:SetsFinitePerimeter}

We refer to \cite{AmbrosioFuscoPallara, MaggiBook} for a complete account on the theory of sets of finite perimeter.
The perimeter of a measurable set $E\subset\R^n$ in an open set $A\subset \R^n$ is defined by
\begin{equation*}
P(E,A):=\sup\left\{\int_E {\rm div}\, T \st T \in C_c^1(A;\R^n), \,\, \|T\|_\infty \le1\right\}.
\end{equation*}
Denoting $P(E):=P(E,\R^n)$, we say that $E$ is a set of finite perimeter if $P(E)<+\infty$.
If $E$ is a set of finite perimeter, the characteristic function $\chi_E$ has a distributional gradient $D\chi_E$ that is a vector-valued Radon measure on $\R^n$ such that
\begin{equation*}
\int_E {\rm div}\, T  = - \int_{\R^n} T  \de D\chi_E ,
\end{equation*} 
for any field $T \in C_c^1(\R^n;\R^n)$.
The set function $P(E,\cdot)$ defined above can be extended as a positive Borel measure on $\R^n$. Moreover, the measure $P(E,\cdot)$ coincides with the total variation $|D\chi_E|$ of the distributional gradient, and it is concentrated on the reduced boundary
\begin{equation*}
\partial^*E : = \left\{ x \in {\rm spt} |D\chi_E| \st \exists\, \nu^E(x):=-\lim_{r\to0} \frac{D\chi_E(B_r(x))}{|D\chi_E(B_r(x))|} \text{ and } |\nu^E(x)|=1 \right\}.
\end{equation*}
The vector $\nu^E$ is called the generalized outer normal of $E$.
Moreover $P(E,\cdot) = \hausdorff^{n-1}\mres \partial^*E$, and the distributional gradient can be written as $D\chi_E = - \nu^E \hausdorff^{n-1}\mres \partial^*E$.

\subsection{Capillarity problem in the half-space and related asymmetry indexes}\label{sec:capillarity}

We begin with the basic definitions and properties about the quantities involved with the capillarity problem.

\begin{remark}[Alternative expression for $P_{\lambda}$]\label{rem:PlambdaConDivergenza}
Let $E \subset H$ be a measurable set. We observe that
\[
P_\lambda(E) = \int_{\partial^*E\cap H^+}1 - \lambda \Braket{e_n, \nu_E} \de \hausdorff^{n-1},
\]
where $\nu_E$ is the generalized outer normal to $E$. In particular, since $|\lambda|<1$, we have that $P_\lambda(E)\ge0$. The previous identity follows from the divergence theorem, indeed
\begin{equation*} 0 = \int_E \divergence\, e_n \de x = - \hausdorff^{n - 1}(\partial^*E \cap \partial H) + \int_{\partial^*E \cap H^+} \Braket{e_n, \nu_E}  \de \hausdorff^{n-1}.
\end{equation*}
\end{remark}

\begin{definition}[Capillarity Fraenkel asymmetry and isoperimetric deficit]\label{def:Asymmetry}
Let $E \subset H$ be a Borel set with measure $|E|=v\in(0,+\infty)$. Denoting $B^\lambda(v) := \frac{v^{\frac{1}{n}}}{|B^\lambda|^{\frac1n}} B^\lambda$, and $B^\lambda(v,x) := B^\lambda(v)  +x$, for any $x \in \{x_n=0\}$, we define the capillarity Fraenkel asymmetry by
\begin{equation*}
\alpha_\lambda(E) := \inf\left\{\frac{|E \Delta B^\lambda(v,x)|}{v} \st  x \in \partial H\right\}. \end{equation*}
It is readily checked that the Fraenkel asymmetry of $E$ is a minimum. We further define the capillarity isoperimetric deficit by
\begin{equation*}
D_\lambda(E) := \frac{P_\lambda(E) - P_\lambda(B^\lambda(v))}{P_\lambda(B^\lambda(v))}. \end{equation*}
\end{definition}

We recall here the quantitative isoperimetric inequality for the Fraenkel asymmetry in capillarity problems.

\begin{theorem}[{\cite{PascalePozzettaQuantitative, KreutzSchmidt24}}]\label{thm:PPquantitativaFraenkelCapillare}
There exists $C_{\ref{thm:PPquantitativaFraenkelCapillare}} = C_{\ref{thm:PPquantitativaFraenkelCapillare}}(n,\lambda)>0$ such that for any set of finite perimeter $E \subset H$ with $|E|<+\infty$ there holds
\begin{equation}\label{eq:PPquantitativaFraenkelCapillare}
\alpha_\lambda^2(E) \le C_{\ref{thm:PPquantitativaFraenkelCapillare}} D_\lambda
(E).
\end{equation}
\end{theorem}

\begin{definition}\label{def:DistanzaC1}
Denote $\S^{n-1}_+:=\partial B_1 \cap \{x_n\ge 0\}$.
We denote by $w_\lambda : \S^{n-1}_+\to \R$ the function such that
\[
 \overline{\partial B^\lambda \cap H^+}= \{ w_\lambda(x) \, x \st x \in \S^{n-1}_+\}.
\]
Let $k \in \N$ with $k\ge 1$, $\beta \in (0,1)$.
Let $E \subset H$ be a bounded open set. We say that the boundary of $E$ is a $C^{k,\beta}$-graph over $\S^{n-1}_+$ (or, more precisely, that $\overline{\partial E\cap H^+}$ is a $C^{k,\beta}$-graph over $\S^{n-1}_+$) if $\partial E$ is Lipschitz and if there exists a $C^{k,\beta}$-function $f:\S^{n-1}_+\to (0,+\infty)$ such that
\[
 \overline{\partial E\cap H^+}= \{ f(x) \, x \st x \in \S^{n-1}_+\}.
\]
If the boundaries of two sets $E,F$ are $C^{k,\beta}$-graphs over $\S^{n-1}_+$ with respect to some functions $f,g$, we define
\[
\d_{C^{k,\beta}}(E,F):=\|f-g\|_{C^{k,\beta}(\S^{n-1}_+)}.
\]
If $\{E_i\}_{i \in \N}, E$ have boundary given by $C^{k,\beta}$-graphs over $\S^{n-1}_+$, we say that $\overline{\partial E_i \cap H^+}$ converges $\overline{\partial E \cap H^+}$ in $C^{k,\beta}$ (or, simply, that $E_i$ converges to a $E$ in $C^{k,\beta}$) if $\d_{C^{k,\beta}}(E_i,E)\to 0$ as $i\to\infty$.
\end{definition}

\begin{definition}[Capillarity oscillation asymmetry]\label{def:capillarity-oscillation}
    Let $E \subset H$ be a Borel set with measure $|E|=|B^\lambda|$. We define the capillarity oscillation asymmetry by
\[
    \mu_\lambda^2(E) \coloneqq \inf\left\{ \frac{1}{2} \int_{\partial^* E \cap H^+} \left| \nu_E(x) - \frac{x-(x',-\lambda)}{|x-(x',-\lambda)|} \right|^2 \de \hausdorff^{n-1} (x) \st x' \in\partial H
\right\}.
\]
For a generic Borel set $E \subset H$ with finite perimeter and finite measure, setting $r_E:=(|E|/|B^\lambda|)^{\frac1n}$, we define
\[
    \mu_\lambda^2(E) \coloneqq \mu_\lambda^2\left( r_E^{-1} E \right) = \inf\left\{
    \frac{1}{2 r_E^{n-1}} \int_{\partial^* E \cap H^+} \left| \nu_E(x) - \frac{x-r_E(x',-\lambda)}{|x-r_E(x',-\lambda)|} \right|^2 \de \hausdorff^{n-1} (x) \st x' \in\partial H
\right\}.
\]
For technical reasons, it will be convenient to also consider an oscillation asymmetry that is not translation or scaling invariant. For any Borel set $E\subset H$ with finite measure and finite perimeter, we define
\[
    \mu_{\lambda,0}^2(E) = \frac{1}{2}\int_{\partial^* E\cap H^+}\left|\nu_E(x)-\frac{x+\lambda e_n}{|x+\lambda e_n|}\right|^2\de \hausdorff^{n-1}(x).
\]
\end{definition}

\begin{lemma}\label{lem:RiscritturaDiMuLambdaZero}
Let $E\subset H$ be a set of finite perimeter with finite volume. Then
\[
 \mu_{\lambda,0}^2(E) = \int_{\partial^*E \cap H^+} 1 - \left\langle \nu_E , \frac{\lambda e_n}{|x- \Braket{x,e_n}e_n +\lambda e_n|} \right\rangle\de\hausdorff^{n-1}_x
 - \int_E \frac{n-1}{|x+\lambda e_n|}\de x.
\]
\end{lemma}

\begin{proof}
The claim directly follows form the divergence theorem since
\[
\begin{split}
\mu_{\lambda,0}^2(E) 
& = \int_{\partial^*E \cap H^+} 1 - \left\langle\nu_E , \frac{x+ \lambda e_n}{|x+\lambda e_n |}\right\rangle 
\\&=
P(E,H^+) - 
\int_E {\rm div} \left( \frac{x+ \lambda e_n}{|x+\lambda e_n |}\right) 
+ \int_{\partial^*E \cap \partial H} \left\langle - e_n, \frac{x+ \lambda e_n}{|x+\lambda e_n |}\right\rangle \\
&=
P(E,H^+) - 
\int_E \frac{n-1}{|x+\lambda e_n|}
+ \int_{\partial^*E \cap \partial H} \left\langle - e_n, \frac{x+ \lambda e_n}{|x+\lambda e_n |}\right\rangle,
\end{split}
\]
and
\[
\begin{split}
    \int_{\partial^*E \cap \partial H} \left\langle - e_n, \frac{x+ \lambda e_n}{|x+\lambda e_n |}\right\rangle
    &= \int_{\partial^*E \cap \partial H} \left\langle
    -e_n, \frac{\lambda e_n}{|x -\Braket{x,e_n}e_n +\lambda e_n|}
    \right\rangle 
    \\
    &= 
    \int_E {\rm div}\left(  \frac{\lambda e_n}{|x - \Braket{x,e_n}e_n +\lambda e_n|}\right)
    - \int_{\partial^* E \cap H^+} \left\langle
    \nu_E,   \frac{\lambda e_n}{|x - \Braket{x,e_n}e_n +\lambda e_n|}
    \right\rangle,
\end{split}
\]
and ${\rm div}\left(  \frac{\lambda e_n}{|x - \Braket{x,e_n}e_n +\lambda e_n|}\right)=0$.
\end{proof}

\subsection{Fuglede-type computations for the capillarity problem}\label{subsec:fuglede-capillarity}
The following three technical results, namely \cref{lem:OScillationBoundedByH1Norm}, \cref{prop:EspansionePlambda} and \cref{prop:FugledeCapillarity}, proceed along the lines of \cite[Section 4]{Neumayer16}. However, an additional complication here is that we need to parametrize the boundaries of both $B^{\lambda}$ and a of generic competitor as graphs over $\S^{n-1}_+$. Hence, both of the parametrizing graphs are non-constant, leading to a more involved computation. Moreover, in \cref{prop:FugledeCapillarity} we will derive a Fuglede-type inequality for the capillarity perimeter keeping track of a possible barycentric displacement.

\begin{lemma}\label{lem:OScillationBoundedByH1Norm}
There exist $\eps_{\ref{lem:OScillationBoundedByH1Norm}}, C_{\ref{lem:OScillationBoundedByH1Norm}}>0$ depending on $n,\lambda$ such that the following bounds hold.\\
Let $E \subset H$ be an open set such that $\overline{\partial E \cap H^+} = \{ (w_\lambda(x) + u(x))x \st x \in \S^{n-1}_+\}$ for some $C^1$ function $ u : \S^{n-1}_+ \to \R$ with $\|u\|_{C^1(\S^{n-1}_+)}\le \eps_{\ref{lem:OScillationBoundedByH1Norm}}$. Then
\[
\mu_{\lambda,0}^2(E)\leq C_{\ref{lem:OScillationBoundedByH1Norm}} \| u \|^2_{H^1(\S^{n-1}_+)} .
\]
In particular, if also $|E|=|B^\lambda|$, then 
\[
\mu_\lambda^2(E) \le \mu_{\lambda,0}^2(E)\leq C_{\ref{lem:OScillationBoundedByH1Norm}} \| u \|^2_{H^1(\S^{n-1}_+)} .
\]
\end{lemma}

\begin{proof}
We take $\eps_{\ref{lem:OScillationBoundedByH1Norm}}\in(0,1)$ sufficiently small so that $|w_\lambda + u| \ge c=c(n,\lambda)>0$ on $\S^{n-1}_+$.
We will exploit identities from \cref{lemma:formulas}.
We can rewrite
\begin{equation}\label{eq:zzz0}
\begin{split}
     &\mu_{\lambda,0}^2(E)
     = \frac12 \int_{\partial E \cap H^+} 
     \left| \nu_E(y) - \frac{y+\lambda e_n}{|y+\lambda e_n|} \right|^2 \de \hausdorff^{n-1}(y) \\
     &= \frac12 \int_{\S^{n-1}_+} 
     \left| \nu_E((w_\lambda+u)x) - \frac{(w_\lambda+u)x+\lambda e_n}{|(w_\lambda+u)x+\lambda e_n|} \right|^2
     (w_\lambda+u)^{n-2}[ (w_\lambda+u)^2 + |\nabla w_\lambda+ \nabla u|^2]^{\frac12} \de \hausdorff^{n-1}(x) \\
     &\le C \int_{\S^{n-1}_+} 
     \left| \nu_E((w_\lambda+u)x) - \frac{(w_\lambda+u)x+\lambda e_n}{|(w_\lambda+u)x+\lambda e_n|} \right|^2 \de \hausdorff^{n-1}(x).
\end{split}
\end{equation}
Moreover
\begin{equation}\label{eq:Normali}
    \begin{split}
        \nu_E((w_\lambda+u)x) = \frac{(w_\lambda+u)x - \nabla w_\lambda - \nabla u}{[(w_\lambda+u)^2 + |\nabla w_\lambda + \nabla u|^2]^{\frac12}},\\
        \nu_{B^\lambda}(w_\lambda x) = w_\lambda x +\lambda e_n = \frac{w_\lambda x - \nabla w_\lambda}{[w_\lambda^2+ |\nabla w_\lambda|^2]^{\frac12}},
    \end{split}
\end{equation}
for any $x \in \S^{n-1} \cap H^+$. Hence, for $x \in \S^{n-1} \cap H^+$, we have
\begin{equation}\label{eq:zzz1}
    \begin{split}
        \nu_E((w_\lambda+u)x) &- \frac{(w_\lambda+u)x+\lambda e_n}{|(w_\lambda+u)x+\lambda e_n|} 
        \\&\qquad\qquad=
        \underbrace{
        \frac{(w_\lambda+u)x - \nabla w_\lambda - \nabla u}{[(w_\lambda+u)^2 + |\nabla w_\lambda + \nabla u|^2]^{\frac12}} - \nu_{B^\lambda}(w_\lambda x)}_{A}
        + 
        \underbrace{
        w_\lambda x +\lambda e_n - \frac{(w_\lambda+u)x+\lambda e_n}{| u x + \nu_{B^\lambda}(w_\lambda x)|}}_{B}.
    \end{split}
\end{equation}
We study the summands in \eqref{eq:zzz1} separately. First, since $|w_\lambda + u| \ge c>0$, we have
\begin{equation}\label{eq:zzz2}
    \begin{split}
        |A&| 
        = 
        \left|
        \frac{w_\lambda x - \nabla w_\lambda }{[(w_\lambda+u)^2 + |\nabla w_\lambda + \nabla u|^2]^{\frac12}} - \nu_{B^\lambda}(w_\lambda x)
        +
        \frac{u x - \nabla u}{[(w_\lambda+u)^2 + |\nabla w_\lambda + \nabla u|^2]^{\frac12}} \right| \\
        &\overset{\eqref{eq:Normali}}{\le}
        |w_\lambda x - \nabla w_\lambda| \left( \frac{1}{[(w_\lambda+u)^2 + |\nabla w_\lambda + \nabla u|^2]^{\frac12}}
        -
        \frac{1}{[w_\lambda^2 + |\nabla w_\lambda|^2]^{\frac12}} 
        \right) + O(|u| + |\nabla u|) 
        \\
        &= \frac{|w_\lambda x - \nabla w_\lambda|}{[(w_\lambda+u)^2 + |\nabla w_\lambda + \nabla u|^2]^{\frac12}[w_\lambda^2 + |\nabla w_\lambda|^2]^{\frac12}} 
        \left( 
        [w_\lambda^2 + |\nabla w_\lambda|^2]^{\frac12}
        -
        [(w_\lambda+u)^2 + |\nabla w_\lambda + \nabla u|^2]^{\frac12}
        \right) \\
        &\qquad\qquad\qquad\qquad\qquad\qquad\qquad\qquad\qquad\qquad\qquad\qquad\qquad\qquad\qquad\qquad\qquad\qquad
        + O(|u| + |\nabla u|) 
        \\
        &=\frac{|w_\lambda x - \nabla w_\lambda|}{[(w_\lambda+u)^2 + |\nabla w_\lambda + \nabla u|^2]^{\frac12}[w_\lambda^2 + |\nabla w_\lambda|^2]^{\frac12}} 
        \left( 
        -\frac{2uw_{\lambda}+u^2+|\grad u|^2+2\grad u\cdot\grad w_{\lambda}}{2[w_{\lambda}^2+|\grad w_{\lambda}|^2]^{1/2}}+O(|u|+|\grad u|)
        \right)\\
    & = O(|u|+|\grad u|)
    \end{split}
\end{equation}
where $O(\cdot)$ may change from line to line.

Moreover
\begin{equation}\label{eq:zzz3}
    \begin{split}
        |B|
        &= \left| w_\lambda x +\lambda e_n - \frac{(w_\lambda+u)x+\lambda e_n}{[1+ u^2 + 2u \Braket{\nu_{B^\lambda}, x}]^{\frac12}} \right|
        = \left|  w_\lambda x +\lambda e_n - \frac{(w_\lambda+u)x+\lambda e_n}{1 + u \Braket{\nu_{B^\lambda}, x} + \tfrac12 u^2 + O(|u|)} \right|
        \\
        &= \left| w_\lambda x +\lambda e_n - (w_\lambda x+u x+\lambda e_n) \left( 1 - u \Braket{\nu_{B^\lambda}, x} - \frac12 u^2 + O(|u|) \right) \right| =O(|u|).
    \end{split}
\end{equation}
Inserting \eqref{eq:zzz1}, \eqref{eq:zzz2}, \eqref{eq:zzz3} into \eqref{eq:zzz0} we find
\begin{equation*}
    \begin{split}
    \mu_{\lambda,0}^2(E)
    &\le C \int_{\S^{n-1}_+} 
     \left|A + B\right|^2 \de \hausdorff^{n-1}(x) 
     \le C \int_{\S^{n-1}_+} |A|^2 + |B|^2 \le C \| u \|^2_{H^1(\S^{n-1}_+)}.
    \end{split}
\end{equation*}
\end{proof}


\begin{proposition}\label{prop:EspansionePlambda}
There exist $c_{\ref{prop:EspansionePlambda}}, C_{\ref{prop:EspansionePlambda}}, \epsilon_{\ref{prop:EspansionePlambda}}>0$ depending on $n,\lambda$ such that the following holds.
Let $E \subset H$ be an open bounded set with measure $|E|=|B^\lambda|$. Suppose that the boundary of $E$ is a $C^1$-graph over $\S^{n-1}_+$, see \cref{def:DistanzaC1}, with respect to a function $f= w_\lambda + u$, with $\|u \|_{C^1(\S^{n-1}_+)} \le\epsilon\le \epsilon_{\ref{prop:EspansionePlambda}}$. Then
\begin{equation*}
    P_\lambda(E) = P_\lambda(B^\lambda) + B(u) + \epsilon \, O \left( \|u \|^2_{H^1(\S^{n-1}_+)} \right),
\end{equation*}
where $B(u)$ satisfies
\[
B(u) \ge c_{\ref{prop:EspansionePlambda}} \int_{\S^{n-1}_+} |\grad u|^2 - C_{\ref{prop:EspansionePlambda}} \int_{\S^{n-1}_+} u^2.
\]
\end{proposition}

\begin{proof}
Recalling \cref{rem:PlambdaConDivergenza}, letting $\phi(v) := |v|-\lambda\Braket{v,e_n}$ for any $v \in \R^n$, there holds
    \[
        P_{\lambda}(E) = \int_{\partial^* E\cap H^+}\phi(\nu_E) \de \hausdorff^{n-1}.
    \]
We compute
    \[
        \grad\phi(v) = \frac{v}{|v|}-\lambda e_n,\qquad \grad^2 \phi(v) = \frac{1}{|v|}\left(\Id - \frac{v}{|v|}\otimes\frac{v}{|v|}\right).
    \]
    By \cref{lemma:formulas}, the outer unit normals to $E$ and $B^\lambda$ are given by
    \begin{equation}\label{eq:normal-to-graph}
        \nu_E(f(x)\,x)=\frac{f(x)\,x-\grad f(x)}{\sqrt{f(x)^2+|\grad f(x)|^2}} ,
        \qquad
        \nu_{B^\lambda}( w_\lambda(x) \, x ) =
        \frac{w_\lambda(x)\,x-\grad w_\lambda(x)}{\sqrt{w_\lambda(x)^2+|\grad w_\lambda(x)|^2}} ,
    \end{equation}
for any $x\in \S^{n-1}_+ \cap H^+$.

To simplify the computations, in this proof we will adopt Einstein's sum convention (i.e., summation is understood over the repeated indexes). We will denote by $\{\tau_1,\ldots,\tau_{n-1}\}$ an orthonormal basis of $T_x\S^{n-1}_+$, for generic $x \in \S^{n-1}_+$, and we will denote by $\partial_i:= \partial_{\tau_i}$ the tangential derivative along $\S^{n-1}_+$ in direction $\tau_i$.

We consider the basis $\{g_i^B\}_{i=1}^{n-1}$ for the tangent space $T_{w_{\lambda}(x)x}B^{\lambda}$ given by
    \[
        g_i^B = \partial_i(w_{\lambda}\, x)=w_{\lambda}\tau_i+\partial_iw_{\lambda} \, x,
    \]
for $x \in \S^{n-1}_+$. 


Define the vectors
\[
\omega^B(x)  = w_{\lambda}^{n-1}x-w_{\lambda}^{n-2}\partial_i w_{\lambda}\tau_i,
\qquad
\omega(x)  = (w_{\lambda}+u)^{n-1}x-(w_{\lambda}+u)^{n-2}\partial_i (w_{\lambda}+u)\tau_i,
\]
for $x \in \S^{n-1}_+$. Notice that $\Braket{\omega^B, g_j^B}= w_\lambda^{n-2} \Braket{w_\lambda x - \partial_i w_\lambda \tau_i, w_\lambda \tau_j + \partial_j w_\lambda x} = w_\lambda^{n-2} ( w_\lambda \partial_j w_\lambda - w_\lambda \partial_j w_\lambda) =0$ for any $x \in \S^{n-1}_+$, hence we can write $\nu_{B^\lambda}( w_\lambda \, x) = \omega^B/|\omega^B|$.

By 1-homogeneity of $\varphi$ and by \cref{lemma:formulas} we have
\[
\begin{split}
 P_{\lambda}(B^{\lambda}) &= \int_{\partial B^{\lambda}\cap H^+} \phi(\nu_{B^{\lambda}}) = \int_{\S_+^{n-1}}\phi\left( \nu_{B^\lambda}(w_\lambda(x) \, x)\right) w_\lambda^{n-2} \sqrt{w_\lambda^2 + |\nabla w_\lambda|^2}  
 \de\hausdorff^{n-1}(x)
 \\ &
 = \int_{\S_+^{n-1}}\phi( w_\lambda \, x - \partial_i w_\lambda \, \tau_i ) w_\lambda^{n-2} \de \hausdorff^{n-1}(x) =
 \int_{\S_+^{n-1}}\phi(\omega^B),
\end{split}
\]
where we used that $\sqrt{w_\lambda^2 + |\nabla w_\lambda|^2}  \nu_{B^\lambda}(w_\lambda(x) \, x) = w_\lambda \, x - \nabla w_\lambda  = w_\lambda \, x - \partial_i w_\lambda \, \tau_i $.

Analogously one gets
    \begin{equation}\label{eq:expansion-perimeter}
    \begin{split}
        P_{\lambda}(E) &= \int_{\S_+^{n-1}}\phi(\omega)\de\hausdorff^{n-1}\\
        &= \int_{\S_+^{n-1}}\phi(\omega^B) + \Braket{\grad\phi(\omega^B),\omega-\omega^B} + \frac12\Braket{\grad^2\phi(\omega^B)(\omega-\omega^B),(\omega-\omega^B)}+O(|\omega-\omega^B|^3)\de\hausdorff^{n-1}.
    \end{split}
    \end{equation}
A direct computation gives
    \begin{equation}\label{eq:ExpansionOmegaOmegaB}
    \begin{split}
        \omega &= \omega^B + w_{\lambda}^{n-2}\left[(n-1)ux-\left((n-2)w_{\lambda}^{-1}\partial_iw_{\lambda}u+\partial_i u\right)\tau_i\right]\\
        &\quad+\frac{(n-1)(n-2)}{2}w_{\lambda}^{n-3}u^2x-\left(\frac{(n-2)(n-3)}{2}w_{\lambda}^{n-4}\partial_iw_{\lambda}u^2+(n-2)w_{\lambda}^{n-3}u\partial_iu\right)\tau_i\\
        &\quad +\epsilon O\left(u^2+|\grad u|^2\right).
    \end{split}
    \end{equation}
    It is straightforward to see that $\nu_{B^{\lambda}}(w_{\lambda} \, x)=w_{\lambda}\, x+\lambda e_n$ for every $x\in \S^{n-1}_+$. 
    Moreover we observed above that $\nu_{B^\lambda}( w_\lambda \, x) = \omega^B/|\omega^B|$. 
    Hence
    \[
    \begin{split}
        \grad \phi(\omega^B)&=\frac{\omega^B}{|\omega^B|}-\lambda e_n=\nu_{B^{\lambda}}(w_{\lambda}\, x)-\lambda e_n=w_{\lambda}\, x,\\
        \grad^2\phi(\omega^B) &=\frac{1}{|\omega^B|}\left(\Id - \nu_{B^{\lambda}}(w_{\lambda}\,x)\otimes \nu_{B^{\lambda}}(w_{\lambda}\, x)\right),
    \end{split}
    \]
    for $x \in \S^{n-1}_+$.
    So the gradient term in the expansion for $P_{\lambda}(E)$ in~\eqref{eq:expansion-perimeter} becomes
    \begin{equation}\label{eq:zzExpansionGradient}
        \Braket{\grad\phi(\omega^B),\omega-\omega^B}=\Braket{w_{\lambda}\,x,\omega-\omega^B} \overset{\eqref{eq:ExpansionOmegaOmegaB}}{=} (n-1)w_{\lambda}^{n-1}u+\frac{(n-1)(n-2)}{2}w_{\lambda}^{n-2}u^2+\epsilon O\left(u^2+|\grad u|^2\right).
    \end{equation}
    To expand the Hessian term in \eqref{eq:expansion-perimeter} it is sufficient to consider the first order expansion of $\omega-\omega^B$, and we single out a multiple of $\nu_{B^{\lambda}}$ since it belongs to the kernel of $\grad^2 \phi(\omega^{B})$. We have that
    \[
    \begin{split}
        \omega-\omega^B&=w_{\lambda}^{n-2}\left[(n-1)ux-((n-2)w_{\lambda}^{-1}\partial_iw_{\lambda} u+\partial_i u)\tau_i\right] + O(u^2+|\grad u|^2)\\
        &=w_{\lambda}^{n-2}\left[(n-1)ux-\grad u-(n-2)u\left(\frac{\grad w_{\lambda}}{w_{\lambda}}-x\right)-(n-2)ux\right] + O(u^2+|\grad u|^2)\\
        & \overset{\eqref{eq:normal-to-graph}}{=}
        w_{\lambda}^{n-2}\left[ux-\grad u+(n-2)u\left(1+\frac{|\grad w_{\lambda}|^2}{w_{\lambda}^2}\right)^{1/2}\nu_{B^{\lambda}}\right] + O(u^2+|\grad u|^2),
    \end{split}
    \]
Exploiting \eqref{eq:normal-to-graph} once more, we compute the value of $\grad^2 \phi(\nu_{B^{\lambda}}) = \Id-\nu_{B^{\lambda}}\otimes\nu_{B^{\lambda}}$ on the orthonormal basis $\{x,\tau_1,\ldots,\tau_{n-1}\}$ of $\R^n$:
    \[
    \begin{split}
        \Braket{(\Id-\nu_{B^{\lambda}}\otimes \nu_{B^{\lambda}})x,x} &= 1-\frac{w_{\lambda}^2}{w_{\lambda}^2+|\grad w_{\lambda}|^2} = \frac{|\grad w_{\lambda}|^2}{w_{\lambda}^2+|\grad w_{\lambda}|^2},\\
        \Braket{(\Id-\nu_{B^{\lambda}}\otimes \nu_{B^{\lambda}})\tau_i,\tau_j} &= \delta_{ij}-\frac{\partial_iw_{\lambda}\partial_jw_{\lambda}}{w_{\lambda}^2+|\grad w_{\lambda}|^2},\\
        \Braket{(\Id-\nu_{B^{\lambda}}\otimes \nu_{B^{\lambda}})x,\tau_i} &= \frac{w_{\lambda}\partial_iw_{\lambda}}{w_{\lambda}^2+|\grad w_{\lambda}|^2}.
    \end{split}
    \]
    Since $\nu_{B^\lambda}\in \ker \grad^2 \phi(\omega^B)$, we deduce that
    \begin{equation}\label{eq:zzExpansionHessian}
    \begin{split}
        &\Braket{\grad^2\phi(\omega^B)(\omega-\omega^B),\omega-\omega^B}=
        \frac{w_{\lambda}^{2n-4}}{|\omega^B|}
        \Braket{(\Id-\nu_{B^{\lambda}}\otimes\nu_{B^{\lambda}})(ux-\grad u),ux-\grad u} +\epsilon O\left(u^2+|\grad u|^2\right)
        \\
        &\quad=
        \frac{w_{\lambda}^{2n-4}}{|\omega^B|} \left[\frac{|\grad w_{\lambda}|^2}{w_{\lambda}^2+|\grad w_{\lambda}|^2} u^2-2\frac{w_{\lambda}\partial_iw_{\lambda}}{w_{\lambda}^2+|\grad w_{\lambda}|^2}u\partial_i u + |\grad u|^2-\frac{\partial_iw_{\lambda}\partial_jw_{\lambda}}{w_{\lambda}^2+|\grad w_{\lambda}|^2}\partial_i u\partial_j u\right]+\epsilon O\left(u^2+|\grad u|^2\right).
    \end{split}
    \end{equation}
    
Plugging \eqref{eq:zzExpansionGradient} and \eqref{eq:zzExpansionHessian} into \eqref{eq:expansion-perimeter}, we find
\begin{equation}\label{eq:zzExpansion1}
    \begin{split}
        P_{\lambda}(E)-P_{\lambda}(B^{\lambda}) &= \int_{\S^{n-1}_+} (n-1)w_{\lambda}^{n-1}u+\frac{(n-1)(n-2)}{2}w_{\lambda}^{n-2}u^2\\
        &\qquad\qquad
        +\frac{1}{2}
        \frac{w_{\lambda}^{2n-4}}{|\omega^B|}
         \,\,
        \left[\frac{|\grad w_{\lambda}|^2}{w_{\lambda}^2+|\grad w_{\lambda}|^2} u^2-2\frac{\Braket{w_{\lambda}\grad w_{\lambda},u\grad u}}{w_{\lambda}^2+|\grad w_{\lambda}|^2} + |\grad u|^2 - \frac{\Braket{\grad w_{\lambda},\grad u}^2}{w_{\lambda}^2+|\grad w_{\lambda}|^2}\right]\\
        &\qquad\qquad +\epsilon O\left(u^2+|\grad u|^2\right).
    \end{split}
\end{equation}
We can turn the first order term in the previous expansion into a second order term by exploiting the volume constraint. Indeed, by \cref{lemma:formulas} and by assumptions we get
\begin{equation}\label{eq:zzExpansion2}
    \int_{\S^{n-1}_+} \frac1n w_\lambda^n = |B^\lambda| = |E| = \int_{\S^{n-1}_+} \frac1n (w_\lambda + u)^n
    = \int_{\S^{n-1}_+} \frac1n w_\lambda^n + w_\lambda^{n-1} u + \frac{n-1}{2}w_\lambda^{n-2} u^2 + O(u^3).
\end{equation}
Moreover, since $\min w_\lambda>0$, then there is $c=c(n,\lambda)\in(0,1)$ such that $|\grad w_{\lambda}|^2 /(w_{\lambda}^2+|\grad w_{\lambda}|^2) \le c$, and thus
\begin{equation}\label{eq:zzExpansion3}
        |\grad u|^2-\frac{\Braket{\grad w_{\lambda},\grad u}^2}{w_{\lambda}^2+|\grad w_{\lambda}|^2}\geq (1-c^2)|\grad u|^2.
\end{equation}
On the other hand, by applying the standard inequality $ab\leq \frac{\delta a^2}{2}+\frac{b^2}{2\delta}$ with $a=|\nabla w_\lambda| \,|\grad u|$ and $b= w_\lambda\,|u|$ and $\delta= \delta(n,\lambda)>0$ small enough, we see that
\begin{equation}\label{eq:zzExpansion4}
        \left|2\frac{\Braket{w_{\lambda}\grad w_{\lambda},u\grad u}}{w_{\lambda}^2+|\grad w_{\lambda}|^2}\right|
        \le \frac{2}{w_{\lambda}^2+|\grad w_{\lambda}|^2} \left( \frac\delta2 |\nabla w_\lambda|^2 |\grad u|^2 + \frac{w_\lambda^2 u^2}{2 \delta}\right) 
        \leq \frac{1-c^2}{2}|\grad u|^2 + Cu^2,
\end{equation}
where $C>0$ depends only on $n, \lambda$. Applying \eqref{eq:zzExpansion2}, \eqref{eq:zzExpansion3}, \eqref{eq:zzExpansion4} in \eqref{eq:zzExpansion1}, the statement follows with
\[
    \begin{split}
        B(u)=\int_{\S^{n-1}_+}
        \frac{1-n}{2} w_\lambda^{n-2} u^2
        + 
        \frac{1}{2}
         \frac{w_{\lambda}^{2n-4}}{|\omega^B|}
         \,\,
        \Bigg\{|\grad u|^2- \frac{\Braket{\grad w_{\lambda},\grad u}^2}{w_{\lambda}^2+|\grad w_{\lambda}|^2} &- 2\frac{\Braket{w_{\lambda}\grad w_{\lambda},u\grad u}}{w_{\lambda}^2+|\grad w_{\lambda}|^2}\\
        &\quad+\frac{|\grad w_{\lambda}|^2}{w_{\lambda}^2+|\grad w_{\lambda}|^2}  u^2\Bigg\}\de\hausdorff^{n-1}.
    \end{split}
    \]
\end{proof}

\begin{proposition}[Fuglede-type estimate for capillarity problems]\label{prop:FugledeCapillarity}
There exist $C_{\ref{prop:FugledeCapillarity}}, \epsilon_{\ref{prop:FugledeCapillarity}}>0$ depending on $n,\lambda$ such that the following holds.
Let $E \subset H$ be an open bounded set with measure $|E|=|B^\lambda|$. Suppose that the boundary of $E$ is a $C^1$-graph over $\S^{n-1}_+$, see \cref{def:DistanzaC1}, with respect to a function $f= w_\lambda + u$, with $\|u \|_{C^1(\S^{n-1}_+)} \le\epsilon\le \epsilon_{\ref{prop:FugledeCapillarity}}$. Then, denoting $\barycenter_H E := \frac{1}{|E|}\int_E (x-\Braket{x,e_n}e_n)\de x$, there holds
\begin{equation*}
    \| u \|^2_{H^1(\S^{n-1}_+)} \le C_{\ref{prop:FugledeCapillarity}} (D_\lambda(E)+|\barycenter_H E|^2).
\end{equation*}
\end{proposition}

\begin{proof}
    The proof is inspired by \cite[Proposition~1.9]{Neumayer16}. To simplify the notation, let $b=\barycenter_H E\in\partial H$. We first claim that there exists a constant $C>0$, depending on $n,\lambda$, such that if $\epsilon_{\ref{prop:FugledeCapillarity}}$ is small enough then
    \begin{equation}\label{eq:L1-IsoperimetricDeficit}
        \|u\|_{L^1(\S^{n-1}_+)}^2\leq C (D_{\lambda}(E)+|b|^2).
    \end{equation}
    To prove this, we start by showing that $|E\symmdiff B^{\lambda}|\leq C (D_{\lambda}(E)^{1/2}+|b|)$. The quantitative inequality from \cref{thm:PPquantitativaFraenkelCapillare} guarantees that there exists a point $z\in\partial H$ such that $|E\symmdiff (z+B^{\lambda})|\leq C D_{\lambda}(E)^{1/2}$. Applying the triangle inequality and \cite[Lemma~17.9]{MaggiBook} we obtain the estimate
    \begin{equation}\label{eq:asymmetry+translation}
        |E\symmdiff B^{\lambda}| \leq C D_{\lambda}(E)^{1/2}+|B^{\lambda}\symmdiff(z+B^{\lambda})|\leq C D_{\lambda}(E)^{1/2} + 2|z|P(B^{\lambda}),
    \end{equation}
    where $C$ depends on $n,\lambda$ only and may change from line to line.
    By definition of $\barycenter_H E$, writing $x=(x',x_n) \in \R^n$, since $\int_{B^\lambda}x'\de x=0$, we get that
    \begin{equation*}\label{eq:translation-bound}
        \begin{split}
            z-b&= \frac{1}{|B^{\lambda}|}\int_{B^{\lambda}}(z-b)\de x 
            = \frac{1}{|B^{\lambda}|}\int_{B^{\lambda}}(x'+z-b)\de x - \frac{1}{|E|}\int_E (x'-b)\de x\\
        &= \frac{1}{|B^{\lambda}|}\int_{B^{\lambda}+z-b}x'\de x-\frac{1}{|E|}\int_{E-b}x'\de x.
        \end{split}
    \end{equation*}
    We observe that $|z|\leq 2\diam E$, and since that set is bounded by a unversal constant (as soon as $\epsilon_{\ref{prop:FugledeCapillarity}}<1$), then there exists a constant $C$ depending only on $n$ and $\lambda$ such that, for any $x\in B^{\lambda}+z-b$ and any $x\in E$, we have that $|x|\leq C$. Hence, using again \cref{thm:PPquantitativaFraenkelCapillare}, we estimate
    \[
        |z| \leq 
        \frac{1}{|B^\lambda|} \int_{(B^{\lambda}+z-b) \Delta (E-b)} |x'| \de x  + |b|
        \le
        C|E\symmdiff (z+B^{\lambda})|+|b| \leq CD_{\lambda}(E)^{1/2}+|b|.
    \]
    Plugging the latter estimate in~\eqref{eq:asymmetry+translation} we obtain that $|E\symmdiff B^{\lambda}|\leq C(D_{\lambda}(E)^{1/2}+|b|)$. In order to obtain~\eqref{eq:L1-IsoperimetricDeficit}, we may bound $\|u\|_{L^1}$ from above in terms of $|E\symmdiff B^{\lambda}|$. We perform an expansion as in~\eqref{eq:VolumeExpansion} to get
    \[
    \begin{split}
        |E\setminus B^{\lambda}| &= 
        \int_{\S^{n-1}_+} 
        \int_{w_\lambda}^{w_\lambda + u_+} r^{n-1} \de r \de \hausdorff^{n-1}
        =
        \int_{\{u>0\}} \frac{1}{n}(w_\lambda + u)^n - \frac{w_\lambda^n}{n} \de \hausdorff^{n-1}
        \\
        &=
        \int_{\{u>0\}}
        w_{\lambda}^{n-1}u+O(u^2)\de \hausdorff^{n-1} \geq C\int_{\{u>0\}}u\de \hausdorff^{n-1},
    \end{split}
    \]
    where the last estimate follows by choosing $\epsilon_{\ref{prop:FugledeCapillarity}}$ small enough, since $\inf w_\lambda >0$ on $\S^{n-1}_+$.
    An analogous computation shows $|B^{\lambda}\setminus E| \geq C\int_{\{u<0\}}(-u)\de\hausdorff^{n-1}$. Hence $|E\symmdiff B^{\lambda}|\geq C\|u\|_{L^1}$. Combining this with the upper bound for $|E\symmdiff B^{\lambda}|$ obtained before we readily get~\eqref{eq:L1-IsoperimetricDeficit}.
    
    Taking now $\epsilon_{\ref{prop:FugledeCapillarity}}<\epsilon_{\ref{prop:EspansionePlambda}}$ and applying \cref{prop:EspansionePlambda} we obtain that
    \[
    \begin{split}
        \|u\|_{H^1(\S_+^{n-1})}^2 =\|u\|_{L^2(\S_+^{n-1})}^2+\|\grad u\|_{L^2(\S_+^{n-1})}^2 \leq C\left[D_{\lambda}(E) + \|u\|_{L^2(\S_+^{n-1})}^2\right]+\epsilon O \left( \|u \|^2_{H^1(\S^{n-1}_+)} \right).
    \end{split}
    \]
    Let $\delta>0$ to be chosen. Applying \cref{lemma:NashIneq} with $\Sigma=\S^{n-1}_+$ we get
    \[
    \begin{split}
        \|u\|_{H^1(\S_+^{n-1})}^2&\leq C\left[D_{\lambda}(E)+
        \frac{C_{\ref{lemma:NashIneq}}}{\delta^{\alpha_n}}
        \|u\|_{L^1(\S_+^{n-1})}^2 \right] + (\epsilon+\delta)O \left( \|u \|^2_{H^1(\S^{n-1}_+)} \right)\\
        &\overset{\eqref{eq:L1-IsoperimetricDeficit}}{\leq} C\left( 1 +  \frac{C_{\ref{lemma:NashIneq}}}{\delta^{\alpha_n}}\right)(D_{\lambda}(E)+|b|^2)+(\epsilon+\delta)O \left( \|u \|^2_{H^1(\S^{n-1}_+)} \right).
    \end{split}
    \]
    To conclude the proof, it is sufficient to take $\delta=\epsilon_{\ref{prop:FugledeCapillarity}}$ small enough to absorb the last term on the right hand side in the last inequality.
\end{proof}

\subsection{Proof of the strong quantitative isoperimetric inequality for capillarity problems}\label{sec:proof-capillarity}

We recall the following basic lemma, which readily follows by the isoperimetric inequality \eqref{eq:zzIsoperimetricaCapillare} and by testing the functional with any rescaling of the isoperimetric bubble.

\begin{lemma}\label{lem:ProblemPlambda+LambdaVolumi}
For any $\bar\Lambda>n$, there holds
\begin{equation*}\label{eq:ProblemPlambda+LambdaVolumi}
    \inf\left\{
        P_\lambda(E) + \bar \Lambda \left| |E| - |B^\lambda| \right| \st E \subset H, \,\, |E|<+\infty
    \right\} = P_\lambda(B^\lambda),
\end{equation*}
and the infimum is only achieved by translations $B^\lambda + y$ of the standard bubble $B^\lambda$, for $y \in \partial H$.
\end{lemma}

We now get prepared for our formulation of the selection principle. From now on we let $\psi:\R\to \R$ be the function
\begin{equation}\label{eq:DefPSI}
    \psi(t) := \max\{ -100 , \min\{ t, 100\}\}.
\end{equation}
We define
\begin{equation}\label{eq:AlternativeBarycenter}
    \widetilde{b}(F) := \int_F \left( \psi(x_1), \ldots, \psi(x_{n-1}) \right) \de x,
\end{equation}
for any set $F\subset H$ with $|F|<+\infty$. 
Observe that if $F \subset B_{100}(0) \cap H$ then
\begin{equation*}\label{eq:RelationBarycenterBtilde}
    \widetilde{b}(F) = |F| \, {\rm bar}_H F.
\end{equation*}
Moreover
\begin{equation}\label{eq:LipschitzianitaDiBtilde}
\begin{split}
    \left| |\widetilde{b}(F)|^2 - |\widetilde{b}(G)|^2 \right| 
    &= \left||\widetilde{b}(F)| + |\widetilde{b}(G)| \right| \, \left| |\widetilde{b}(F)| - |\widetilde{b}(G)| \right| 
    \\&
    \le C(|F|+|G|) \left| \int_{F\Delta G} \left( \chi_F - \chi_G\right) (\psi(x_1),\ldots,\psi(x_{n-1}) ) \right|
    \\&
    \le C(|F|+|G|) |F\Delta G|,    
\end{split}
\end{equation}
for any $F,G\subset H$ with $|F|, |G| <+\infty$, for a universal constant $C>0$.\\
We recall that in \cref{def:capillarity-oscillation} we defined
\begin{equation*}
    \mu_{\lambda,0}^2(F) = \frac{1}{2} \int_{\partial^* F\cap H^+} \left|\nu_F(x)-\frac{x+\lambda e_n}{|x+\lambda e_n|}\right|^2\de \hausdorff^{n-1},
\end{equation*}
for any set of finite perimeter $F\subset H$ with $|F|<+\infty$.

We can fix $k_\lambda \in \N$ the least integer such that
\begin{equation}\label{eq:DefKlambda}
    k_\lambda\ge \max\{10, 4/(1-|\lambda|)\}.
\end{equation}
Hence for any set of finite perimeter $F\subset H$ with $|F|<+\infty$ and for any $k \in \N$ with $k\ge k_\lambda$, there holds
\begin{equation}\label{eq:PreConvention}
    P_\lambda(F) - \frac{1}{k}\mu_{\lambda,0}^2(F) \ge (1-|\lambda|) P(F,H^+) - \frac{2}{k} P(F,H^+) 
    \ge \frac{1-|\lambda|}{2} P(F,H^+) .
\end{equation}
Hence it is natural to define
\begin{equation}\label{eq:Convention}
    P_\lambda(F) - \frac{1}{k}\mu_{\lambda,0}^2(F) := + \infty,
\end{equation}
for any set $F\subset H$ with $|F|<+\infty$ and $P(F)=+\infty$, for any $k\ge k_\lambda$.

\begin{proposition}\label{prop:MinimizationProblemSelection}
    There exist $\Lambda_{\ref{prop:MinimizationProblemSelection}} >2n$, $k_{\ref{prop:MinimizationProblemSelection}} \ge k_\lambda, R_{\ref{prop:MinimizationProblemSelection}} \ge 10$,
    depending on $n,\lambda$ such that the following holds. For any $k \ge k_{\ref{prop:MinimizationProblemSelection}}$, $k \in \N$, adopting convention \eqref{eq:Convention}, the minimization problem
    \begin{equation}\label{eq:MinimizationProblemSelection}
        \inf \left\{
        \mathcal{F}_k(F) :=
            P_{\lambda}(F)+\Lambda_{\ref{prop:MinimizationProblemSelection}} \left||F|-|B^{\lambda}|\right|-\frac{1}{k}\mu_{\lambda,0}^2(F)+|\widetilde{b}(F)|^2
            \st 
            F\subset H, \,\, |F|<+\infty
        \right\}
    \end{equation}
    has an open minimizer $F_k$ such that
    \begin{itemize}
        \item $F_k \subset B_{R_{\ref{prop:MinimizationProblemSelection}}}(0)$,
    
         \item $|F_k|= |B^\lambda|$ for any $k\ge k_{\ref{prop:MinimizationProblemSelection}}$ sufficiently large,
    
         \item $\overline{\partial F_k \cap H^+}$ is a $C^{1,1/2}$-graph over $\S^{n-1}_+$ for any $k\ge k_{\ref{prop:MinimizationProblemSelection}}$ sufficiently large, and $\overline{\partial F_k \cap H^+} \to \overline{\partial B^\lambda \cap H^+}$ in $C^{1,\beta}$, for any $\beta \in (0,1/2)$, as $k\to \infty$, in the sense of \cref{def:DistanzaC1}.     
    \end{itemize}
\end{proposition}

\begin{proof}
Let $\Lambda> 2n$ and let $k \in \N$ with $k\ge k_\lambda$. We will fix $\Lambda=\Lambda_{\ref{prop:MinimizationProblemSelection}}$, depending on $n,\lambda$ only, at the end of the proof.
If $B^\lambda$ is a minimizer for \eqref{eq:MinimizationProblemSelection} for some $k\ge k_\lambda$, we can take $F_k= B^\lambda$, and any choice of $k_{\ref{prop:MinimizationProblemSelection}} \ge k_\lambda, R_{\ref{prop:MinimizationProblemSelection}} \ge 10$ satisfies the claim for such $k$'s. Hence, without loss of generality, we can assume that
\begin{equation*}
    \inf \left\{
    \mathcal{F}_k(F) 
        \st 
        F\subset H, \,\, |F|<+\infty
    \right\}
    < \mathcal{F}_k(B^\lambda) = P_\lambda(B^\lambda),
\end{equation*}
for any $k\ge k_\lambda$.

For any $k\ge k_\lambda$, let $\{\tilde E_j^k \st j \in \N, \, j\ge 1\}$ be a minimizing sequence for the $k$-th problem \eqref{eq:MinimizationProblemSelection} such that
\begin{equation*}
    \mathcal{F}_k(\tilde E_j^k) \le \min\left\{ \inf \left\{
    \mathcal{F}_k(F) 
        \st 
        F\subset H, \,\, |F|<+\infty
    \right\} 
    + \frac1j
    , \mathcal{F}_k(B^\lambda) \right\}
\end{equation*}

We claim that there exist $\tilde{k}_{\ref{prop:MinimizationProblemSelection}}= \tilde{k}_{\ref{prop:MinimizationProblemSelection}}(n,\lambda)\ge k_\lambda$,
$K_{\ref{prop:MinimizationProblemSelection}} = K_{\ref{prop:MinimizationProblemSelection}}(n,\lambda,\Lambda)\ge 1$, $r_{\ref{prop:MinimizationProblemSelection}} = r_{\ref{prop:MinimizationProblemSelection}}(n,\Lambda, \lambda)  \in (0,1)$ such that there exists a new minimizing sequence $\{E_j^k \st j \in \N,\, j \ge 1\}$ for the $k$-th problem \eqref{eq:MinimizationProblemSelection} such that $E_j^k$ is a $(K_{\ref{prop:MinimizationProblemSelection}}, r_{\ref{prop:MinimizationProblemSelection}})$-quasi-minimizer for any $k \ge \tilde{k}_{\ref{prop:MinimizationProblemSelection}}$ and any $j \ge 1$.\\
To obtain the new sequence $\{E_j^k \st j \ge 1\}$, we apply Ekeland's variational principle, recalled in \cref{thm:Ekeland}, on the complete metric space
\[
    (\mathbb X,\d) := \Big( \{ F\subset H \st |F|<+\infty \}, \,\,\, \d(F,G):= \|\chi_F - \chi_G\|_{L^1(H)} =  |F \Delta G| \} \Big),
\]
and on the function $\mathcal{F}_k:\mathbb X \to [0,+\infty]$ for $k \ge k_\lambda$. In fact, the function $\mathcal{F}_k$ is lower semicontinuous on $(\mathbb X,\d)$ for any $k\ge k_\lambda$. Indeed, the summand $F \mapsto\Lambda\left| |F| - |B^\lambda| \right| + |\widetilde{b}(F)|^2$ is continuous with respect to $L^1(H)$ convergence. 
While, by \cref{lem:RiscritturaDiMuLambdaZero}, we can rewrite the other terms in $\mathcal{F}_k$ as
    \begin{equation}\label{eq:RewritingAsAnisotropicPerimeter}
    \begin{split}
        P_{\lambda}(F)&-\frac{1}{k}\mu_{\lambda,0}^2(F) =\\
        &= \int_{\partial^* F\cap H^+}1-\lambda\Braket{\nu_F(x),e_n}-\frac{1}{k}\left(1-\Braket{\nu_F(x),\frac{\lambda e_n}{\left|x-\Braket{x,e_n}e_n+\lambda e_n\right|}}\right)\de\hausdorff^{n-1}
         \\
        &\qquad\qquad\qquad\qquad\qquad\,\,\,\qquad
        -\frac{1}{k}\int_F\frac{n-1}{|x+\lambda e_n|}\de x\\
        &=
        P_{\Phi_k}(F) -\frac{1}{k}\int_F\frac{n-1}{|x+\lambda e_n|}\de x,
    \end{split}
    \end{equation}
    where $P_{\Phi_k}(F) := \int_{\partial^*F \cap H^+}\Phi_k(x, \nu_F) \de \hausdorff^{n-1}$ is the anisotropic perimeter with density $\Phi_k:H^+ \times \R^n \to \R$ defined by
    \begin{equation}\label{eq:DefPhik}
        \Phi_k(x, \nu) := \left( 1- \frac{1}{k}\right) |\nu| + \Braket{\nu, \left( \frac{1}{k}  \, \frac{1}{\left|x-\Braket{x, e_n}e_n + \lambda e_n\right|} -1\right) \lambda e_n }.
    \end{equation}
    The volume term $\int_F\frac{n-1}{|x+\lambda e_n|}\de x$ appearing in \eqref{eq:RewritingAsAnisotropicPerimeter} is clearly continuous with respect to $L^1(H)$ convergence of sets. On the other hand, for any $|\nu|=1$ there holds
    \[
    \begin{split}
        |\Phi_k(x,\nu)|
        &\ge 1- \frac{1}{k} - |\lambda|\left( 1 + 
        \frac{1}{k}  \, \frac{1}{|x-\Braket{x, e_n}e_n + \lambda e_n|} \right) 
        \ge 1- \frac{1}{k} - |\lambda|\left( 1 + 
        \frac{1}{k|\lambda|} \right) >0,
    \end{split}
    \]
for any $x\in H^+$ for any $k \ge k_\lambda$. Hence, by Reshetnyak lower semicontinuity theorem (see \cite[Theorem 2.38]{AmbrosioFuscoPallara} or \cite[Theorem~1.2]{Spector2011}), it follows that $P_{\Phi_k}$ is lower semicontinuous with respect to $L^1(H)$ convergence of sets.

Therefore, we can apply Ekeland's variational principle \cref{thm:Ekeland} with $x_0= \tilde{E}_j^k$, $\eps=1/j$ and $s=1$, for any $k \ge k_\lambda$, $j\ge 1$. We obtain sets $E_j^k \subset H$ with finite measure such that
\begin{equation}\label{eq:EkelandApplied}
\begin{split}
    \mathcal{F}_k(E_j^k) &\le \min\left\{ \inf \left\{
    \mathcal{F}_k(F) 
        \st 
        F\subset H, \,\, |F|<+\infty
    \right\} 
    + \frac1j
    , \mathcal{F}_k(B^\lambda) \right\}, \\
    \mathcal{F}_k(E_j^k)  &\le \mathcal{F}_k(G) + \frac1j |G \Delta E_j^k|, 
\end{split}
\end{equation}
for any $G\subset H$ with finite measure, for any $k \ge k_\lambda$, $j\ge 1$. We are left to check that $E_j^k$ is a quasi-minimizer, uniformly for large $k$ and $j\ge 1$ as claimed above. 

We first notice that
\begin{equation}\label{eq:MassMinimizingSeqCloseToBlambda}
   \frac12 \le \frac{|E_j^k|}{|B^\lambda|} \le \frac32,
\end{equation}
for any $k \ge k_\lambda$ and any $j\ge 1$. Indeed 
\begin{equation*}
    2n \left||E_j^k| - |B^\lambda| \right| < \Lambda  \left||E_j^k| - |B^\lambda| \right|
    \overset{\eqref{eq:PreConvention}}{\le} \mathcal{F}_k(E_j^k)
    \overset{\eqref{eq:EkelandApplied}}{\le} \mathcal{F}_k(B^\lambda) = P_\lambda(B^\lambda) = n |B^\lambda|.
\end{equation*}
Now let $G\subset H$ be a set of finite perimeter such that $G \Delta E_j^k \Subset B_r(x_0)$ for some $r\in(0,1), x_0 \in H$. 
Then
\begin{equation}\label{eq:zzmainstima}
    \begin{split}
        P_\lambda(E_j^k) &
        \overset{\eqref{eq:EkelandApplied}}{\le} \mathcal{F}_k(G) +  \frac1j |G \Delta E_j^k| -\Lambda \left| |E_j^k| - |B^\lambda| \right| + \frac{1}{k} \mu^2_{\lambda, 0}(E_j^k) - |\widetilde{b}(E_j^k)|^2 \\
        &\le \frac1j |G \Delta E_j^k|+ P_\lambda(G) +\Lambda \left| |G|-|E_j^k| \right| + \frac{1}{k} \left( \mu^2_{\lambda, 0}(E_j^k) - \mu^2_{\lambda, 0}(G) \right) 
        + |\widetilde{b}(G) |^2 - |\widetilde{b}(E_j^k)|^2
        \\
        & \overset{\eqref{eq:LipschitzianitaDiBtilde}}{\le} P_\lambda(G) + \left(\Lambda + \frac{1}{j}\right) |G \Delta E_j^k| 
        +\frac{1}{2k}\bigg(\int_{\partial^* E_j^k\cap H^+\cap B_r(x_0)}\left|\nu_{E_j^k}(x)-\frac{x+\lambda e_n}{|x+\lambda e_n|}\right|^2\de\hausdorff^{n-1}
         \\ 
         &\qquad -\int_{\partial^* G\cap H^+\cap B_r(x_0)}\left|\nu_{G}(x)-\frac{x+\lambda e_n}{|x+\lambda e_n|}\right|^2\de\hausdorff^{n-1} \bigg) 
        +C( n, \lambda) |G \Delta E_j^k|
        \\
        &\le 
         P_\lambda(G) + \left(\Lambda + \frac{1}{j}+ C( n, \lambda) \right) C_n r P(G \Delta E_j^k , H^+) \\
         &\qquad +\frac{2}{k} \left( P(E_j^k, B_r(x_0)\cap H^+)+P(G, B_r(x_0)\cap H^+) \right) \\
        &\le P_\lambda(G) + \left(\left( \Lambda + \frac{1}{j}+ C( n, \lambda) \right) C_n r 
        +\frac{2}{k}\right)\cdot
        \left( P(E_j^k;B_r(x_0)\cap H^+)+P(G;B_r(x_0)\cap H^+) \right) 
    \end{split} 
\end{equation}
where we estimated
\[
\begin{split}
    |G \Delta E_j^k | &= |G \Delta E_j^k |^{\frac1n}|G \Delta E_j^k |^{\frac{n-1}{n}} \le |B_r(x_0)|^{\frac1n} |G \Delta E_j^k |^{\frac{n-1}{n}} \le C_n r P(G \Delta E_j^k, H^+)\\
    &\le
    C_nr\left( P(E_j^k;B_r(x_0)\cap H^+)+P(G;B_r(x_0)\cap H^+) \right) ,
\end{split}
\]
for a dimensional constant $C_n$ obtained using the isoperimetric inequality in a half-space. Now recalling that
\[
P_\lambda(G) = \int_{\partial^* G} 1-\lambda\Braket{\nu_G, e_n}
= P_\lambda(E_j^k)
+
\int_{\partial^* G \cap B_r(x_0) \cap H^+} 1-\lambda\Braket{\nu_G, e_n}
- \int_{\partial^* E_j^k \cap B_r(x_0) \cap H^+} 1-\lambda\Braket{\nu_{E_j^k}, e_n},
\]
we conclude that
\begin{equation*}
\begin{split}
    (1-|\lambda|)& P(E_j^k, B_r(x_0)\cap H^+)
    \le (1+|\lambda|) P(G,B_r(x_0)\cap H^+) +\\
    &\qquad+\left(\left( \Lambda +\frac{1}{j}+ C( n, \lambda) \right) C_n r 
        +\frac{2}{k}\right)
        \left( P(E_j^k;B_r(x_0)\cap H^+)+P(G;B_r(x_0)\cap H^+) \right) 
\end{split}
\end{equation*}
It follows that there exist $\tilde{k}_{\ref{prop:MinimizationProblemSelection}}\ge k_\lambda$ depending on $n,\lambda$, and $r_{\ref{prop:MinimizationProblemSelection}} >0$ depending on $\Lambda, n,\lambda$ such that, as long as $r \le r_{\ref{prop:MinimizationProblemSelection}}$, there holds
\begin{equation*}
    P(E_j^k,B_r(x_0)\cap H^+) \le K_{\ref{prop:MinimizationProblemSelection}} P(G,B_r(x_0)\cap H^+),
\end{equation*}
for a computable constant $K_{\ref{prop:MinimizationProblemSelection}} >0$ depending on $n,\lambda,\Lambda$, for any $k \ge \tilde{k}_{\ref{prop:MinimizationProblemSelection}}$ and any $j\ge 1$.

It follows from the regularity theory for quasi-minimizers, see \cref{thm:quasi-regularity}, that $E_j^k$ has an open representative and there exist $D_{\ref{prop:MinimizationProblemSelection}}, \rho_{\ref{prop:MinimizationProblemSelection}}\in(0,1)$ depending on $n,\lambda,\Lambda$ such that
\begin{equation}\label{eq:DensityMinimizingSeq}
    |E_j^k \cap B_\rho(x) | \ge D_{\ref{prop:MinimizationProblemSelection}} \rho^n,
\end{equation}
for any $\rho \le \rho_{\ref{prop:MinimizationProblemSelection}}$, $x \in \overline{E_j^k}$, $k \ge \tilde{k}_{\ref{prop:MinimizationProblemSelection}}$ and any $j\ge 1$.

\medskip

Up to enlarging the value of $\tilde{k}_{\ref{prop:MinimizationProblemSelection}}$, depending only on $n,\lambda$, we now want to prove the following crucial claim. Let $k_i\nearrow\infty$ be a subsequence, with $k_i\ge \tilde{k}_{\ref{prop:MinimizationProblemSelection}}$ for any $i$, and let $j_i \ge k_i$ be another subsequence. Let $A_i:= E_{j_i}^{k_i}$. We claim that 
\begin{equation}\label{eq:ClaimKEY}
|A_i\symmdiff B^{\lambda}|\leq \frac{\bar C_{n,\lambda}}{\sqrt{k_i}} \quad \forall i, \qquad
    \text{$\overline{\partial A_i \cap H^+} \to \overline{\partial B^\lambda \cap H^+}$ in Hausdorff distance, as $i\to \infty$.}
\end{equation}

    By definition of $A_i$, we have that $\mathcal{F}_{k_i}(A_i)\leq\mathcal{F}_{k_i}(B^{\lambda})$, and thus
    \begin{equation}\label{eq:subsub-Plambda}
        P_{\lambda}(A_i)+\Lambda||A_i|-|B^{\lambda}||-\frac{1}{k_i}\mu_{\lambda,0}^2(A_i)+|\widetilde b(A_i)|^2 \leq P_{\lambda}(B^{\lambda}),
    \end{equation}
    for any $i$.
    Since $k_i\geq k_{\lambda}$, by~\eqref{eq:PreConvention} we deduce that $P(A_i, H^+)\leq \frac{2}{1-|\lambda|}P_{\lambda}(B^{\lambda})$. Hence $\mu_{\lambda,0}^2(A_i) \le 2 P(A_i, H^+)\leq \frac{4}{1-|\lambda|}P_{\lambda}(B^{\lambda})$. Since $\Lambda>2n$, using the parameter $\bar\Lambda=\Lambda-n>n$ in \cref{lem:ProblemPlambda+LambdaVolumi}, we deduce from~\eqref{eq:subsub-Plambda} that
    \begin{equation}\label{eq:difference-measure-Plambda}
        n||A_i|-|B^{\lambda}||\leq \frac{1}{k_i}\mu_{\lambda,0}^2(A_i)\leq \frac{1}{k_i}\frac{4P_{\lambda}(B^{\lambda})}{1-|\lambda|} = \frac{1}{k_i}\frac{4n|B^{\lambda}|}{1-|\lambda|},
    \end{equation}
    for any $i$.
    Using \cref{thm:PPquantitativaFraenkelCapillare}, rearranging terms in~\eqref{eq:subsub-Plambda} and introducing the parameter $t_i:=\frac{|A_i|}{|B^{\lambda}|}$, we get
    \begin{equation}\label{eq:alpha-bar-Plambda}
        C_{n,\lambda}t_i^{\frac{n-1}{n}}P_{\lambda}(B^{\lambda})\alpha_{\lambda}^2(A_i)+|\widetilde b(A_i)|^2\leq P_{\lambda}(B^{\lambda})\left(1-t_i^{\frac{n-1}{n}}\right)+\frac{4P_{\lambda}(B^{\lambda})}{k_i(1-|\lambda|)},
    \end{equation}
    for some $C_{n,\lambda}>0$ and for any $i$, where we also used the bound $\mu_{\lambda,0}^2(A_i)\leq \frac{4}{1-|\lambda|}P_{\lambda}(B^{\lambda})$.\\
    Notice that $|t_i-1|\leq \frac{4}{k_i(1-|\lambda|)}$ for any $i$ thanks to~\eqref{eq:difference-measure-Plambda}. Therefore, we obtain $C_{n,\lambda}\alpha_{\lambda}^2(A_i)+|\widetilde b(A_i)|^2\leq \frac{C_{n,\lambda}'}{k_i}$ for any $i$, for some constants $C_{n,\lambda}, C'_{n,\lambda}>0$ that may change from line to line. This shows that, denoting by $x_i\in\partial H$ a point realizing the asymmetry $\alpha_{\lambda}(A_i)$, we get
    \begin{equation}\label{eq:zaqs}
        |A_i\symmdiff (x_i+B^{\lambda})| \leq \left|A_i\symmdiff (x_i+t_i^{1/n}B^{\lambda})\right|+\left|t_i^{1/n}B^{\lambda}\symmdiff B^{\lambda}\right|\overset{\eqref{eq:alpha-bar-Plambda}}{\leq} \frac{C_{n,\lambda}}{\sqrt{k_i}}+||A_i|-|B^{\lambda}|| \overset{\eqref{eq:difference-measure-Plambda}}{\leq} \frac{C_{n,\lambda}}{\sqrt{k_i}},
    \end{equation}
    for any $i$.\\
    Now it can be readily checked that $|\widetilde{b}(x_i + B^\lambda)| \ge c_{n,\lambda} \min\{1, |x_i|\}$ for any $i$. Moreover from~\eqref{eq:alpha-bar-Plambda} we also deduce that $|\widetilde b(A_i)|^2\leq \frac{C_{n,\lambda}}{k_i}$ for any $i$. Hence, arguing as in \eqref{eq:LipschitzianitaDiBtilde}, we estimate
    \begin{equation*}
        c_{n,\lambda} \min\{1, |x_i|\} \le |\widetilde b(A_i)| + \left| |\widetilde b(x_i+B^\lambda)| -  |\widetilde b(A_i)|\right|
        \le \frac{C_{n,\lambda}}{\sqrt{k_i}} + C_{n,\lambda} |A_i\symmdiff (x_i+B^{\lambda})|
        \overset{\eqref{eq:zaqs}}{\le} \frac{C_{n,\lambda}}{\sqrt{k_i}},
    \end{equation*}
    for any $i$. So there is $\bar{k}_{\ref{prop:MinimizationProblemSelection}}$ depending on $n,\lambda$ such that $\min\{1, |x_i|\} =|x_i|$ whenever $k_i \ge \bar{k}_{\ref{prop:MinimizationProblemSelection}}$, and the previous estimate yields an upper bound on $|x_i|$.
    Recalling \eqref{eq:zaqs}, this immediately implies that $|A_i\symmdiff B^{\lambda}|\leq \frac{\bar C_{n,\lambda}}{\sqrt{k_i}}$ for any $i$ such that $k_i \ge \bar{k}_{\ref{prop:MinimizationProblemSelection}}$. Up to updating $\tilde{k}_{\ref{prop:MinimizationProblemSelection}}$ with $\max\{ \tilde{k}_{\ref{prop:MinimizationProblemSelection}}, \bar{k}_{\ref{prop:MinimizationProblemSelection}}\}$, we can assume that $\tilde{k}_{\ref{prop:MinimizationProblemSelection}}\ge  \bar{k}_{\ref{prop:MinimizationProblemSelection}}$, whence $|A_i\symmdiff B^{\lambda}|\leq \frac{\bar C_{n,\lambda}}{\sqrt{k_i}}$ holds for any $i$.\\
    We already know that the sets $A_i$ satisfy the density estimates~\eqref{eq:DensityMinimizingSeq}. Up to decreasing $D_{\ref{prop:MinimizationProblemSelection}}, \rho_{\ref{prop:MinimizationProblemSelection}}$, we can assume that the same density estimates hold for $B^\lambda$.
    It is then standard to prove that whenever $\sqrt{k_i}>\tfrac{\bar C_{n,\lambda}}{D_{\ref{prop:MinimizationProblemSelection}}\rho_{\ref{prop:MinimizationProblemSelection}}^n}$, then $d_{\hausdorff}(\overline{\partial A_i\cap H^+}, \overline{\partial B^{\lambda}\cap H^+})\leq \rho_{\ref{prop:MinimizationProblemSelection}}$, hence    
    $|A_i\symmdiff B^{\lambda}|\geq D_{\ref{prop:MinimizationProblemSelection}}d_{\hausdorff}(\overline{\partial A_i\cap H^+}, \overline{\partial B^{\lambda}\cap H^+})^n$. Therefore
    \[
        d_{\hausdorff}(\overline{\partial A_i\cap H^+}, \overline{\partial B^{\lambda}\cap H^+})\leq C_{n,\lambda} k_i^{-1/2n}\qquad \text{when }\sqrt{k_i}>\frac{C_{n,\lambda}}{D_{\ref{prop:MinimizationProblemSelection}}\rho_{\ref{prop:MinimizationProblemSelection}}^n},
    \]
This completes the proof of claim \eqref{eq:ClaimKEY}.

\medskip

The key claim \eqref{eq:ClaimKEY} immediately implies
\begin{equation}\label{eq:FinalClaim}
    \exists\, \hat{k}_{\ref{prop:MinimizationProblemSelection}}(n,\lambda, \Lambda) \ge \tilde{k}_{\ref{prop:MinimizationProblemSelection}}, \, \hat{R}_{\ref{prop:MinimizationProblemSelection}}(n,\lambda, \Lambda)>0 \st \forall\, k \ge \hat{k}_{\ref{prop:MinimizationProblemSelection}}, \,\, j \ge k
    \implies
    E_j^k \subset B_{R_{\ref{prop:MinimizationProblemSelection}}}(0).
\end{equation}
Indeed, negating \eqref{eq:FinalClaim} means that for any $i \in \N$, $i \ge \tilde{k}_{\ref{prop:MinimizationProblemSelection}}$, there is $k_i\ge i$ and $j_i\ge k_i$ such that $|B_i(0) \setminus E_{j_i}^{k_i}| >0$. But \eqref{eq:DensityMinimizingSeq} together with \eqref{eq:ClaimKEY} give a contradiction.

For $k \ge \hat{k}_{\ref{prop:MinimizationProblemSelection}}$, letting $j\to \infty$ in the sequence $E_j^k$, by \eqref{eq:FinalClaim} we find a minimizer $F_k \subset B_{\hat{R}_{\ref{prop:MinimizationProblemSelection}}}$ for the minimization problem \eqref{eq:MinimizationProblemSelection}. Arguing as we did before with $F_k$ in place of $E_j^k$, it clearly follows that $F_k$ is a $(K_{\ref{prop:MinimizationProblemSelection}}, r_{\ref{prop:MinimizationProblemSelection}})$-quasi-minimizer, hence it has an open representative. Moreover, repeating the arguments of Step 2 above applied on the sequence $F_k$, we obtain that for every $k>\max\left\{\hat{k}_{\ref{prop:MinimizationProblemSelection}}, \left(\frac{C_{n,\lambda}}{D_{\ref{prop:MinimizationProblemSelection}}\rho_{\ref{prop:MinimizationProblemSelection}}^n}\right)^2\right\}$
\begin{equation}\label{eq:zzConvergenzaHausFk}
    |F_k\symmdiff B^{\lambda}|\leq \frac{\bar C_{n,\lambda}}{\sqrt{k}},\qquad d_{\hausdorff}(\overline{\partial F_k\cap H^+}, \overline{\partial B^{\lambda}\cap H^+})\leq C_{n,\lambda} k^{-1/2n}.
\end{equation}

\medskip

In order to improve convergence of boundaries to $C^{1,\beta}$-convergence, we wish to apply \cref{thm:almost-regularity}. We first show that there exist $\Lambda'_{\ref{prop:MinimizationProblemSelection}}, r'_{\ref{prop:MinimizationProblemSelection}}>0$, depending on $n,\lambda, \Lambda$, such that $F_k$ is a $(\Lambda'_{\ref{prop:MinimizationProblemSelection}}, r'_{\ref{prop:MinimizationProblemSelection}})$-almost minimizer of $P_{\Phi_k}$ in $H$, for any large $k$. 
So let $G\subset H$ be a set of finite perimeter such that $G \Delta F^k \Subset B_r(x_0)$ for some $r\in(0,\tfrac14)$. Assume first that $x_0 \in H \setminus B_{\frac12}(-\lambda e_n)$. By \eqref{eq:RewritingAsAnisotropicPerimeter}, by minimality of $F_k$, estimating $|\widetilde{b}(G)|^2-|\widetilde{b}(F_k)|^2$ as done in \eqref{eq:zzmainstima}, we find
\begin{equation*}
    \begin{split}
        P_{\Phi_k}(F_k) 
        &\le 
               P_{\Phi_k}(G) + (C(n,\lambda)+ \Lambda) |G \Delta F_k|
               + \frac{n-1}{k} \left(
               \int_{F_k} \frac{1}{|x+\lambda e_n|}
                - \int_{G} \frac{1}{|x+\lambda e_n|}
               \right)
        \\
        &
        \le 
            P_{\Phi_k}(G) + (C(n,\lambda)+ \Lambda) |G \Delta F_k|
               + \frac{n-1}{k} \left(
               \int_{G\Delta F_k} \frac{1}{|x+\lambda e_n|}
               \right).
    \end{split}
\end{equation*}
Since  $r\in(0,\tfrac14)$ and $x_0 \in H \setminus B_{\frac12}(-\lambda e_n)$, then $\frac{1}{|x+\lambda e_n|} \le 4$ almost everywhere on $G\Delta F_k$, hence
\begin{equation*}
    \begin{split}
        P_{\Phi_k}(F_k, B_r(x_0)) 
        &\le
            P_{\Phi_k}(G, B_r(x_0)) + \left( C(n,\lambda)+ \Lambda + \frac{4(n-1)}{k} \right) |G \Delta F_k|\\&
            \le 
             P_{\Phi_k}(G, B_r(x_0)) + \Lambda'_{\ref{prop:MinimizationProblemSelection}} |G \Delta F_k|.
    \end{split}
\end{equation*}
On the other hand, if $x_0 \in B_{\frac12}(-\lambda e_n)$ instead, by \eqref{eq:zzConvergenzaHausFk} for $k$ large we simply have $P_{\Phi_k}(F_k, B_r(x_0)) =0 \le P_{\Phi_k}(G, B_r(x_0)) + \Lambda'_{\ref{prop:MinimizationProblemSelection}} |G \Delta F_k|.$ 
Hence the claimed almost minimality of $F_k$ for $P_{\Phi_k}$ follows. Moreover, recalling the definition \eqref{eq:DefPhik} of $\Phi_k$, it can be readily checked that $\Phi_k$ is a regular elliptic integrand with $\Phi_k \in \mathcal{E}(H^+,l_1,l_2)$ for some constants $l_1\ge1, l_2 >0$, for any $k$ large, in the sense of \cref{def:unif-elliptic-perimeter}. Therefore, recalling \eqref{eq:zzConvergenzaHausFk}, we can apply \cref{thm:almost-regularity}, which implies that $\overline{\partial F_k \cap H^+}$ is a $C^{1,1/2}$-graph over $\S^{n-1}_+$ for large $k$ and that $\overline{\partial F_k \cap H^+} \to \overline{\partial B^\lambda \cap H^+}$ in $C^{1,\beta}$ for any $\beta \in (0,1/2)$, as $k\to \infty$.

\medskip

Finally we want to prove that there is a choice of $\Lambda = \Lambda_{\ref{prop:MinimizationProblemSelection}}(n,\lambda)$ such that, taking $R_{\ref{prop:MinimizationProblemSelection}}:=\hat{R}_{\ref{prop:MinimizationProblemSelection}}(n,\lambda,$ $\Lambda_{\ref{prop:MinimizationProblemSelection}})$ and $k_{\ref{prop:MinimizationProblemSelection}} = \hat{k}_{\ref{prop:MinimizationProblemSelection}}(n,\lambda,\Lambda_{\ref{prop:MinimizationProblemSelection}})$, then the statement follows. So we are left to prove that there is such a choice of $\Lambda$ so that $|F_k|=|B^\lambda|$ for any $k \ge k_{\ref{prop:MinimizationProblemSelection}}$ large enough.\\
Similarly to \cite[Step 3 in the proof of Theorem~1.1]{Neumayer16}, we test the minimality of $F_k$ with the competitor $s_kF_k$, where $s_k = (|B^{\lambda}|/|F_k|)^{1/n}$. We get
    \[
    \begin{split}
        P_{\lambda}(F_k)&+\Lambda||F_k|-|B^{\lambda}||-\frac{1}{k}\mu_{\lambda,0}^2(F_k)+
        \left| 
        \int_{F_k} \psi(x') \de x
        \right|^2
        \\ &
        \leq s_k^{n-1}P_{\lambda}(F_k)-\frac{1}{k}\mu_{\lambda,0}^2(s_kF_k)+s_k^{2n}
        \left| 
        \int_{F_k} \psi(s_k x') \de x
        \right|^2
        ,
    \end{split}
    \]
where $\psi(z'):=(\psi(z_1), \ldots, \psi(z_{n-1}))$ for $z = (z',z_n) \in \R^n$. Since $\psi$ is $1$-Lipschitz, then $|\psi(s_k x') - \psi(x')| \le |s_k-1||x'| \le 10 |s_k-1|$ for any $x \in F_k$ and $k \ge \hat{k}_{\ref{prop:MinimizationProblemSelection}}$ large enough by \eqref{eq:zzConvergenzaHausFk}. Hence
\begin{equation*}
\begin{split}
    \Lambda||F_k|-|B^{\lambda}|| & \leq |s_k^{n-1}-1| P_{\lambda}(F_k)
        +\frac{1}{k} \left( \mu_{\lambda,0}^2(F_k) -\mu_{\lambda,0}^2(s_kF_k) \right)
        +|s_k^{2n} -1|\, |\widetilde{b}(F_k)|^2
        \\
        &\qquad + s_k^{2n}(|\widetilde b(s_kF_k)|+|\widetilde{b}(F_k)|)(|\widetilde b(s_kF_k)|-|\widetilde{b}(F_k)|)\\
        &\leq |s_k^{n-1}-1| P_{\lambda}(F_k)
        +\frac{1}{k} \left( \mu_{\lambda,0}^2(F_k) -\mu_{\lambda,0}^2(s_kF_k) \right)
        +|s_k^{2n} -1|\, |\widetilde{b}(F_k)|^2
        \\
        &\qquad+ 10 s_k^{2n}(|\widetilde b(s_kF_k)|+|\widetilde{b}(F_k)|)|s_k-1|\, |F_k|
\end{split}
\end{equation*}
We already know that $F_k\to B^{\lambda}$ in $L^1(H)$, and thus $s_k\to1$. Moreover, there exists a geometric constant $C>0$ such that
    \begin{equation}\label{eq:GeometricEstimate}
        \left|v-\frac{w+z}{|w+z|}\right|^2\geq \left|v-\frac{w}{|w|}\right|^2-C|z|,
    \end{equation}
    for any $v,w,z\in\R^n$ with $|v|=1$, $|w|\geq 1/2$ and $|z|\leq 1/4$.    
    Recalling that $\partial F_k\cap H^+\subset B_{2}(-\lambda e_n)\setminus B_{1/2}(-\lambda e_n)$ for $k$ large, we can apply \eqref{eq:GeometricEstimate} with $v=\nu_{F_k}(x)$, $w=x+\lambda e_n$ and $z=(s_k-1)x$, for $x\in\partial F_k\cap H^+$.
    In particular, as also $P_{\lambda}(F_k), |\widetilde{b}(F_k)|, \mu_{\lambda,0}(F_k), P(F_k) \le C(n,\lambda)$ for any $k$ large, we obtain
    \[
    \begin{split}
        \Lambda &\left||F_k|-|B^{\lambda}|\right|\leq 
        C(|s_k^{n-1}-1| + |s_k^{2n} -1 | +  |s_k-1|)
        +\\
        & + \frac{1}{2k}\left(\int_{\partial^* F_k \cap H^+}\left|\nu_{F_k}(x)-\frac{x+\lambda e_n}{|x+\lambda e_n|}\right|^2\de\hausdorff^{n-1}
        -
        s_k^{n-1}\int_{\partial^* F_k \cap H^+}\left|\nu_{F_k}(x)-\frac{s_kx+\lambda e_n}{|s_kx+\lambda e_n|}\right|^2\de\hausdorff^{n-1}\right)\\
        &\overset{\eqref{eq:GeometricEstimate}}{\leq} 
        C(|s_k^{n-1}-1| + |s_k^{2n} -1 | +  |s_k-1|)
        + \frac{1}{2k} |s_k^{n-1}-1|\int_{\partial^* F_k \cap H^+}\left|\nu_{F_k}(x)-\frac{x+\lambda e_n}{|x+\lambda e_n|}\right|^2\de\hausdorff^{n-1}\\
        &\qquad + \frac{s_k^{n-1}}{2k}\int_{\partial^* F_k \cap H^+} C|s_k-1||x|\de\hausdorff^{n-1}\\
        &\leq C|s_k-1|,
    \end{split}
    \]
    where we used that the power functions $t\to |t^{n-1}-1|$ and $t\to|t^{2n}-1|$ are locally Lipschitz. On the other hand
    \[
    \left||F_k|-|B^{\lambda}|\right| = |F_k| |s_k^n -1| = |F_k| |s_k-1| (1+s_k+\ldots+s_k^{n-1}) \overset{\eqref{eq:MassMinimizingSeqCloseToBlambda}}{\ge} \frac 12 |B^\lambda| |s_k-1|.
    \]
    Therefore, if by contradiction $s_k\neq 1$ for some $k$ large, the last two estimates imply $\Lambda \le C=C(n,\lambda)$, which is impossible for a choice of $\Lambda$ sufficiently large depending on $n,\lambda$ only.
\end{proof}

\begin{theorem}[Strong quantitative inequality in capillarity problems]\label{thm:capillarityStrongQuantitative}
There exists $C_{\ref{thm:capillarityStrongQuantitative}}>0$ depending on $n,\lambda$ such that for any set of finite perimeter $E \subset H$ with $|E|<+\infty$ there holds
\begin{equation*}\label{eq:strong-quantitative-capillarity}
    \mu_\lambda^2(E) \le C_{\ref{thm:capillarityStrongQuantitative}} D_\lambda(E).
\end{equation*}
\end{theorem}

\begin{proof}
By scale invariance of the oscillation asymmetry and of the deficit, we can prove the claim for sets having volume equal to $|B^\lambda|$. Suppose by contradiction that there exists a sequence of sets of finite perimeter $E_k\subset H$ such that $|E_k|=|B^{\lambda}|$ but $k^2D_{\lambda}(E_k)< \mu_{\lambda}^2(E_k)$, for any $k \in \N$ with $k\ge 1$. In particular, $E_k$ does not coincide with the translation of a truncated ball $B^{\lambda}$ for every $k\in\N$, and thus $D_{\lambda}(E_k)>0$. Without loss of generality, up to translation along $\partial H$ we can assume that $\widetilde{b}( E_k)=0$. Since $\mu_{\lambda}^2(E_k)\leq 2P(E_k;H^+)\leq \frac{2}{1-|\lambda|}P_{\lambda}(E_k)$, it follows that $E_k$ satisfies $(k^2-C)P_{\lambda}(E_k)\leq k^2P_{\lambda}(B^{\lambda})$ for some constant $C$ depending on $n$ and $\lambda$, so in particular $D_{\lambda}(E_k)\to0$ as $k\to\infty$. Moreover, by definition of $\mu_{\lambda}^2(E_k)$ and by the contradiction assumption, we have
\begin{equation}\label{eq:zzMuZeroEk}
    0< D_{\lambda}(E_k)< \frac{1}{k^2}\mu_{\lambda}^2(E_k)\leq\frac{1}{k^2}\mu_{\lambda,0}^2(E_k).
\end{equation}

We apply \cref{prop:MinimizationProblemSelection}. There exist constants $\Lambda_{\ref{prop:MinimizationProblemSelection}} >2n$, $k_{\ref{prop:MinimizationProblemSelection}} \ge k_\lambda, R_{\ref{prop:MinimizationProblemSelection}} \ge 10$,
depending on $n,\lambda$ such that, for any $k \ge k_{\ref{prop:MinimizationProblemSelection}}$, the minimization problem
\begin{equation*}
    \inf \left\{
    \mathcal{F}_k(F) :=
        P_{\lambda}(F)+\Lambda_{\ref{prop:MinimizationProblemSelection}} \left||F|-|B^{\lambda}|\right|-\frac{1}{k}\mu_{\lambda,0}^2(F)+|\widetilde{b}(F)|^2
        \st 
        F\subset H, \,\, |F|<+\infty
    \right\}
\end{equation*}
has an open minimizer $F_k$ such that 
\begin{itemize}
    \item $F_k \subset B_{R_{\ref{prop:MinimizationProblemSelection}}}(0)$, 

     \item $|F_k|= |B^\lambda|$ for any $k\ge k_{\ref{prop:MinimizationProblemSelection}}$ sufficiently large,

     \item for any $k\ge k_{\ref{prop:MinimizationProblemSelection}}$ sufficiently large, there exists $u_k\in C^{1,1/2}(\S^{n-1}_+)$ such that $\overline{\partial F_k \cap H}= \{(w_{\lambda}(x)+u_k(x))x:x\in\S^{n-1}_+\}$, where $w_\lambda$ was defined in \cref{def:DistanzaC1}, and $u_k\to 0$ in $C^{1, \beta}(\S^{n-1}_+)$ for any $\beta \in (0,1/2)$.   
\end{itemize}

Testing minimality of $F_k$ with the competitor $E_k$, which satisfies $\widetilde{b} (E_k)=0$ and $|E_k|=|B^{\lambda}|$, and exploiting \cref{lem:ProblemPlambda+LambdaVolumi}, we get the chain of estimates
    \begin{equation}\label{eq:energy-comparison-FE}
    \begin{split}
    P_{\lambda}(F_k)-\frac{1}{k}\mu_{\lambda,0}^2(F_k)+|\widetilde{b} (F_k)|^2
        &=P_{\lambda}(F_k)+\Lambda_{\ref{prop:MinimizationProblemSelection}} \left||F_k|-|B^{\lambda}|\right|-\frac{1}{k}\mu_{\lambda,0}^2(F_k)+|\widetilde{b} (F_k)|^2\\
        &\leq P_{\lambda}(E_k)-\frac{1}{k} \mu_{\lambda,0}^2(E_k)\\
        &\overset{\eqref{eq:zzMuZeroEk}}{\leq} P_{\lambda}(B^{\lambda})+\frac{P_{\lambda} (B^{\lambda})}{k^2}  \mu_{\lambda,0}^2(E_k) -\frac{1}{k}\mu_{\lambda,0}^2(E_k)\\
        &\leq P_{\lambda}(F_k)+\Lambda_{\ref{prop:MinimizationProblemSelection}} \left||F_k|-|B^{\lambda}|\right|+\frac{P_{\lambda} (B^{\lambda})}{k^2}  \mu_{\lambda,0}^2(E_k)-\frac{1}{k} \mu_{\lambda,0}^2(E_k)
        \\
        &=  P_{\lambda}(F_k)+\frac{P_{\lambda} (B^{\lambda})}{k^2}  \mu_{\lambda,0}^2(E_k)-\frac{1}{k} \mu_{\lambda,0}^2(E_k)
    \end{split}
    \end{equation}
    for any $k\ge k_{\ref{prop:MinimizationProblemSelection}}$ sufficiently large. Observe that
    \begin{equation}\label{eq:zwq}
        \frac{P_{\lambda} (B^{\lambda})}{k^2}  \mu_{\lambda,0}^2(E_k)-\frac{1}{k} \mu_{\lambda,0}^2(E_k) 
        \le -\frac{1}{2k} \mu_{\lambda,0}^2(E_k),
    \end{equation}
    for any $k$ large enough. Hence \eqref{eq:energy-comparison-FE} implies that for every $k$ large enough
    \begin{equation}\label{eq:mu-bar-estimate}
        0<  \mu_{\lambda,0}^2(E_k) \leq 2\mu_{\lambda,0}^2(F_k),
        \qquad\qquad
        |\widetilde{b} (F_k)|^2\leq \frac{1}{k}\mu_{\lambda,0}^2(F_k).
    \end{equation}
Since we know that $|F_k|=|B^{\lambda}|$ for $k\ge k_{\ref{prop:MinimizationProblemSelection}}$ large enough, using \eqref{eq:zwq} we may rearrange the terms in the first and third lines of~\eqref{eq:energy-comparison-FE} to get
    \[
    D_{\lambda}(F_k)\leq \frac{\mu_{\lambda,0}^2(F_k)}{k\cdot P_{\lambda}(B^{\lambda})} .
    \]
For $k$ large enough, we have $\|u_k\|_{C^1(\S^{n-1}_+)}<\min\{\eps_{\ref{lem:OScillationBoundedByH1Norm}},\epsilon_{\ref{prop:FugledeCapillarity}}\}$, thus we can apply \cref{lem:OScillationBoundedByH1Norm} and \cref{prop:FugledeCapillarity} to continue the previous estimate as
    \begin{equation}\label{eq:final-quantitative}
    \begin{split}
        D_{\lambda}(F_k)&\leq \frac{\mu_{\lambda,0}^2(F_k)}{k P_{\lambda}(B^{\lambda})} 
        \le \frac{1}{k P_{\lambda}(B^{\lambda})} C_{\ref{lem:OScillationBoundedByH1Norm}} \|u_k \|^2_{H^1(\S^{n-1}_+)}
        \leq \frac{1}{k P_{\lambda}(B^{\lambda})}
        C_{\ref{lem:OScillationBoundedByH1Norm}}
        C_{\ref{prop:FugledeCapillarity}}
        \left(D_{\lambda}(F_k)+|\barycenter_H F_k|^2\right).
    \end{split}
    \end{equation}
    Since $|F_k|=|B^\lambda|$ and $\barycenter_H F_k = |B^\lambda|^{-1}\widetilde{b}(F_k)$ by \eqref{eq:AlternativeBarycenter}, then \eqref{eq:mu-bar-estimate} implies $|\barycenter_H F_k|^2 \le \tfrac{1}{k|B^\lambda|^2} \mu^2_{\lambda,0}(F_k)$. Hence, using \cref{lem:OScillationBoundedByH1Norm} and \cref{prop:FugledeCapillarity} again as above, we find
    \[
        |\barycenter_H F_k|^2 \le \frac{\mu_{\lambda,0}^2(F_k)}{k|B^\lambda|^2}
        \leq \frac{1}{k|B^\lambda|^2}C_{\ref{lem:OScillationBoundedByH1Norm}}
        C_{\ref{prop:FugledeCapillarity}} \left(D_{\lambda}(F_k)+|\barycenter_H F_k|^2\right).
    \]
    Hence, for $k$ large enough, there holds $|\barycenter_H F_k|^2\leq D_{\lambda}(F_k)$. Plugging this estimate in \eqref{eq:final-quantitative} we deduce
    \[
    D_\lambda(F_k) \le \frac12 D_\lambda(F_k),
    \]
    for any $k$ large enough. Hence $D_\lambda(F_k)=0$ for $k$ large, which, recalling that $|F_k|=|B^\lambda|$ and that $|\barycenter_H F_k|^2\leq D_{\lambda}(F_k)$, implies that $F_k$ exactly coincides with the standard bubble $B^\lambda$ with zero barycenter. In particular, $\mu^2_{\lambda,0}(F_k)=0$ for large $k$. This contradicts \eqref{eq:mu-bar-estimate}, which guarantees that $\mu_{\lambda,0}^2(F_k)>0$ for every $k$ large enough.
\end{proof}

\subsection{Strong quantitative inequalities with barycentric asymmetry in capillarity problems}\label{sec:BarycentricCapillarity}

As pointed out in the introduction, the proof of \cref{thm:capillarityStrongQuantitative} provides, in fact, the following more explicit and stronger inequality.

\begin{corollary}[Barycentric strong quantitative inequality in capillarity problems 1]\label{cor:tilde-bar-quantitative}
    There exists a constant $C_{\ref{cor:tilde-bar-quantitative}}>0$ depending only on $n$ and $\lambda$ such that, for every $E\subset H$ of finite perimeter with $|E|=|B^{\lambda}|$, and for any point $x^*\in\partial H$ such that $\widetilde b(E-x^*)=0$ there holds
    \[
        \int_{\partial^*E\cap H^+} \left|\nu_E(x)-\frac{x-(x^*,-\lambda)}{|x-(x^*,-\lambda)|}\right|^2\de\hausdorff^{n-1}_x\leq C_{\ref{cor:tilde-bar-quantitative}} D_{\lambda}(E).
    \]
\end{corollary}

\cref{cor:tilde-bar-quantitative} follows by arguing by contradiction exactly as in the proof of \cref{thm:capillarityStrongQuantitative}. Indeed, one starts by assuming by contradiction that there is a sequence of sets $U_k\subset H$ with $|U_k|=|B^{\lambda}|$, such that $k^2 D_{\lambda}(U_k)<\int_{\partial^*U_k\cap H^+} \left|\nu_{U_k}(x)-\frac{x-(x^*_k,-\lambda)}{|x-(x^*_k,-\lambda)|}\right|^2$ for some $x^*_k$ such that $\widetilde{b}(U_k- x^*_k)=0$. Taking the translated sets $E_k:=U_k- x^*_k$, then $k^2 D_\lambda(E_k)=k^2 D_{\lambda}(U_k) <\mu^2_{\lambda,0}(E_k)$, and $E_k$ is the starting contradicting sequence in the proof of \cref{thm:capillarityStrongQuantitative}.\\
However, $x^*$ in \cref{cor:tilde-bar-quantitative} might differ from $\frac{1}{|E|}\widetilde b(E)$ since the function $\psi$ in the definition of $\widetilde b$ is not linear.

Notice that we can take $x^*=\barycenter_H E$ in \cref{cor:tilde-bar-quantitative} for every set $E$ with $\diam_H E\leq 100$ thanks to~\eqref{eq:AlternativeBarycenter}, where $\diam_H E$ is the diameter of the orthogonal projection of $E$ onto $\partial H$. More generally, given $d>\diam_H B^{\lambda}$ an upper bound of the diameter of the considered sets, taking the function $\psi_d \coloneqq d\psi(\cdot/d)$ in place of $\psi$, and running the same proof presented above, one obtains the following barycentric quantitative inequality for the oscillation asymmetry.

\begin{corollary}[Barycentric strong quantitative inequality in capillarity problems 2]\label{cor:bar-quantitative}
    Let $d>\diam_H B^{\lambda}$ be given. Then there exists a constant $C_{\ref{cor:bar-quantitative}}>0$ depending only on $n$, $\lambda$ and $d$ such that, for every $E\subset H$ of finite perimeter with $|E|=|B^{\lambda}|$ and $\diam_H E\leq d$ there holds
    \begin{equation}\label{eq:bar-quantitative}
        \int_{\partial^*E\cap H^+} \left|\nu_E(x)-\frac{x-(\barycenter_H E,-\lambda)}{|x-(\barycenter_H E,-\lambda)|}\right|^2\de\hausdorff^{n-1}_x\leq C_{\ref{cor:bar-quantitative}} D_{\lambda}(E).
    \end{equation}
\end{corollary}
Our approach does not provide an explicit control of the constant $C_{\ref{cor:bar-quantitative}}$ in terms of the diameter of the sets. However, we remark that the dependence of the constant $C_{\ref{cor:bar-quantitative}}$ on the diameter is necessary in dimension $n\geq3$. Indeed, as analogously pointed out in \cite{GambicchiaPratelli2025}, one can take a sequence of sets $E_j = B^\lambda(|B^\lambda|-\eps_j, x_j) \cup B^\lambda(\eps_j, y_j)$, with $\eps_j\to0$ and suitable mutually diverging points $x_j, y_j \in \partial H$ so that $\barycenter_H E_j =0$, $D_\lambda(E_j) \to 0$, but the left hand side of \eqref{eq:bar-quantitative} with $E=E_j$ tends to a positive constant as $j\to\infty$. 

\medskip

Additionally, we remark that our method can be applied to obtain the quantitative isoperimetric inequality with barycentric Fraenkel asymmetry as follows.

\begin{corollary}[Barycentric quantitative inequality in capillarity problems]\label{cor:barycentric-Fraenkel}
    There exists a constant $C_{\ref{cor:barycentric-Fraenkel}}>0$ depending only on $n$ and $\lambda$ such that, for every $E\subset H$ of finite perimeter with $|E|=|B^{\lambda}|$, and for any point $x^*\in\partial H$ such that $\widetilde b(E-x^*)=0$, we have the following inequality:
    \[
        |E\symmdiff (B^{\lambda}+x^*)|^2\leq C_{\ref{cor:barycentric-Fraenkel}} D_{\lambda}(E).
    \]
    If $\diam_H E\le d <+\infty$, there exists a constant $C_{\ref{cor:barycentric-Fraenkel}}'>0$ depending only on $n$, $\lambda$ and $d$ such that
    \[
        |E\symmdiff (B^{\lambda}+\barycenter_H E)|^2\leq C_{\ref{cor:barycentric-Fraenkel}}' D_{\lambda}(E).
    \]
\end{corollary}

\begin{proof}
    We provide some details about the necessary adaptations to obtain the first inequality since here we use a different notion of asymmetry. 
    For a set $E\subset H$ with $\overline{\partial E\cap H^+} = \{(w_{\lambda}(x)+u(x))x: x\in\S^{n-1}_+\}$, and $\|u\|_{C^1}\ll1$, as in the proof of \cref{prop:FugledeCapillarity} one estimates $|E\symmdiff B^{\lambda}|\lesssim \|u\|_{L^1}\lesssim \|u\|_{L^2}\leq \|u\|_{H^1}$. This replaces \cref{lem:OScillationBoundedByH1Norm}. Then, we can run the selection principle in \cref{prop:MinimizationProblemSelection} and \cref{thm:capillarityStrongQuantitative} replacing $\mu_{\lambda,0}(F)$ with
    \[
        A_{\lambda,0}(F) := |F\symmdiff B^{\lambda}|.
    \]
    Since the Fraenkel asymmetry corresponds to a ``volume term'',  the application of the regularity theory is just simpler, and the proof follows.
\end{proof}

\subsection{Sharpness of \cref{thm:capillarityStrongQuantitative}.}\label{subsec:sharp}

We prove that exponents in the inequality in \cref{thm:capillarityStrongQuantitative} are sharp. 
Let $\varphi:\S^{n-1}_+\to \R$ be a smooth function such  that $\varphi\not\equiv0$,
\begin{equation}\label{eq:ConditionSharpnessVolume}
    \int_{\S^{n-1}_+} w_\lambda^{n-1} \varphi =0,
\end{equation}
and $\varphi$ is ``rotationally symmetric'', in the sense that we impose $\varphi(x) = f(\Braket{x,e_n})$ for some $f \in C^\infty_c(0,1)$.
Consider the one-parameter family of sets $E_t:= (|B^\lambda|/|\widetilde{E}_t| )^{\frac1n} \widetilde{E}_t$ where $\overline{\partial \widetilde{E}_t \cap H^+} := \{(w_\lambda + t \varphi) x \st x \in \S^{n-1}_+\}$. We also write $\overline{\partial E_t \cap H^+}=: \{(w_\lambda + h_t) x \st x \in \S^{n-1}_+\}$ for a smooth one-parameter family of smooth functions $h_t:\S^{n-1}_+\to \R$. By construction, we have $(|B^\lambda|/|\widetilde{E}_t| )^{\frac1n} (w_\lambda + t \varphi) x = (w_\lambda + h_t) x $. Differentiating the previous relation with respect to $t$ and evaluating at $t=0$, using \eqref{eq:ConditionSharpnessVolume}, one sees that $\partial_t h_t \big|_{t=0} = \varphi$. Hence $h_t$ has the expansion
\[
h_t
= t \,\varphi + \frac12 t^2 \Phi +O(t^3)
\quad
\text{as $t\to0$,}
\]
for some smooth function $ \Phi :\S^{n-1}_+ \to \R$.

Summing up, we constructed a one-parameter family of sets $E_t$, for $t\in(-\eps,\eps)$ for some $\eps>0$ small, such that


\begin{enumerate}
    \item[a)] $|E_t|=|B^\lambda|$ for any $t \in (-\eps,\eps)$,
    \item[b)] $E_t$ is rotationally symmetric with respect to the $x_n$-axis for any $t \in (-\eps,\eps)$,
    \item[c)] $E_t\neq B^\lambda$ for any for any $t \in (-\eps,\eps)\setminus\{0\}$, and $E_0=B^\lambda$,
    \item[d)] $\overline{\partial E_t \cap H^+}= \{(w_\lambda + h_t) x \st x \in \S^{n-1}_+\}$.
\end{enumerate}

We want to compute the expansion of the product $\Braket{\nu_{E_t} , \nu_{B^\lambda}}$ as $t\to0$. Clearly $\Braket{\nu_{E_t} |_{t=0}, \nu_{B^\lambda}}=1$ and $0 = \partial_t |\nu_{E_t}|^2  |_{t=0} = 2 \Braket{\partial_t  \nu_{E_t} |_{t=0}, \nu_{B^\lambda}}=0$. Moreover, a direct computation using the formula for the unit normal in \cref{lemma:formulas} shows that
\begin{equation}\label{eq:Sharp1}
\begin{split}
\Braket{\partial_t^2 \nu_{E_t} \Big|_{t=0}, \nu_{B^\lambda}}
    &= 
    \frac{1}{(w_\lambda^2 + |\nabla w_\lambda|^2)^2} \left(
    (\varphi w_\lambda + \Braket{\nabla \varphi, \nabla w_\lambda})^2 - (\varphi^2 + |\nabla \varphi|^2) (w_\lambda^2 + |\nabla w_\lambda|^2)
    \right).
\end{split}
\end{equation}
Therefore, denoting by $C$ a positive constant independent of $t$ that may change from line to line, for $t$ sufficiently small the first two lines in \eqref{eq:zzz0} imply
\begin{equation}\label{eq:Sharp2}
\begin{split}
    \mu_{\lambda,0}^2(E_t) &\ge C \int_{\S^{n-1}_+} 1- \Braket{\nu_{E_t}((w_\lambda + h_t) x) , \nu_{B^\lambda}(w_\lambda  x)} \\
    &\overset{\eqref{eq:Sharp1}}{\ge} t^2 \, C
    \int_{\S^{n-1}_+} \left(
     (\varphi^2 + |\nabla \varphi|^2) (w_\lambda^2 + |\nabla w_\lambda|^2)
     - (\varphi w_\lambda + \Braket{\nabla \varphi, \nabla w_\lambda})^2
    \right) \\
    &\ge C \, t^2,
\end{split}
\end{equation}
where in the last inequality we used that $\varphi\not\equiv0$ and we used Cauchy--Schwarz inequality, observing that $\varphi$ is not a multiple of $w_\lambda$ because of \eqref{eq:ConditionSharpnessVolume}. 

On the other hand, going back to the proof of \cref{prop:EspansionePlambda}, the expansion in \eqref{eq:zzExpansion1} implies
\begin{equation}\label{eq:Sharp3}
\begin{split}
D_\lambda(E_t) &\le C \left( t^2 + \int_{\S^{n-1}_+} (n-1) w_\lambda^{n-1} h_t \right) 
\\
&
\overset{\eqref{eq:ConditionSharpnessVolume}}{\le} C \left( t^2  + \int_{\S^{n-1}_+} (n-1) w_\lambda^{n-1} \left(\frac12 t^2 \Phi + O(t^3) \right)\right) \le C \, t^2,
\end{split}
\end{equation}
for any $t$ sufficiently small.

Finally we claim that, being $E_t$ rotationally symmetric, implies that
\begin{equation}\label{eq:Sharp4}
\mu_{\lambda,0}^2(E_t) \le 20 \mu_\lambda^2(E_t).
\end{equation}
Indeed, fix $t$ and suppose $x'\in\partial H \neq\{0\}$ is such that $ \tfrac12 \int_{\partial^* E_t \cap H^+} \Big| \nu_{E_t} - \tfrac{x-(x',-\lambda)}{|x-(x',-\lambda)|} \Big|^2 \le 2  \mu^2_\lambda(E_t)$. 
By symmetry, the same holds with $x'$ replaced by $-x'$. 
We shall now employ the geometric inequality
\begin{equation}\label{eq:ElementarySharpness}
\bigg| 
\frac{x-(0,-\lambda)}{|x-(0,-\lambda)|} - \frac{x-(x',-\lambda)}{|x-(x',-\lambda)|} \bigg|^2 
\le
\bigg| 
\frac{x-(-x',-\lambda)}{|x-(-x',-\lambda)|} - \frac{x-(x',-\lambda)}{|x-(x',-\lambda)|} \bigg|^2 ,
\end{equation}
for any $x \in H^+$, that will be justified below. Hence we conclude that
\begin{equation}
\begin{split}
    \mu_{\lambda,0}^2(E_t)
    &= \frac12  \int_{\partial^* E_t \cap H^+} \bigg| \nu_{E_t} - \frac{x-(0,-\lambda)}{|x-(0,-\lambda)|} \bigg|^2 
    \\&
    \le  \int_{\partial^* E_t \cap H^+} \bigg| \nu_{E_t} - \frac{x-(x',-\lambda)}{|x-(x',-\lambda)|} \bigg|^2 
    +
    \bigg| 
    \frac{x-(0,-\lambda)}{|x-(0,-\lambda)|} - \frac{x-(x',-\lambda)}{|x-(x',-\lambda)|} \bigg|^2 
    \\&
    \overset{\eqref{eq:ElementarySharpness}}{\le}
     \int_{\partial^* E_t \cap H^+} \bigg| \nu_{E_t} - \frac{x-(x',-\lambda)}{|x-(x',-\lambda)|} \bigg|^2 
    + \bigg| 
    \frac{x-(-x',-\lambda)}{|x-(-x',-\lambda)|} - \frac{x-(x',-\lambda)}{|x-(x',-\lambda)|} \bigg|^2 
    \\&
    \le 
    \int_{\partial^* E_t \cap H^+} 3 \bigg| \nu_{E_t} - \frac{x-(x',-\lambda)}{|x-(x',-\lambda)|} \bigg|^2 
    + 
    2\bigg| \nu_{E_t} - \frac{x-(-x',-\lambda)}{|x-(-x',-\lambda)|} \bigg|^2 
    \\&
    \le 20 \mu^2_\lambda(E_t),
\end{split}
\end{equation}
proving \eqref{eq:Sharp4}. Putting together \eqref{eq:Sharp2}, \eqref{eq:Sharp3} and \eqref{eq:Sharp4} we conclude that $D_\lambda(E_t) \le C \mu^2_\lambda(E_t)$ for any $t$ small enough, proving the claimed sharpness.

It remains to justify \eqref{eq:ElementarySharpness}. Expanding the squares in \eqref{eq:ElementarySharpness} and changing variables into $z:= x-(x',-\lambda)$, we see that \eqref{eq:ElementarySharpness} is equivalent to
\begin{equation}\label{eq:ElementarySharpness2}
\bigg\langle \frac{z+(x',0)}{|z+(x',0)|} , z \bigg\rangle
\ge
\bigg\langle \frac{z+2(x',0)}{|z+2(x',0)|} , z \bigg\rangle,
\end{equation}
for any $z \in \R^n$. It is sufficient to suppose that $z,(x',0)$ are linearly independent. Hence, considering the $2$-dimensional subspace generated by $z,(x',0)$, in order to prove \eqref{eq:ElementarySharpness2} it is equivalent to prove that if $v,w \in \R^2$ are independent, then
\begin{equation}\label{eq:ElementarySharpness3}
\bigg\langle \frac{v+w}{|v+w|} , v \bigg\rangle
\ge
\bigg\langle \frac{v+2w}{|v+2w|} , v \bigg\rangle.
\end{equation}
Indeed, $\big\langle \tfrac{v+w}{|v+w|} , v \big\rangle = |v| \cos \theta_1$ where $\theta_1 \in (0,\pi)$ is the angle between $v$ and $\tfrac{v+w}{|v+w|}$, and the same can be said for the second scalar product for a corresponding angle $\theta_2\in (0,\pi)$. Clearly $\theta_1\le\theta_2$, hence $\cos \theta_1 \ge \cos \theta_2$, and \eqref{eq:ElementarySharpness3} follows.

\section{Strong quantitative isoperimetric inequalities in convex cones}\label{sec:cones}

\noindent\textbf{List of new symbols related to this section.}

\begin{itemize}

\item A closed convex cone $\mathcal{C}\subset \R^n$ with non-empty interior will be always written as $\mathcal{C}= \R^m \times \widetilde{\mathcal{C}}$, where $\widetilde{\mathcal{C}}$ does not contain lines. We will always assume that $0$ is a tip of $\mathcal{C}$.

\item $\interior{\mathcal{C}}$ denotes the interior of a cone $\mathcal{C}$.

\item $P_\mathcal{C}(E):= \hausdorff^{n-1}(\partial^*E \setminus \partial \mathcal{C}) = P(E,\interior{\mathcal{C}})$ denotes the relative perimeter in a cone $\mathcal{C}$.

\item $K_{\mathcal{C}}:= B_1(0) \cap \mathcal{C}$.

\item $\alpha_{\mathcal{C}}(E)$, $D_{\mathcal{C}}(E)$, $\mu_{\mathcal{C}}(E)$, $\mu_{\mathcal{C},0}(E)$ denote Fraenkel asymmetry, deficit, and oscillation asymmetries of a set $E$ in a cone $\mathcal{C}$, defined in \cref{def:AsymmetriesDeficitInCones}.

\item $\Sigma_{\mathcal{C}} := \mathcal{C} \cap \S^{n-1}$ denotes the cross section of a cone $\mathcal{C}$.

\item $\barycenter_{\R^m} E$ denotes the barycenter of a set $E$ with respect to the factor $\R^m$ in a cone $\mathcal{C}= \R^m \times \widetilde{\mathcal{C}}$, defined in \eqref{eq:DefBaricentroConi}. 

\item $ \widetilde b_{\R^m}(F) = \int_F (\psi(x_1),\ldots,\psi(x_m))\de x$, for any measurable set $F$ in a cone $\mathcal{C}=\R^m\times \widetilde{\mathcal{C}}$ with $|F|<+\infty$, defined in \eqref{eq:DefBaricentroTILDEconi}.

\item For $m \in\{0,\ldots, n-1\}$ fixed, spaces $\R^m, \R^{n-m}$ are identified with subspaces of $\R^n$ by $\R^n= \R^m \times \R^{n-m}$. Vectors $x' \in \R^m$, $\widetilde{x} \in \R^{n-m}$ are thus identified with $(x',0)\in \R^n$, $(0,\widetilde{x})\in \R^n$, respectively.

\end{itemize}


We aim to obtain a sharp quantitative isoperimetric inequality in convex cones in the strong form, similarly to what we did in the previous section. We first collect the notions of asymmetry and deficit that we are going to use.

\begin{definition}\label{def:AsymmetriesDeficitInCones}
    Let $\mathcal{C}\subset \R^n$ be a closed convex cone with non-empty interior $\interior{\mathcal{C}}$. Write $\mathcal{C}=\R^m\times \widetilde{\mathcal{C}}$ where $\widetilde{\mathcal{C}}$ does not contain lines. Let $K_{\mathcal{C}}:= B_1(0) \cap \mathcal{C}$ and let $E \subset \mathcal{C}$ be a measurable set with $|E|\in(0,+\infty)$. For $s>0$ such that $|E|=|sK_{\mathcal{C}}|$, we define the Fraenkel asymmetry
    \[
    \alpha_{\mathcal{C}}(E) := \inf \left\{
    \frac{|E \Delta ((x',0)+sK_{\mathcal{C}})|}{|E|} \st x' \in \R^m,
    \right\}
    \]
    and the isoperimetric deficit
    \[
    D_{\mathcal{C}}(E) := \frac{P_{\mathcal{C}}(E) - P_{\mathcal{C}}(sK_{\mathcal{C}})}{P_{\mathcal{C}}(sK_{\mathcal{C}})} = \frac{P_{\mathcal{C}}(E)}{n|K_{\mathcal{C}}|^{\frac1n} |E|^{\frac{n-1}{n}}} -1,
    \]
    where $P_{\mathcal{C}}(E):= P(E,\interior{\mathcal{C}})= \hausdorff^{n-1}(\partial^*E \setminus\partial \mathcal{C})$ denotes the relative perimeter in $\mathcal{C}$.

    The strong isoperimetric inequality in convex cones will bound from above the following the natural notion of oscillation asymmetry
    \[
        \mu_{\mathcal{C}}^2(E) := \inf\left\{
        \frac{1}{2 s_E^{n-1}} \int_{\partial^* E \cap\interior{\mathcal{C}}} \left| \nu_E(x) - \frac{x-(x',0)}{|x-(x',0)|} \right|^2 \de \hausdorff^{n-1} (x) \st x' \in \R^m \right\},
    \]
    where $\nu_E$ as usual denotes the generalized outer unit normal to $E$, and $s_E=(|E|/|K_{\mathcal{C}}|)^{1/n}$.
    Similarly to what we did in \cref{sec:capillarity}, for technical reasons it will be convenient to introduce an asymmetry index that is not translation, nor it is scaling invariant:
    \[
        \mu_{\mathcal{C},0}^2(E) := \frac{1}{2} \int_{\partial^* E \cap\interior{\mathcal{C}}} \left| \nu_E(x) - \frac{x}{|x|} \right|^2 \de \hausdorff^{n-1} (x).
    \]
\end{definition}

\begin{lemma}\label{lemma:mu-C-alternative}
    Let $\mathcal{C}\subset \R^n$ be a closed cone with non-empty interior. Then, for any measurable set $E\subset \mathcal{C}$ with finite volume there holds
    \[
        \mu_{\mathcal{C},0}^2(E) = P_{\mathcal{C}}(E)-\int_E\frac{n-1}{|x|}\de x.
    \]
\end{lemma}

\begin{proof}
    The proof is a simplified version of \cref{lem:RiscritturaDiMuLambdaZero}. It is sufficient to apply the divergence theorem to the vector field $X=\frac{x}{|x|}$, observing that $\Braket{X ,\nu_E}=0$ on $\partial^*E \cap \partial^* \mathcal{C}$.
\end{proof}

    As mentioned in the introduction, in \cite{FigalliIndrei} the authors proved the following sharp quantitative isoperimetric inequality for the isoperimetric problem in convex cones.    
    \begin{theorem}[{\cite[Theorem 1.2]{FigalliIndrei}}]\label{thm:QuantitativeConesFraenkel}
    Let $\mathcal{C}\subset \R^n$ be a closed convex cone with non-empty interior. Then there exists $C_{\ref{thm:QuantitativeConesFraenkel}}=C_{\ref{thm:QuantitativeConesFraenkel}}(\mathcal{C})>0$ such that the following holds.\\
    Let $E\subset \mathcal{C}$ be a measurable set with $|E| \in (0,+\infty)$. Then
    \begin{equation*}
        \alpha^2_{\mathcal{C}}(E) \le C_{\ref{thm:QuantitativeConesFraenkel}} D_{\mathcal{C}}(E).
    \end{equation*}
    \end{theorem}
    It is known from \cite{FigalliIndrei} that $C_{\ref{thm:QuantitativeConesFraenkel}}=C_{\ref{thm:QuantitativeConesFraenkel}}(\mathcal{C})$ is a constant depending on the geometry of $\mathcal{C}$, and a careful reading of \cite{FigalliIndrei} allows to detect such a dependence. For this reason, in the rest of the work we will keep track of the dependence of newly defined constants on $C_{\ref{thm:QuantitativeConesFraenkel}}(\mathcal{C})$, in place of stating generic dependencies on $\mathcal{C}$.

\begin{definition}\label{def:DistanzaC1-cone}
    Let $\mathcal{C}\subset\R^n$ be a closed convex cone with non-empty interior. Let $k \in \N$ with $k\ge 1$, $\beta \in (0,1)$.
    Let $E \subset \mathcal{C}$ be a bounded open set. We say that the boundary of $E$ is a $C^{k,\beta}$-graph over $\Sigma_{\mathcal{C}}$ (or, more precisely, that $\overline{\partial E\cap \interior{\mathcal{C}}}$ is a $C^{k,\beta}$-graph over $\Sigma_{\mathcal{C}}$) if $\partial E$ is Lipschitz and if there exists a $C^{k,\beta}$-function $f:\Sigma_{\mathcal{C}}\to (0,+\infty)$ such that
    \[
     \overline{\partial E\cap \interior{\mathcal{C}}}= \{ f(x) \, x \st x \in \Sigma_{\mathcal{C}}\}.
    \]
    If the boundaries of two sets $E,F$ are $C^{k,\beta}$-graphs over $\Sigma_{\mathcal{C}}$ with respect to some functions $f,g$, we define
    \[
        \d_{C^{k,\beta}}(E,F):=\|f-g\|_{C^{k,\beta}(\Sigma_{\mathcal{C}})}.
    \]
    If $\{E_i\}_{i \in \N}, E$ have boundary given by $C^{k,\beta}$-graphs over $\Sigma_{\mathcal{C}}$, we say that $\overline{\partial E_i \cap \interior{\mathcal{C}}}$ converge to $\overline{\partial E \cap \interior{\mathcal{C}}}$ in $C^{k,\beta}$ (or, simply, that $E_i$ converge to a $E$ in $C^{k,\beta}$) if $\d_{C^{k,\beta}}(E_i,E)\to 0$ as $i\to\infty$.
\end{definition}

\begin{definition}\label{def:piecewise-C2}
    Let $\mathcal{C}\subset \R^n$ be a closed convex cone. We say that it is piecewise $C^2$ if $\partial B_1\cap \mathcal{C}$ is a piecewise $C^2$ domain in $\S^{n-1}$.
\end{definition}

\subsection{Proof of the strong quantitative isoperimetric inequality in convex cones}

Let $\mathcal{C}=\R^m\times \widetilde{\mathcal{C}} \subset  \R^n$, where $\widetilde{\mathcal{C}}$ does not contain lines, be a closed convex cone with non-empty interior. Denoting any point $x \in \R^n$ as $x=(x',\widetilde{x})$ with $x' \in \R^m$ and $\widetilde{x} \in \R^{n-m}$, for any bounded measurable set $E\subset \mathcal{C}$ we define
\begin{equation}\label{eq:DefBaricentroConi}
    \barycenter_{\R^m} E := \frac{1}{|E|}\int_E x' \de x.
\end{equation}

\begin{proposition}[Fuglede-type estimate in cones]\label{prop:FugledeConi}
%
Let $\mathcal{C}=\R^m\times \widetilde{\mathcal{C}} \subset  \R^n$ be a closed convex cone with non-empty interior, where $\widetilde{\mathcal{C}}$ does not contain lines. Suppose that $ |B_1 \cap \mathcal{C}|\ge v_0$. 
Then there exist $c_{\ref{prop:FugledeConi}}, C_{\ref{prop:FugledeConi}}, \epsilon_{\ref{prop:FugledeConi}}>0$ depending on $v_0, n, C_{\ref{thm:QuantitativeConesFraenkel}}(\mathcal{C})$ such that the following holds.

Let $E\subset \mathcal{C}$ be a bounded open set with Lipschitz boundary and volume $|E|=|K_{\mathcal{C}}|$ such that
\[
\overline{\partial E \setminus \partial \mathcal{C}} = \left\{ (1+u(x)) x \st x \in \Sigma_{\mathcal{C}}\right\},
\]
for a function $u :\Sigma_{\mathcal{C}}\to \R$ with $\|u\|_{C^1} \le \epsilon_{\ref{prop:FugledeConi}}$, where $\Sigma_{\mathcal{C}}:=  \S^{n-1} \cap \mathcal{C}$.
Then
\begin{equation*}
    \mu_{\mathcal{C}}^2(E) \le
    \mu_{\mathcal{C},0}^2(E) \le
    c_{\ref{prop:FugledeConi}} \| u\|^2_{H^1(\Sigma_{\mathcal{C}})} \le C_{\ref{prop:FugledeConi}} \left( D_{\mathcal{C}}(E) + |\barycenter_{\R^m}E|^2\right).
\end{equation*}
\end{proposition}

\begin{proof}
Inequality $\mu_{\mathcal{C}}^2(E) \le \mu_{\mathcal{C},0}^2(E)$ follows by definition.
Moreover, the computations contained in the proof \cref{lem:OScillationBoundedByH1Norm} can be readily adapted to the present case. Indeed, here the domain of integration is $\Sigma_{\mathcal{C}}$ (instead of $\S^{n-1}_+$), and $w_{\lambda}$ is replaced by the constant function $1$. Hence, for $\epsilon_{\ref{prop:FugledeConi}}$ small enough, one gets the constant $c_{\ref{prop:FugledeConi}}$ such that $\mu_{\mathcal{C},0}^2(E) \le c_{\ref{prop:FugledeConi}} \| u\|^2_{H^1(\Sigma_{\mathcal{C}})}$.\\
So, it remains to prove the bound of $\|u\|_{H^1}^2$ from above in terms of $D_{\mathcal{C}}(E)+|\barycenter_{\R^m} E|^2$. One can go through the expansion of the perimeter performed in \cref{prop:EspansionePlambda}, observing that all the computations leading to \eqref{eq:zzExpansion1} hold pointwise. Also, here the identity \eqref{eq:zzExpansion2}, which imposes the volume constraint, is replaced with
    \begin{equation}\label{eq:volumeCone}
        \int_{\Sigma_{\mathcal{C}}}\frac{1}{n}=|K_{\mathcal{C}}|=|E|=\int_{\Sigma_{\mathcal{C}}}\frac{(1+u)^n}{n} = \int_{\Sigma_{\mathcal{C}}}\frac{1}{n}+u+\frac{n-1}{2}u^2+O(u^3).
    \end{equation}
Thus $\int_{\Sigma_{\mathcal{C}}}u = -\frac{n-1}{2}\int_{\Sigma_{\mathcal{C}}}u^2+O(u^3)$. Hence using the (significantly easier) computations with $w_{\lambda}$ replaced by $1$ leading to \eqref{eq:zzExpansion1} we obtain that
    \[
        P_{\mathcal{C}}(E)-P_{\mathcal{C}}(K_{\mathcal{C}})\geq \frac{1}{2}\int_{\Sigma_{\mathcal{C}}}|\grad u|^2-(n-1)u^2+\|u\|_{C^1}\cdot O(\|u\|_{H^1}^2).
    \]
We point out that a similar expansion was obtained in \cite[Lemma~3.1]{BaerFigalli17} in a suitable different class of cones. In fact, the proof of \cite[Lemma~3.1]{BaerFigalli17} might be carried out in our setting as well.

We can now proceed as we did in \cref{prop:FugledeCapillarity}. First, using the quantitative isoperimetric inequality recalled in \cref{thm:QuantitativeConesFraenkel}, one obtains that $|E\symmdiff K_{\mathcal{C}}|\leq C(D_{\mathcal{C}}(E)^{\frac12}+|\barycenter_{\R^m}E|)$. Moreover, using a volume expansion similar to~\eqref{eq:volumeCone}, the previous estimate can be turned into $\|u\|_{L^1(\Sigma_{\mathcal{C}})}^2\leq C(D_{\mathcal{C}}(E)+|\barycenter_{\R^m}E|^2)$, which is analogous to \eqref{eq:L1-IsoperimetricDeficit}. Notice that, having applied \cref{thm:QuantitativeConesFraenkel}, the constants appearing in the last inequality only depends on $n, C_{\ref{thm:QuantitativeConesFraenkel}}(\mathcal{C})$. Finally, the conclusion is again obtained as in \cref{prop:FugledeCapillarity}, using the Nash-type inequality in \cref{lemma:NashIneq} with a parameter $\delta$ small enough. The resulting constant will depend on $v_0, n, C_{\ref{thm:QuantitativeConesFraenkel}}(\mathcal{C})$. Indeed the constant in the Nash-type inequality depends only on the dimension and on a lower bound on $\hausdorff^{n-1}(\Sigma_{\mathcal{C}})$, which is uniformly bounded from below away from zero in terms of $n,v_0$; also, the choice of the parameter $\delta$ to plug into the Nash-type inequality depends only on the constants appearing in the proof.
\end{proof}

The next lemma is an easy exercise and it is analogous to \cref{lem:ProblemPlambda+LambdaVolumi}.

\begin{lemma}\label{lem:ProblemCone+LambdaVolumi}
Let $\mathcal{C}=\R^m\times \widetilde {\mathcal C}\subset \R^n$ be a closed convex cone with non-empty interior and such that $\widetilde{\mathcal{C}}\subset \R^{n-m}$ does not contain lines. Then for any $\bar\Lambda>n$ there holds
\begin{equation*}\label{eq:ProblemCone+LambdaVolumi}
    \inf\left\{
        P_{\mathcal{C}}(E) + \bar\Lambda \left| |E| - |K_{\mathcal{C}}| \right| \st E \subset \mathcal{C}, \,\, |E|<+\infty
    \right\} = P_{\mathcal{C}}(K_{\mathcal{C}}),
\end{equation*}
and the infimum is only achieved by  $K_{\mathcal{C}} + x'$ for any $x'\in\R^m$.
\end{lemma}


We now collect some considerations analogous to those in \eqref{eq:PreConvention}--\eqref{eq:Convention}. For any $k\geq4$ and for any set of finite perimeter $F\subset \mathcal{C}$ with $|F|<+\infty$ we have that
\begin{equation}\label{eq:cone-PreConvention}
    P_{\mathcal{C}}(F)-\frac{1}{k}\mu_{\mathcal{C},0}^2(F)\geq P_{\mathcal{C}}(F)-\frac{2}{k}P_{\mathcal{C}}(F)\geq \frac{1}{2}P_{\mathcal{C}}(F),
\end{equation}
and thus we set $P_{\mathcal{C}}(F)-\frac{1}{k}\mu_{\mathcal{C},0}^2(F)=+\infty$ whenever $P_{\mathcal{C}}(F)=+\infty$ and $k\geq 4$.

Moreover, similarly to what we did in~\eqref{eq:AlternativeBarycenter}, it is convenient to also introduce the following related quantity that is well-defined for any set $F\subset \mathcal{C}$ with finite measure:
\begin{equation}\label{eq:DefBaricentroTILDEconi}
    \widetilde b_{\R^m}(F) = \int_F (\psi(x_1),\ldots,\psi(x_m))\de x,
\end{equation}
where $\psi$ was defined in \eqref{eq:DefPSI}.

\begin{proposition}\label{prop:selection-cones}
    Let $\mathcal{C}=\R^m\times\widetilde{\mathcal{C}}\subset\{x_n\geq0\}$ be a piecewise $C^2$ closed convex  cone, where $\tilde{\mathcal{C}}$ does not contain lines. Suppose $|B_1 \cap \mathcal{C}| \ge v_0$, for some $v_0>0$. 
    Then there exist $\Lambda_{\ref{prop:selection-cones}}>2n$, $R_{\ref{prop:selection-cones}}\geq 10$, $\Lambda_{\ref{prop:selection-cones}}'>0$, $ r'_{\ref{prop:selection-cones}}\in(0,1)$,
    depending on $n, v_0, C_{\ref{thm:QuantitativeConesFraenkel}}(\mathcal{C})$,
    and
    $k_{\ref{prop:selection-cones}}>4$, depending on $n, v_0, C_{\ref{thm:QuantitativeConesFraenkel}}(\mathcal{C}), \delta_{\ref{thm:almost-regularity-cones}}(\mathcal{C}, \Lambda_{\ref{prop:selection-cones}}', r'_{\ref{prop:selection-cones}} , \epsilon_{\ref{prop:FugledeConi}})$,
    such that the following holds. For any $k\geq k_{\ref{prop:selection-cones}}$, $k\in\N$, the minimization problem
    \begin{equation*}
        \inf \left\{
    \mathcal{G}_k(F) :=
        P_{\mathcal{C}}(F)+\Lambda_{\ref{prop:selection-cones}} \left||F|-|K_{\mathcal{C}}|\right|-\frac{1}{k}\mu_{\mathcal{C},0}^2(F)+|\widetilde{b}_{\R^m}(F)|^2
        \st 
        F\subset \mathcal{C}, \,\, |F|<+\infty
    \right\}
    \end{equation*}
    has an open minimizer $F_k$ such that
    \begin{itemize}
    \item $F_k$ is a $(\Lambda_{\ref{prop:selection-cones}}', r'_{\ref{prop:selection-cones}})$-almost minimizer of the perimeter in $\mathcal{C}$,
    
    \item $F_k \subset B_{R_{\ref{prop:selection-cones}}}(0)\cap \mathcal{C}$,

     \item $|F_k|= |K_{\mathcal{C}}|$ for any $k\ge k_{\ref{prop:selection-cones}}$,

     \item $\overline{\partial F_k \cap \interior{\mathcal{C}}}$ is a $C^{1, \beta_{\ref{thm:almost-regularity-cones}}(\mathcal{C})}$-graph over $\Sigma_{\mathcal{C}}$ for any $k\ge k_{\ref{prop:selection-cones}}$, and $\overline{\partial F_k \cap \interior{\mathcal{C}}} \to \overline{\partial K_{\mathcal{C}} \cap \interior{\mathcal{C}}}$ in $C^{1,\beta}$, for any $\beta \in (0,\beta_{\ref{thm:almost-regularity-cones}}(\mathcal{C}))$, as $k\to \infty$, in the sense of \cref{def:DistanzaC1-cone}.\\
     Additionally, we have that $\|u_k\|_{C^{1, \beta_{\ref{thm:almost-regularity-cones}}(\mathcal{C})}(\Sigma_{\mathcal{C}})}<\epsilon_{\ref{prop:FugledeConi}}$ for every $k\geq k_{\ref{prop:selection-cones}}$.    
\end{itemize}
\end{proposition}

\begin{proof}
    The proof is very similar to that of \cref{prop:MinimizationProblemSelection} (which actually presents further complications), thus we only emphasize here the main differences.
    Let $\Lambda> 2n$ and let $k \in \N$ with $k>4$. We will fix $\Lambda=\Lambda_{\ref{prop:selection-cones}}$ at the end of the proof.
    As in \cref{prop:MinimizationProblemSelection}, we can use Ekeland's variational principle (recalled in \cref{thm:Ekeland}) in the metric space $\mathbb X = \{F\subset \mathcal{C}:|F|<+\infty\}$ endowed with the $L^1$ distance. Therefore, one finds a sequence of sets $E_j^k\subset\mathcal{C}$ such that
    \begin{equation}\label{eq:ekeland-set-cone}
    \begin{split}
        \mathcal{G}_k(E_j^k)&\le \min\left\{ \inf \left\{
    \mathcal{G}_k(F) 
        \st 
        F\subset \mathcal{C}, \,\, |F|<+\infty
    \right\} 
    + \frac1j
    , \mathcal{G}_k(K_{\mathcal{C}}) \right\}, \\
    \mathcal{G}_k(E_j^k)  &\le \mathcal{G}_k(G) + \frac1j |G \Delta E_j^k|, 
    \end{split}
    \end{equation}
    for any $G\subset\mathcal{C}$ with finite measure. 
    First we show that $E_j^k$ is a $(K_{\ref{prop:selection-cones}},r_{\ref{prop:selection-cones}})$-quasi-minimizer of the perimeter in $\mathcal{C}$, where $K_{\ref{prop:selection-cones}}>0, r_{\ref{prop:selection-cones}} \in(0,1)$ depend only on $n$ and $v_0$.
    Exactly as in \eqref{eq:MassMinimizingSeqCloseToBlambda}, one can see that $\tfrac{|E_j^k|}{|K_{\mathcal{C}}|}\in\left(\frac{1}{2}, \frac{3}{2}\right)$. 
    Let now $x_0\in\mathcal{C}$, and $G\subset\mathcal{C}$ be a set of finite perimeter such that $G\symmdiff E_j^k\Subset B_r(x_0)$ for some $r\le r_{\ref{prop:selection-cones}}$ which we should take small enough. 
    Condition~\eqref{eq:ekeland-set-cone} gives the following estimate for any $j\geq 1$
    \begin{equation*}
    \begin{split}
        P_{\mathcal{C}}(E_j^k)&\leq P_{\mathcal{C}}(G)+\left(\Lambda+1\right)|G\symmdiff E_j^k|+\frac{1}{k}\left(\mu_{\mathcal{C},0}^2(E_j^k)-\mu_{\mathcal{C},0}^2(G)\right)+|\widetilde b_{\R^m}(G)|^2-|\widetilde b_{\R^m}(E_j^k)|^2\\
        &\leq P_{\mathcal{C}}(G)+\left(\Lambda+1\right)|G\symmdiff E_j^k|+\frac{1}{2k}\bigg(\int_{\partial^*E_j^k \cap \interior{\mathcal{C}}\cap B_r(x_0)}\left|\nu_{E_j^k}(x)-\frac{x}{|x|}\right|^2\de \hausdorff^{n-1}\\
        &\qquad -\int_{\partial^*G \cap \interior{\mathcal{C}}\cap B_r(x_0)}\left|\nu_{G}(x)-\frac{x}{|x|}\right|^2\de \hausdorff^{n-1}\bigg)+ C(n)|G\symmdiff E_j^k|\\
        &\leq P_{\mathcal{C}}(G)+\left(\Lambda+1+C(n)\right)\tilde C(n,v_0)rP_{\mathcal{C}}(E_j^k\symmdiff G)+\frac{2}{k}\left(P_{\mathcal{C}}(E_j^k;B_r(x_0))+P_{\mathcal{C}}(G;B_r(x_0))\right)\\
        &\leq P_{\mathcal{C}}(G)+\left(\left(\Lambda+1+C(n)\right)\tilde C(n,v_0)r+\frac{2}{k}\right)\left(P_{\mathcal{C}}(E_j^k;B_r(x_0))+P_{\mathcal{C}}(G;B_r(x_0))\right),
    \end{split}
    \end{equation*}
    where we used the following rough estimate coming from the isoperimetric inequality in $\mathcal{C}$:
    \[
    \begin{split}
        |G\symmdiff E_j^k|\leq |B_1|^{\frac1n}r|G\symmdiff E_j^k|^{\frac{n-1}{n}}\leq |B_1|^{\frac1n}r\cdot \frac{1}{n|K_{\mathcal{C}}|^{1/n}}\left(P_{\mathcal{C}}(E_j^k;B_r(x_0))+P_{\mathcal{C}}(G;B_r(x_0))\right).
    \end{split}
    \]
    Since $k>4$, choosing $r_{\ref{prop:selection-cones}}=\min\{1,(4\tilde C(n,|K_{\mathcal{C}}|)(\Lambda+1+C(n)))^{-1}\}$ we deduce that $E_j^k$ is a $(K_{\ref{prop:selection-cones}},r_{\ref{prop:selection-cones}})$-quasi-minimizer of the perimeter in $\mathcal{C}$ with $K_{\ref{prop:selection-cones}}=7$. Therefore, we apply \cref{thm:quasi-regularity} to get the density estimates: there exist $D_{\ref{prop:selection-cones}},\rho_{\ref{prop:selection-cones}}\in(0,1)$, depending on $n, v_0$, such that for every $k>4$, every $j\geq 1$, every radius $\rho\in(0,\rho_{\ref{prop:selection-cones}})$ and every point $x\in\partial E_j^k$
    \begin{equation}\label{eq:density-estimates-cone}
        |\mathcal{C}\cap E_j^k\cap B_{\rho}(x)|\geq D_{\ref{prop:selection-cones}}\rho^n, \qquad |\mathcal{C}\cap B_{\rho}(x)\setminus E_j^k|\geq D_{\ref{prop:selection-cones}}\rho^n.
    \end{equation}

    Analogously to the proof of \cref{prop:MinimizationProblemSelection}, we claim that there is $\tilde k_{\ref{prop:selection-cones}}>4$ depending on $n,v_0, C_{\ref{thm:QuantitativeConesFraenkel}}$ such that the following holds. For every subsequence $k_i\to+\infty$, with $k_i\ge \tilde k_{\ref{prop:selection-cones}}$, and any subsequence $j_i\geq k_i$, denoting by $A_i=E_{j_i}^{k_i}$, there holds
    \begin{equation}\label{eq:KeyClaimCones}
        |A_i\symmdiff K_{\mathcal{C}}|\leq \frac{ \bar  C(n,v_0, C_{\ref{thm:QuantitativeConesFraenkel}})}{\sqrt{k_i}} \quad \forall i, \qquad d_{\hausdorff}\left(\overline{\partial A_i\cap \interior{\mathcal{C}}} , \overline{\partial K_{\mathcal{C}}\cap \interior{\mathcal{C}}} \right) \le 
        \frac{C(n,v_0, C_{\ref{thm:QuantitativeConesFraenkel}}) }{k_i^{\frac{1}{2n}}}
        \quad \forall i.
    \end{equation}
    By definition of $A_i$, we have that $\mathcal{G}_{k_i}(A_i)\leq\mathcal{G}_{k_i}(K_{\mathcal{C}})$ for any $i$, and thus
    \begin{equation}\label{eq:subsub-coni}
        P_{\mathcal{C}}(A_i)+\Lambda||A_i|-|K_{\mathcal{C}}||-\frac{1}{k_i}\mu_{\mathcal{C},0}^2(A_i)+|\widetilde b_{\R^m} (A_i)|^2 \leq P_{\mathcal{C}}(K_{\mathcal{C}}),
    \end{equation}
    for any $i$.
    As $k_i> 4$, we use~\eqref{eq:cone-PreConvention} and we deduce that $P_{\mathcal{C}}(A_i)\leq 2P_{\mathcal{C}}(K_{\mathcal{C}})$. Hence, we also control the asymmetry as $\mu_{\mathcal{C},0}^2(A_i)\leq 4P_{\mathcal{C}}(K_{\mathcal{C}})$. Since $\Lambda>2n$, using the parameter $\bar\Lambda=\Lambda-n>n$ in \cref{lem:ProblemCone+LambdaVolumi}, we deduce from~\eqref{eq:subsub-coni} that
    \begin{equation}\label{eq:difference-measure}
        n||A_i|-|K_{\mathcal{C}}||\leq \frac{1}{k_i}\mu_{\mathcal{C},0}^2(A_i)\leq \frac{4P_{\mathcal{C}}(K_{\mathcal{C}})}{k_i} = \frac{4n|K_{\mathcal{C}}|}{k_i},
    \end{equation}
    for any $i$.
    Using \cref{thm:QuantitativeConesFraenkel}, rearranging the terms in~\eqref{eq:subsub-coni} and introducing the parameter $t_i=\frac{|A_i|}{|K_{\mathcal{C}}|}$, we get
    \begin{equation}\label{eq:alpha-bar}
        \frac{t_i^{\frac{n-1}{n}}P_{\mathcal{C}}(K_{\mathcal{C}})}{C_{\ref{thm:QuantitativeConesFraenkel}}}\alpha_{\mathcal{C}}^2(A_i)+|\widetilde b_{\R^m}(A_i)|^2\leq P_{\mathcal{C}}(K_{\mathcal{C}})\left(1-t_i^{\frac{n-1}{n}}\right)+\frac{4P_{\mathcal{C}}(K_{\mathcal{C}})}{k_i},
    \end{equation}
    for any $i$, where we also used the bound $\mu_{\mathcal{C},0}^2(A_i)\leq 4P_{\mathcal{C}}(K_{\mathcal{C}})$.\\
    Notice that $|t_i-1|\leq \frac{4}{k_i}$ for any $i$ thanks to~\eqref{eq:difference-measure}. Therefore, we obtain $\frac{P_{\mathcal{C}}(K_{\mathcal{C}})}{C_{\ref{thm:QuantitativeConesFraenkel}}}\alpha_{\mathcal{C}}^2(A_i)+|\widetilde b_{\R^m}(A_i)|^2\leq \frac{C(n)P_{\mathcal{C}}(K_{\mathcal{C}})}{k_i}$ for any $i$. This shows that, denoting by $x_i\in\R^m$ a point realizing $\alpha_{\mathcal{C}}(A_i)$, we get
    \[
        |A_i\symmdiff (x_i+K_{\mathcal{C}})| \leq \left|A_i\symmdiff (x_i+t_i^{1/n}K_{\mathcal{C}})\right|+\left|t_i^{1/n}K_{\mathcal{C}}\symmdiff K_{\mathcal{C}}\right|\leq \frac{C(n,C_{\ref{thm:QuantitativeConesFraenkel}})}{\sqrt{k_i}}+||A_i|-|K_{\mathcal{C}}||\leq \frac{C(n,C_{\ref{thm:QuantitativeConesFraenkel}})}{\sqrt{k_i}},
    \]
    for any $i$, where $C =C(n,C_{\ref{thm:QuantitativeConesFraenkel}})$ may change from one expression to the other.\\    
Now it can be readily checked that $|\widetilde{b}_{\R^m}(x_i + K_{\mathcal{C}})| \ge c(n,v_0) \min\{1, |x_i|\}$ for any $i$. Moreover from~\eqref{eq:alpha-bar} we also deduce that $|\widetilde b_{\R^m}(A_i)|^2\leq \frac{C(n)}{k_i}$ for any $i$. Hence
    \begin{equation*} \begin{split}
        c(n,v_0) \min\{1, |x_i|\} \le |\widetilde b_{\R^m}(A_i)| + \left| |\widetilde b_{\R^m}(x_i+K_{\mathcal{C}})| -  |\widetilde b_{\R^m}(A_i)|\right|
         &\le \frac{C(n)}{\sqrt{k_i}} + C(n) |A_i\symmdiff (x_i+K_{\mathcal{C}})|
         \\&\le \frac{C(n,C_{\ref{thm:QuantitativeConesFraenkel}})}{\sqrt{k_i}}, \end{split}
    \end{equation*}
for any $i$. 
Hence there is $\tilde k_{\ref{prop:selection-cones}}>4$ depending on $n,v_0, C_{\ref{thm:QuantitativeConesFraenkel}}$ such that the previous estimate yields an upper bound on $|x_i|$ for any $k_i \ge \tilde k_{\ref{prop:selection-cones}}$. This immediately implies that $|A_i\symmdiff K_{\mathcal{C}}|\leq \frac{ \bar C(n,v_0, C_{\ref{thm:QuantitativeConesFraenkel}})}{\sqrt{k_i}}$ for any $i$ with $k_i \ge \tilde k_{\ref{prop:selection-cones}}$.\\
We already know that the sets $A_i$ satisfy the density estimates~\eqref{eq:density-estimates-cone}. Up to decreasing $D_{\ref{prop:selection-cones}}$, depending on $n,v_0$ only, we can assume that the same density estimates hold for $K_{\mathcal{C}}$ as well.
It is then standard to prove that whenever $\sqrt{k_i}>\tfrac{ \bar C(n,v_0, C_{\ref{thm:QuantitativeConesFraenkel}})}{D_{\ref{prop:selection-cones}}\rho_{\ref{prop:selection-cones}}^n}$, then $d_{\hausdorff}(\overline{\partial A_i\cap \interior{\mathcal{C}}}, \overline{\partial K_{\mathcal{C}}\cap \interior{\mathcal{C}}} )\leq \rho_{\ref{prop:selection-cones}}$, hence    $|A_i\symmdiff K_{\mathcal{C}}|\geq D_{\ref{prop:selection-cones}}d_{\hausdorff}(\overline{\partial A_i\cap \interior{\mathcal{C}}}, \overline{\partial K_{\mathcal{C}}\cap \interior{\mathcal{C}}})^n$. Therefore
    \[
        d_{\hausdorff}(\overline{\partial A_i\cap \interior{\mathcal{C}}}, \overline{\partial K_{\mathcal{C}}\cap \interior{\mathcal{C}}})
        \leq 
        C(n,v_0, C_{\ref{thm:QuantitativeConesFraenkel}}) k_i^{-1/2n}
        \qquad 
        \text{when } \sqrt{k_i}>\tfrac{ \bar C(n,v_0, C_{\ref{thm:QuantitativeConesFraenkel}})}{D_{\ref{prop:selection-cones}}\rho_{\ref{prop:selection-cones}}^n}.
    \]
Up to updating $\tilde k_{\ref{prop:selection-cones}}$ with $\max\Big\{ \tilde k_{\ref{prop:selection-cones}}, \Big( \tfrac{ \bar C(n,v_0, C_{\ref{thm:QuantitativeConesFraenkel}})}{D_{\ref{prop:selection-cones}}\rho_{\ref{prop:selection-cones}}^n} \Big)^2 \Big\} $, the claim \eqref{eq:KeyClaimCones} follows.

\medskip

Similarly to the proof of \cref{prop:MinimizationProblemSelection}, the main claim \eqref{eq:KeyClaimCones} together with density estimates imply the existence of a radius $R_{\ref{prop:selection-cones}}=R_{\ref{prop:selection-cones}}(n,v_0, C_{\ref{thm:QuantitativeConesFraenkel}}) \ge 10$ and of an open $(K_{\ref{prop:selection-cones}}, r_{\ref{prop:selection-cones}})$-quasi-minimizer $F_k$ which solves the minimization problem for $\mathcal{G}_k$, such that $F_k \subset B_{R_{\ref{prop:selection-cones}}} (0) \cap \mathcal{C}$, for any $k \ge \tilde k_{\ref{prop:selection-cones}}$. Also, the same arguments presented above for the sets $A_i$ can be repeated for the minimizers $F_k$, leading to the estimates
\begin{equation}\label{eq:StimeQuantitativeFkConi}
        |F_k\symmdiff K_{\mathcal{C}}|\leq \frac{ \bar  C(n,v_0, C_{\ref{thm:QuantitativeConesFraenkel}})}{\sqrt{k}} , \qquad d_{\hausdorff}\left(\overline{\partial F_k\cap \interior{\mathcal{C}}} , \overline{\partial K_{\mathcal{C}}\cap \interior{\mathcal{C}}} \right) \le 
        \frac{C(n,v_0, C_{\ref{thm:QuantitativeConesFraenkel}}) }{k^{\frac{1}{2n}}},
    \end{equation}
for every $k \ge \tilde k_{\ref{prop:selection-cones}}$.

\medskip

In order to improve the estimates in \eqref{eq:StimeQuantitativeFkConi} into $C^1$-bounds, we claim that there exist $\Lambda_{\ref{prop:selection-cones}}'>0, r'_{\ref{prop:selection-cones}}\in(0,1)$ depending on  $n, v_0, C_{\ref{thm:QuantitativeConesFraenkel}}(\mathcal{C})$ such that $F_k$ is a $(\Lambda_{\ref{prop:selection-cones}}', r'_{\ref{prop:selection-cones}} )$-almost minimizer of the perimeter in $\mathcal{C}$ for any $k \ge \tilde k_{\ref{prop:selection-cones}}$.\\
This claim follows at once by combining the minimality of $F_k$, the second estimate in \eqref{eq:StimeQuantitativeFkConi}, and \cref{lemma:mu-C-alternative} (see also the proof of the analogous claim in the proof of \cref{prop:MinimizationProblemSelection}).

\medskip

Consider now the parameter $\epsilon_{\ref{prop:FugledeConi}} = \epsilon_{\ref{prop:FugledeConi}}(n, v_0, C_{\ref{thm:QuantitativeConesFraenkel}})$ given by \cref{prop:FugledeConi}. We can apply the $\eps$-regularity \cref{thm:almost-regularity-cones} on the piecewise $C^2$ cone $\mathcal{C}$, getting constants $\delta_{\ref{thm:almost-regularity-cones}}=\delta_{\ref{thm:almost-regularity-cones}}(\mathcal{C}, \Lambda_{\ref{prop:selection-cones}}', r'_{\ref{prop:selection-cones}} , \epsilon_{\ref{prop:FugledeConi}}) \in(0,1)$ and $\beta_{\ref{thm:almost-regularity-cones}}(\mathcal{C}) \in (0,1/2)$. Now for $k \ge \hat k_{\ref{prop:selection-cones}} := \max\{ \tilde k_{\ref{prop:selection-cones}} , \bar C( n,v_0, C_{\ref{thm:QuantitativeConesFraenkel}})^2 / \delta_{\ref{thm:almost-regularity-cones}}^2 \} $, the first inequality in \eqref{eq:StimeQuantitativeFkConi} implies that $|F_k \Delta K_{\mathcal{C}} |\le \delta_{\ref{thm:almost-regularity-cones}}$. Hence \cref{thm:almost-regularity-cones} implies that $\overline{\partial F_k \cap \interior{\mathcal{C}}}$ is parametrized over $\S^{n-1} \cap \mathcal{C}$ by a $C^{1, \beta_{\ref{thm:almost-regularity-cones}}}$ function $u_k$ with $\|u_k\|_{C^{1, \beta_{\ref{thm:almost-regularity-cones}}}} \le \epsilon_{\ref{prop:FugledeConi}}$.

\medskip

It only remains to prove that, up to enlarging $\hat k_{\ref{prop:selection-cones}} $ into a final number $ k_{\ref{prop:selection-cones}} $ (possibly adding further dependence on $n,v_0, C_{\ref{thm:QuantitativeConesFraenkel}} $), then $|F_k|=|K_{\mathcal{C}}|$. This claim follows exactly as in the proof of \cref{prop:MinimizationProblemSelection} by testing the minimality of $F_k$ with the rescaled set $s_kF_k$ such that $|s_kF_k|= |K_{\mathcal{C}}|$. It is also in this final passage that one fixes $\Lambda= \Lambda_{\ref{prop:selection-cones}}$ depending on  $n,v_0, C_{\ref{thm:QuantitativeConesFraenkel}} $ (see the proof of \cref{prop:MinimizationProblemSelection}).
\end{proof}


\begin{theorem}[Strong quantitative inequality in cones]\label{thm:ConesStrongQuantitative}
Let $\mathcal{C}=\R^m\times\widetilde{\mathcal{C}}\subset\{x_n\geq0\}$ be a piecewise $C^2$ closed convex  cone, where $\tilde{\mathcal{C}}$ does not contain lines. Suppose $|B_1 \cap \mathcal{C}| \ge v_0$, for some $v_0>0$. Then there exists $ C_{\ref{thm:ConesStrongQuantitative}}>0$ depending on $n, v_0, C_{\ref{thm:QuantitativeConesFraenkel}}(\mathcal{C}), \delta_{\ref{thm:almost-regularity-cones}}(\mathcal{C}, \Lambda_{\ref{prop:selection-cones}}', r'_{\ref{prop:selection-cones}} , \epsilon_{\ref{prop:FugledeConi}})$ such that the following holds.

Let $E\subset \mathcal{C}$ be a set of finite perimeter with $|E| \in (0,+\infty)$. Then
\begin{equation*}
    \mu_{\mathcal{C}}^2(E) \le C_{\ref{thm:ConesStrongQuantitative}} D_{\mathcal{C}}(E).
\end{equation*}
\end{theorem}

\begin{proof}
We claim it is possible to take 
\begin{equation*}
    C_{\ref{thm:ConesStrongQuantitative}}
    := 
    \left( 1 + \max\left\{
    k_{\ref{prop:selection-cones}} , \bar k ,
    \frac{2C_{\ref{prop:FugledeConi}}}{v_0^2},
    \frac{4C_{\ref{prop:FugledeConi}}}{n v_0}
    \right\} \right)^2
    \ge \left( 1 + \max\left\{
    k_{\ref{prop:selection-cones}} , \bar k ,
    \frac{2C_{\ref{prop:FugledeConi}}}{|K_{\mathcal{C}}|^2},
    \frac{4C_{\ref{prop:FugledeConi}}}{P_{\mathcal{C}}(K_{\mathcal{C}})}
    \right\} \right)^2,
\end{equation*}
where $\bar k = \bar k (n)>0$ is the least integer such that $\tfrac{n\omega_n}{k^2} - \tfrac1k \le - \tfrac{1}{2k}$ for any $k\ge \bar k$. 
Let $k\in\N$ be the least integer such that
\[
k \ge \max\left\{
    k_{\ref{prop:selection-cones}} , \bar k ,
    \frac{2C_{\ref{prop:FugledeConi}}}{v_0^2},
    \frac{4C_{\ref{prop:FugledeConi}}}{n\omega_n v_0}
    \right\}.
\]
We stress that $k$ will be a fixed parameter in this proof. Since the inequality is scaling invariant, it is sufficient to prove the claim for sets $E\subset\mathcal{C}$ such that  $|E|=|K_{\mathcal{C}}|$. 
Let us suppose by contradiction that there exists a set $E\subset \mathcal{C}$ with $|E| = |K_{\mathcal{C}}|$ and $k^2D_{\mathcal{C}}(E) <  C_{\ref{thm:ConesStrongQuantitative}} D_{\mathcal{C}}(E)<\mu_{\mathcal{C}}^2(E)$.
In particular, $E\neq x'+K_{\mathcal{C}}$ for any $x'\in\R^m$.
Without loss of generality, we can suppose that $\tilde b_{\R^m}(E)=0$.
Similarly to the proof of \cref{thm:capillarityStrongQuantitative}, we have that 
    \begin{equation}\label{eq:DMuZero-cone}
        D_{\mathcal{C}}(E)<\frac{1}{k^2}\mu_{\mathcal{C}}^2(E)\leq \frac{1}{k^2}\mu_{\mathcal{C},0}^2(E).
    \end{equation}
We now apply \cref{prop:selection-cones}. We stress that in this proof we do not need the whole sequence of sets generated by \cref{prop:selection-cones}, since $k$ is fixed in this proof. Since $k\ge k_{\ref{prop:selection-cones}}$, by \cref{prop:selection-cones} there exist constants $\Lambda_{\ref{prop:selection-cones}}>2n$, $R_{\ref{prop:selection-cones}}\ge 10$ such that the auxiliary minimization problem
    \begin{equation*}
        \inf \left\{
    \mathcal{G}_k(F) :=
        P_{\mathcal{C}}(F)+\Lambda_{\ref{prop:selection-cones}} \left||F|-|K_{\mathcal{C}}|\right|-\frac{1}{k}\mu_{\mathcal{C},0}^2(F)+|\widetilde{b}_{\R^m}(F)|^2
        \st 
        F\subset \mathcal{C}, \,\, |F|<+\infty
    \right\}
    \end{equation*}
    admits an open minimizer $F \subset B_{R_{\ref{prop:selection-cones}}}(0)\cap \mathcal{C}$. Additionally $|F|=|K_{\mathcal{C}}|$ and there exists $u\in C^1(\Sigma_{\mathcal{C}})$ parametrizing $\overline{\partial F\cap \interior{\mathcal{C}}}$, with $\|u\|_{C^1}<\epsilon_{\ref{prop:FugledeConi}}$. Testing the minimality of $F$ with $E$, which satisfies $\widetilde b_{\R^m}(E)=0$ and $|E|=|K_{\mathcal{C}}|$, and using \cref{lem:ProblemCone+LambdaVolumi}, we get
     \begin{equation}\label{eq:energy-comparison-FE-cone}
    \begin{split}
        P_{\mathcal{C}}(F)-\frac{1}{k}\mu_{\mathcal{C},0}^2(F)+|\widetilde{b}_{\R^m} (F)|^2
        &=P_{\mathcal{C}}(F)+\Lambda_{\ref{prop:selection-cones}} \left||F|-|K_{\mathcal{C}}|\right|-\frac{1}{k}\mu_{\mathcal{C},0}^2(F)+|\widetilde{b}_{\R^m} (F)|^2
        \\&\leq P_{\mathcal{C}}(E)-\frac{1}{k} \mu_{\mathcal{C},0}^2(E)\\
        &\overset{\eqref{eq:DMuZero-cone}}{\leq} P_{\mathcal{C}}(K_{\mathcal{C}})+\frac{P_{\mathcal{C}} (K_{\mathcal{C}})}{k^2}  \mu_{\mathcal{C},0}^2(E) -\frac{1}{k}\mu_{\mathcal{C},0}^2(E)\\
        &\leq P_{\mathcal{C}}(F)+\Lambda_{\ref{prop:selection-cones}} \left||F|-|K_{\mathcal{C}}|\right|+\frac{P_{\mathcal{C}} (K_{\mathcal{C}})}{k^2}  \mu_{\mathcal{C},0}^2(E)-\frac{1}{k} \mu_{\mathcal{C},0}^2(E)
        \\
        &=  P_{\mathcal{C}}(F)+\frac{P_{\mathcal{C}} (K_{\mathcal{C}})}{k^2}  \mu_{\mathcal{C},0}^2(E)-\frac{1}{k} \mu_{\mathcal{C},0}^2(E),
    \end{split}
    \end{equation}
    Since $k\geq \bar k(n)$, then
    \begin{equation*}
        \frac{P_{\mathcal{C}}(K_{\mathcal{C}})}{k^2}\mu_{\mathcal{C},0}^2(E)-\frac{1}{k}\mu_{\mathcal{C},0}^2(E) \leq -\frac{1}{2k}\mu_{\mathcal{C},0}^2(E).
    \end{equation*}
    Hence, from~\eqref{eq:energy-comparison-FE-cone} we deduce that 
    \begin{equation}\label{eq:mu-btilde-cone}
        0<\mu_{\mathcal{C},0}^2(E)\leq 2\mu_{\mathcal{C},0}^2(F),\qquad |\widetilde b_{\R^m}(F)|^2\leq \frac{1}{k}\mu_{\mathcal{C},0}^2(F).
    \end{equation}
    Since $|F|=|K_{\mathcal{C}}|$, using the first and third lines in~\eqref{eq:energy-comparison-FE-cone} we obtain
    \[
        D_{\mathcal{C}}(F)\leq \frac{\mu_{\mathcal{C},0}^2(F)}{k\cdot P_{\mathcal{C}}(K_{\mathcal{C}})}.
    \]
    Since $\|u\|_{C^1}<\epsilon_{\ref{prop:FugledeConi}}$, we can apply \cref{prop:FugledeConi} to continue this estimate as
    \begin{equation}\label{eq:arew}
        D_{\mathcal{C}}(F)\leq \frac{\mu_{\mathcal{C},0}^2(F)}{k\cdot P_{\mathcal{C}}(K_{\mathcal{C}})}\leq \frac{1}{k\cdot P_{\mathcal{C}}(K_{\mathcal{C}})}C_{\ref{prop:FugledeConi}}\left(D_{\mathcal{C}}(F)+|\barycenter_{\R^m} F|^2\right).
    \end{equation}
    Moreover, since $|F|=|K_{\mathcal{C}}|$ and $F\subset B_{10}(0)$, hence $\barycenter_{\R^m} F=|K_{\mathcal{C}}|^{-1}\widetilde b_{\R^m}(F)$. Therefore, using~\eqref{eq:mu-btilde-cone} and \cref{prop:FugledeConi} again we get that
    \[
        |\barycenter_{\R^m}F|^2\leq \frac{1}{k|K_{\mathcal{C}}|^2}\mu_{\mathcal{C},0}^2(F)\leq \frac{C_{\ref{prop:FugledeConi}}}{k|K_{\mathcal{C}}|^2}\left(D_{\mathcal{C}}(F)+|\barycenter_{\R^m}F|^2\right).
    \]
    Since $k\geq \frac{2C_{\ref{prop:FugledeConi}}}{|K_{\mathcal{C}}|^2}$ we see that $|\barycenter_{\R^m}F|^2\leq D_{\mathcal{C}}(F)$. Plugging this inequality in~\eqref{eq:arew}, since $k\geq \frac{4C_{\ref{prop:FugledeConi}}}{P_{\mathcal{C}}(K_{\mathcal{C}})}$, we deduce that $D_{\mathcal{C}}(F)\leq \frac{1}{2}D_{\mathcal{C}}(F)$. Hence, $D_{\mathcal{C}}(F) = 0$. Recalling that $|F|=|K_{\mathcal{C}}|$ and $|\barycenter_{\R^m}F|^2\leq D_{\mathcal{C}}(F)$, then $F=K_{\mathcal{C}}$, contradicting the first estimate in~\eqref{eq:mu-btilde-cone}.
\end{proof}

\subsection{Strong quantitative inequalities with barycentric asymmetry in cones}

As discussed in \cref{sec:BarycentricCapillarity} in the case of capillarity problems, and by the very same reasons, our approach also allows to derive barycentric versions of quantitative isoperimetric inequalities in convex cones. Here, splitting $\R^n=\R^m\times \R^{n-m}$ as usual, we introduce the $\R^m$-diameter of a set, defined as the diameter of the orthogonal projection of a set on $\R^m$, and denoted by $\diam_{\R^m}E:=\diam(\pi_{\R^m}(E))$. We thus obtain the following result.

\begin{corollary}[Barycentric strong quantitative inequalities in cones]\label{cor:barycentric-ineq-cones}
    Let $\mathcal{C}=\R^m\times\widetilde{\mathcal{C}}\subset\{x_n\geq0\}$ be a piecewise $C^2$ closed convex  cone, where $\tilde{\mathcal{C}}$ does not contain lines. Suppose $|B_1 \cap \mathcal{C}| \ge v_0$, for some $v_0>0$. Then there exists $ C_{\ref{cor:barycentric-ineq-cones}}>0$ depending on $n, v_0, C_{\ref{thm:QuantitativeConesFraenkel}}(\mathcal{C}), \delta_{\ref{thm:almost-regularity-cones}}(\mathcal{C}, \Lambda_{\ref{prop:selection-cones}}', r'_{\ref{prop:selection-cones}} , \epsilon_{\ref{prop:FugledeConi}})$ such that the following holds.
    
    Let $E\subset \mathcal{C}$ be a set of finite perimeter with $|E|=|K_{\mathcal{C}}|$, and let $x^*\in\R^m$ be any point such that $\widetilde{b}_{\R^m}(E-x^*)=0$. Then
    \begin{equation}\label{eq:BaricentricaConiConBtilde}
        |E\symmdiff(K_{\mathcal{C}}+x^*)|^2+\int_{\partial^*E\cap \interior{\mathcal{C}}}\left|\nu_E(x)-\frac{x-x^*}{|x-x^*|}\right|^2d\hausdorff^{n-1}_x \le C_{\ref{cor:barycentric-ineq-cones}} D_{\mathcal{C}}(E).
    \end{equation}
    If $d=\diam_{\R^m}E<+\infty$, then there exists a constant $C_{\ref{cor:barycentric-ineq-cones}}'>0$     
    depending on $n$, $v_0$, $C_{\ref{thm:QuantitativeConesFraenkel}}(\mathcal{C})$, $\delta_{\ref{thm:almost-regularity-cones}}(\mathcal{C}, \Lambda_{\ref{prop:selection-cones}}',$ $r'_{\ref{prop:selection-cones}},$ $\epsilon_{\ref{prop:FugledeConi}})$, $d$
    such that
    \[
        |E\symmdiff(K_{\mathcal{C}}+\barycenter_{\R^m}E)|^2+\int_{\partial^*E\cap \interior{\mathcal{C}}}\left|\nu_E(x)-\frac{x-\barycenter_{\R^m}E}{|x-\barycenter_{\R^m}E|}\right|^2d\hausdorff^{n-1}_x\leq C_{\ref{cor:barycentric-ineq-cones}}'D_{\mathcal{C}}(E).
    \]
\end{corollary}

\appendix

\section{Cornerstones in Calculus of Variations}\label{AppendixA}

\subsection{Regularity theory}

In this work, we exploit results in regularity theory for sets satisfying two possible weak notions of perimeter minimality. The first and weaker one is the notion of $(K,r_0)$-quasi-minimaly.

\begin{definition}\label{def:quasiminimizer}
Let $\mathcal{C}\subset \R^n$ be a closed convex cone with non-empty interior ${\rm int}(\mathcal{C})$. Let $K\ge 1, r_0>0$. A set $E\subset \mathcal{C}$ of locally finite perimeter is a $(K,r_0)$-quasi-minimizer in $\mathcal{C}$ if
    \[
        P(E, B_r(x)\cap {\rm int}(\mathcal{C}))\leq K P(G, B_r(x)\cap {\rm int}(\mathcal{C})),
    \]
for any set $G \subset \mathcal{C}$ such that $G\Delta E \Subset B_r(x)$ for some $r \le r_0$ and some $x \in \mathcal{C}$.
\end{definition}

It is well-known that $(K,r_0)$-quasi-minimizers satisfy density estimates.

\begin{theorem}\label{thm:quasi-regularity}
Let $\mathcal{C}\subset \R^n$ be a closed convex cone with non-empty interior, and let $E\subset \mathcal{C}$ be a $(K,r_0)$-quasi-minimizer in $\mathcal{C}$. Denote $v_0:= \inf_{x \in \mathcal{C}} |\mathcal{C} \cap B_1(x)| $.

Then $E$ has an open representative. Moreover, there exist $\rho_{\ref{thm:quasi-regularity}}>0$, $D_{\ref{thm:quasi-regularity}} \in(0,1)$ depending on $n, K, r_0, v_0$ such that, identifying $E$ with its open representative, there holds
\[
|\mathcal{C}\cap E \cap B_\rho(x)| \ge D_{\ref{thm:quasi-regularity}} \rho^n 
\qquad
\forall\, \rho \le \rho_{\ref{thm:quasi-regularity}}, \, x \in \overline{E}
\]
and
\[
|\mathcal{C}\cap B_\rho(x) \setminus E| \ge D_{\ref{thm:quasi-regularity}} \rho^n
\qquad
\forall\, \rho \le \rho_{\ref{thm:quasi-regularity}}, \, x \in \partial E.
\]
\end{theorem}

A proof of \cref{thm:quasi-regularity} holding in the higher generality of PI spaces can be found in \cite[Theorem 4.2]{ShanmugalingamQuasiminimizers}. Observe that in \cite{ShanmugalingamQuasiminimizers}, the perimeter functional automatically coincides with the relative perimeter in ${\rm int}(\mathcal{C})$, hence the definition of quasiminimal set in \cite[Definition 3.1]{ShanmugalingamQuasiminimizers} coincides with our \cref{def:quasiminimizer}. Alternatively, a proof of \cref{thm:quasi-regularity} follows by adapting the proof of \cite[Theorem 21.11]{MaggiBook} working with $(K,r_0)$-quasi-minimal sets.

\bigskip

The second and stronger notion of generalized perimeter minimizer is that of $(\Lambda, r_0)$-almost minimizer, that we introduce here in the context of weighted anisotropic perimeter functionals. We first recall the definition of elliptic integrands as given in \cite[Definition~1.1]{DePhilippisMaggi2015}.

\begin{definition}\label{def:unif-elliptic-perimeter}
    Given an open set $\Omega\subset \R^n$, we say that $\Phi:\overline \Omega\times \R^n\to[0,+\infty]$ is an elliptic integrand on $\Omega$ if it is lower semicontinuous, $\Phi(x,\cdot)$ is convex, and $\Phi(x,t\nu) = t\Phi(x,\nu)$ for every $x\in \overline \Omega$, $\nu\in\R^n$, $t\geq 0$.
    
    If $\Phi$ is an elliptic integrand and $E\subset \Omega$ is a set of locally finite perimeter in $\Omega$, then we define the associated perimeter
    \[
        P_{\Phi}(E,A)\coloneqq\int_{A\cap \partial^*E}\Phi(x,\nu_E(x))\de\hausdorff^{n-1}(x), 
    \]
    for any open set $A\subset \Omega$.
    Given $\kappa_1\geq 1$ and $\kappa_2\geq 0$, we say that $\Phi$ is a regular elliptic integrand on $\Omega$ with ellipticity constant $k_1$ and Lipschitz constant $k_2$, denoted by
    \[
    \Phi \in \mathcal{E}(\Omega, k_1, k_2),
    \]
    if $\Phi(x,\cdot)\in C^{2,1}(\S^{n-1})$ for every $x\in\overline \Omega$, and for every $x,y\in\overline \Omega$ and $\nu_1,\nu_2\in\S^{n-1}$ there holds
    \begin{equation*}
    \begin{split}
        &\frac{1}{\kappa_1}\leq \Phi(x,\nu_1)\leq \kappa_1,\\
        &|\Phi(x,\nu_1)-\Phi(y,\nu_1)|+|\grad \Phi(x,\nu_1)-\grad\Phi(y,\nu_1)|\leq \kappa_2|x-y|,\\
        &|\grad\Phi(x,\nu_1)|+\|\grad^2\Phi(x,\nu_1)\|+\frac{\|\grad^2\Phi(x,\nu_1)-\grad^2\Phi(x,\nu_2)\|}{|\nu_1-\nu_2|}\leq \kappa_1,\\
        &\Braket{\grad^2\Phi(x,\nu_1)\nu_2, \nu_2 } \geq\frac{1}{\kappa_1}|\nu_2-(\nu_2\cdot\nu_1)\nu_1|^2,
    \end{split}
    \end{equation*}
    where $\nabla\Phi$, $\nabla^2 \Phi$ here denote gradient and Hessian with respect to the $\nu$-variable.
\end{definition}

We can thus define a $(\Lambda,r_0)$-almost minimizer for a perimeter $P_\Phi$ as above.

\begin{definition}\label{def:almostminimizer}
    Let $\Omega\subset\R^n$ be an open set. Let $\Phi: \overline{\Omega}\times \R^n\to[0,+\infty]$ be an elliptic integrand on $\Omega$. 
    Let $\Lambda,r_0>0$. We say that a set $E\subset \Omega$ of locally finite perimeter is a $(\Lambda,r_0)$-almost minimizer of $P_{\Phi}$ in $\Omega$ if
    \[
        P_{\Phi}(E, B_r(x)\cap \Omega) \leq P_{\Phi}(G , B_r (x)\cap \Omega)+\Lambda|E\symmdiff G|
    \]
    whenever $E\symmdiff G\Subset B_r(x)\cap \overline{\Omega}$ for some $r\le r_0$ and $x \in \overline{\Omega}$.

    If $\Phi(x,\nu) = |\nu|$, i.e., $P_\Phi$ coincides with the usual (relative) perimeter in $\Omega$, we will just say that $E$ is a $(\Lambda,r_0)$-almost minimizer of the perimeter in $\Omega$.
\end{definition}

It is well-understood that given a sequence of $(\Lambda,r_0)$-almost minimizers of the perimeter in $\R^n$, if such a sequence converges in $L^1$ to a smooth set, then the convergence can be improved to $C^{1,\beta}$ convergence of the topological boundaries (see \cite{MaggiBook} and references therein). The theory has been generalized  to $(\Lambda,r_0)$-almost minimizers for perimeters defined by elliptic integrands in a half-space in \cite{DePhilippisMaggi2015}. In particular, combining Theorem~3.1 and Remark~3.6 in \cite{DePhilippisMaggi2015}, one gets the following theorem, which yields convergence in $C^{1,\beta}$ up to the boundary for sequences of $(\Lambda,r_0)$-almost minimizers corresponding to a sequence of elliptic integrands with uniform ellipticity bounds.

\begin{theorem}\label{thm:almost-regularity}
Let $\Phi_k:H\times \R^n\to[0,+\infty)$ be a sequence of elliptic integrands such that there exist $l_1\ge1, l_2>0$ such that $\Phi_k \in \mathcal{E}(H^+, l_1, l_2)$ for any $k$. Suppose also that $\Phi_k$ converges to an elliptic integrand $\Phi$ in $C^3_{\rm loc}(H \times \S^{n-1})$ as $k\to\infty$.
    
Let $\Lambda, r_0>0$. For any $k\in\N$, let $E_k\subset H$ be a $(\Lambda,r_0)$-almost minimizer of $P_{\Phi_k}$ in $H^+$, and assume that $E_k\to B^\lambda$ in $L^1_{\rm loc}(H)$.
    
Then $\overline{\partial E_k \cap H^+}$ is a hypersurface with boundary of class $C^{1,1/2}$ for $k$ large enough, and $\overline{\partial E_k \cap H^+}\to \overline{\partial B^\lambda \cap H^+}$ in $C^{1,\beta}$ for any $\beta \in(0,1/2)$, in the sense of \cref{def:DistanzaC1}.
\end{theorem}

In the second part of the work, we also need to apply an analogous $\eps$-regularity theorem for $(\Lambda, r_0)$-almost minimizers of the perimeter in convex cones. Assuming that the boundary of a convex cone is piecewise $C^2$, one can apply the theory developed in \cite{EdelenLi22} to derive the following result (see also \cite{GruterJost86} for earlier $\eps$-regularity results in the free boundary context).

\begin{theorem}\label{thm:almost-regularity-cones}
Let $\mathcal{C}\subset\R^n$ be a closed convex cone with non-empty interior, and suppose that $ \S^{n-1} \cap \mathcal{C}$ has piecewise $C^2$ boundary. Let $\Lambda\geq0$, $r_0, \eps>0$. Then there exist $\delta_{\ref{thm:almost-regularity-cones}}= \delta_{\ref{thm:almost-regularity-cones}}(\mathcal{C}, \Lambda, r_0, \eps)\in(0,1)$ and $\beta_{\ref{thm:almost-regularity-cones}}=\beta_{\ref{thm:almost-regularity-cones}}(\mathcal{C})\in(0,1/2)$ such that the following holds.

Let $E\subset \mathcal{C}$ be a $(\Lambda,r_0)$-almost minimizer of the perimeter in 
$\interior{\mathcal{C}}$ such that $|E\Delta (B_1 \cap \mathcal{C})| \le \delta_{\ref{thm:almost-regularity-cones}}$. Then there exists $u:\S^{n-1}\cap \mathcal{C}\to(-1,1)$ such that
    \[
        \overline{\partial E\cap \interior{\mathcal{C}}}=\{(1+u(x))x:x\in\S^{n-1}\cap \mathcal{C}\},
    \]
with $\|u\|_{C^{1,\beta_{\ref{thm:almost-regularity-cones}}}}\le \eps$.
%
    
\end{theorem}

We just give a sketch of the proof, commenting on the main observations leading to the application of the main regularity result from \cite{EdelenLi22}.

\begin{proof}
By compactness and by a covering argument, it is sufficient to prove the claimed $C^{1,\beta}$ bound just in some neighborhood of a point $p_0 \in \partial \mathcal{C} \cap \S^{n-1}$ with $|p_0|=1$. Let $\mathcal{C}_0:= \mathcal{C}-p_0$.

By the density estimates in \cref{thm:quasi-regularity}, it is standard to prove that the Hausdorff distances $d_{\hausdorff}(E, B_1\cap \mathcal{C})$, $d_{\hausdorff}(\overline{\partial E\cap \interior{\mathcal{C}}}, \S^{n-1} \cap \mathcal{C})$ are bounded from above by a dimensional power of $|E\Delta (B_1 \cap \mathcal{C})|$, up to multiplicative constant depending on $\Lambda, r_0, n, |B_1\cap \mathcal{C}|$.

By assumption, there exists $r_1>0$ such that for any $\rho \in(0,r_1]$ the set $B_2(0) \cap \partial \left( \tfrac{1}{\rho}\mathcal{C}_0 \right)$ is the boundary of a curved polyhedral cone domain modeled on some polyhedral cone domain $\Omega^{(0)}$ (independent of $\rho$) and belonging to the class $\mathcal{D}_{\frac12 \delta({\Omega^{(0)}})}(\Omega^{(0)})$, where $ \delta({\Omega^{(0)}})$ is given by \cite[Theorem 3.1]{EdelenLi22} and the terminology is as in \cite[Definition 2.1, Definition 2.5]{EdelenLi22}.

Moreover, notice that $M:= \partial^*E \cap \interior{\mathcal{C}}$ defines an integral varifold in $B_2$ with free boundary in $\interior{\mathcal{C}}$, such that furthermore $\|H_{M}^{\rm tan}\|_\infty \le \Lambda$, where notation and terminology is as in \cite[Section 2.4]{EdelenLi22}. In fact, such a claim is well-known and it follows, e.g., by an adaptation of \cite[Lemma A.2]{PascalePozzettaQuantitative}.
Hence, there is $r_2$, depending on $\Lambda$, such that for any $\rho\in(0,r_2]$, calling $M_\rho:= \tfrac{1}{\rho}(M - p_0)$, then $\|H_{M_{\rho}}^{\rm tan}\|_\infty \le \rho\Lambda \le \tfrac12 \delta({\Omega^{(0)}})$.

Taking into account all the previous observations and since $\S^{n-1} \cap \mathcal{C}$ is smooth (with piecewise $C^2$ boundary), it follows that for any $\eta$ we can fix $\delta_{\ref{thm:almost-regularity-cones}}= \delta_{\ref{thm:almost-regularity-cones}}(\mathcal{C}, \Lambda, r_0, \eta)$ and $\rho\le \min\{r_0,r_1,r_2\}$ such that the excess $E$ of the varifold $M_\rho$ considered in the Allard-type regularity theorem \cite[Theorem 3.1]{EdelenLi22} satisfies
\[
E \le \eta \le \delta({\Omega^{(0)}})^2,
\]
which yields assumption (3.1) in \cite[Theorem 3.1]{EdelenLi22}.\\
On the other hand, it follows by an easy contradiction argument that, if $\delta_{\ref{thm:almost-regularity-cones}}, \rho$ are possibly chosen smaller enough depending on $\mathcal{C}$, then $M_\rho$ also satisfies assumption (3.2) in \cite[Theorem 3.1]{EdelenLi22}. Hence we can apply \cite[Theorem 3.1]{EdelenLi22}, which implies the existence of a $C^{1, \beta_{\ref{thm:almost-regularity-cones}}}$ function $u_{p_0}$ defined in a neighborhood of $p_0$ in $\S^{n-1} \cap \mathcal{C}$ that parametrizes $\overline{\partial E \cap \interior{\mathcal{C}}}$ in a neighborhood of $p_0$. Moreover, by \cite[Theorem 3.1]{EdelenLi22}, there is $\beta_{\ref{thm:almost-regularity-cones}}\in(0,1)$ such that
\[
\|u_{p_0}\|_{C^{1,\beta_{\ref{thm:almost-regularity-cones}}}}\le C(n, c(\Omega^{(0)})) E^{\frac12} \le C(n, c(\Omega^{(0)})) \eta^{\frac12},
\]
where $c(\Omega^{(0)})$ is given by \cite[Theorem 3.1]{EdelenLi22}. Choosing $\eta = \eps^2/C(n, c(\Omega^{(0)}))^2$, the statement follows.
\end{proof}

\subsection{Ekeland's variational principle}

In the work, we also exploit the following version of the Ekeland's variational principle.

\begin{theorem}[{Ekeland's variational principle \cite[Corollary 1.4.2]{ZalinescuBOOK}}]\label{thm:Ekeland}
Let $(\mathbb{X},\d)$ be a complete metric space and let $f:\mathbb{X}\to (-\infty,+\infty]$ be a lower semicontinuous function bounded from below. Let $\eps>0$ and let $x_0\in \mathbb{X}$ be such that $f(x_0) \le  \eps + \inf_{\mathbb{X}} f$. Then, for any $s>0$ there exists $v \in \mathbb{X}$ such that
\begin{equation*}
    f(v) \le f(x_0) ,
    \qquad
    f(v) \le f(x) + \frac{\eps}{s} \d(x,v) \quad  \forall\, x \in \mathbb{X}.
\end{equation*}
\end{theorem}

\section{Technical toolbox}\label{AppendixB}

In this appendix, we collect some elementary technical identities and estimates.

\begin{lemma}\label{lemma:formulas}
    Let $f:\S_+^{n-1}\to (0,+\infty)$ be a $C^1$ function, and let $F:\S_+^{n-1}\to\R^n$ be the embedding $F(x) = f(x)x$. 
    Denote $E := \{ x \in H \st |x| < f(x)\}$.\\    
    Then the tangential Jacobian of $F$ is given by
    \[
        J^T F = f^{n-1} \sqrt{1+\frac{|\grad f|^2
        }{f^2}}.
    \]
    In particular
    \[
    \int_{\partial E \cap H^+} g \de \hausdorff^{n-1} = \int_{\S^{n-1}_+} g(f(x) x) \,  f^{n-1} \sqrt{1+\frac{|\grad f|^2
        }{f^2}} \de \hausdorff^{n-1}(x),
    \]
    for any continuous function $g:\R^n\to\R^n$.\\
    The outer unit normal $\nu_E$ of $E$ is given by
    \[
    \nu_E(f(x)x) = \frac{f(x)x - \nabla f (x)}{( f^2(x) + |\nabla f (x)|^2)^{\frac12}},
    \]
    where $\nabla f$ is the tangential gradient of $f$ on $\S^{n-1}_+$, for any $x \in \S^{n-1}_+$.\\
    Moreover
    \[
    \int_E g = \int_{\S_+^{n-1}}\int_0^{f(x)} g(r x) \, r^{n-1}\de r \de \hausdorff^{n-1}(x) ,
    \]
    for any continuous function $g:\R^n\to\R^n$. In particular
    \[
    |E| = \int_{\S_+^{n-1}}\frac{f^n}{n}\de \hausdorff^{n-1}.
    \]
    If also $f=w_{\lambda}+u$, with $u:\S_+^{n-1}\to\R$ of class $C^1$ such that $f>0$ on $\S^{n-1}_+$, then
    \[
    \begin{split}
        P_{\lambda}(E) &= \int_{\S_+^{n-1}} (1-\lambda\Braket{\nu_E,e_n})f^{n-1}\sqrt{1+\frac{|\grad f|^2}{f^2}}\de \hausdorff^{n-1}\\
        &=\int_{\S_+^{n-1}}(1-\lambda\Braket{\nu_E,e_n})(w_{\lambda}+u)^{n-1}\sqrt{1+\frac{|\grad w_{\lambda}+\grad u|^2}{|w_{\lambda}+u|^2}}\de \hausdorff^{n-1},
    \end{split}
    \]
    and
    \begin{equation}\label{eq:VolumeExpansion}
    \begin{split}
        |E| &= \int_{\S_+^{n-1}}\left[\frac{w_{\lambda}^n}{n}+w_{\lambda}^{n-1}u + \frac{n-1}{2}w_{\lambda}^{n-2}u^2+O(u^3)\right]\de \hausdorff^{n-1}(x),
    \end{split}
    \end{equation}
    as $u\to0$.
\end{lemma}

\begin{proof}
We just need to prove formulas for the tangential Jacobian of $F$ and for the unit normal $\nu_E$.
Denote $\Sigma = F(\S^{n-1}_+) = \overline{\partial E \cap H^+}$.
The differential $\d F_x: T_x\S^{n-1}_+ \to T_{F(x)}\Sigma$ of $F$ is given by
\[
\d F_x(v) = f(x) v + \Braket{\nabla f(x), v}x = \left( f(x) \iota + x \otimes \nabla f(x)
\right) v,
\]
where $\iota:T_x\S^{n-1}_+\hookrightarrow\R^n$ is the natural injection. Hence
\[
\begin{split}
    J^TF
    &= |\det ( (\d F_x)^* \circ \d F_x )|^{\frac12}
    = \left| \det \left( 
    ( f(x) \iota^* + \nabla f(x) \otimes x )
    (f(x) \iota + x \otimes \nabla f(x) )
    \right) \right|^{\frac12}.
\end{split} 
\]
The adjoint $\iota^*: \R^n \to T_x\S^{n-1}_+$ is the projection onto $T_x\S^{n-1}_+$, indeed
\[
\Braket{\iota^* w, v} = \Braket{w, \iota v} = \Braket{w, v}
=\begin{cases}
    \Braket{w,v} & w \in T_x\S^{n-1}_+, \\
    0 & w \in (T_x\S^{n-1}_+)^\perp,
\end{cases}
\]
for any $w \in \R^n$, $v \in T_x\S^{n-1}_+$. Therefore
\begin{equation*}
    \begin{split}
        J^TF
    &=\left| \det \left( 
    f^2(x)  \iota^* \iota + f(x) \, \iota^* \, x\otimes\nabla f(x) + f(x) \nabla f(x) \otimes x \, \iota + \Braket{x,x} \nabla f(x) \otimes \nabla f(x)
    \right) \right|^{\frac12}\\
    &= \left| \det \left( 
    f^2(x) \, \Id_{T_x\S^{n-1}_+} +  \nabla f(x) \otimes \nabla f(x)
    \right) \right|^{\frac12}\\
    & = \left| (f^2(x) + |\nabla f(x)|^2 ) (f^2(x))^{n-2}) \right|^{\frac12}.
    \end{split}
\end{equation*}

We now prove the formula for $\nu_E$. It is sufficient to prove that $f(x)x- \nabla f(x)$ is orthogonal to $T_{F(x)}\Sigma$ for any $x \in \S^{n
-1}_+$. Indeed, for any $v \in T_x\S^{n-1}_+$ we have
\[
\begin{split}
\Braket{\d F_x(v), f(x)x- \nabla f(x)}
&= \Braket{f(x) v + \Braket{\nabla f(x), v}x, f(x)x- \nabla f(x)} =0.
\end{split}
\]
Also, $\nu_E$ points outwards because $\Braket{\nu_E(F(x)), x}>0$ since $f(x)>0$ for any $x \in \S^{n-1}_+$.
\end{proof}

\begin{lemma}[Nash-type inequality]\label{lemma:NashIneq}
Let $\sigma_0>0$ and $n\ge 2$. Then there exists $C_{\ref{lemma:NashIneq}}(n,\sigma_0)>0$ and $\alpha_n>1$ such that the following holds.

Let $\Sigma\subset \S^{n-1}_+$ be a closed geodesically convex set with $\hausdorff^{n-1}(\Sigma)\ge \sigma_0$. Then for any $\delta\in(0,1)$ and any $u \in {\rm Lip}(\Sigma)$ there holds
\begin{equation*}
    \int_{\Sigma} u^2 \le \delta \int_\Sigma |\nabla u|^2 + \frac{C_{\ref{lemma:NashIneq}}}{\delta^{\alpha_n}} \left( \int_\Sigma |u| \right)^2.
\end{equation*}
\end{lemma}

\begin{proof}
The proof is well-known to the experts. For the convenience of the reader, we present a classical argument outlined by \cite{Humbert01Nash}.
Fix $k_0 >2$, define
\[
k:= \begin{cases}
    k_0 & \text{ if }n=2,3, \\
    n & \text{ if } n \ge 4.
\end{cases}
\]
It is well-known that a convex subset such as $\Sigma$ enjoys a Sobolev inequality of the following form. There exists $C_{\rm Sob}=C_{\rm Sob}(n,\sigma_0)>0$ such that
\begin{equation}\label{eq:SobolevSuConvessi}
    \|u\|_{L^{\frac{2k}{k-2}}(\Sigma)}^2 \le C_{\rm Sob} ( \|u\|_{L^2(\Sigma)}^2 + \|\nabla u\|_{L^2(\Sigma)}^2 ) ,
\end{equation}
for any $u \in {\rm Lip}(\Sigma)$. Indeed, if $n=2$ then $\Sigma$ is isometric to an interval having length in the interval $[\sigma_0, \pi]$. 
Instead, if $n\ge 3$, if $\partial\Sigma$ is smooth, then $\Sigma$ is a smooth manifold with boundary, Ricci curvature bounded below by $(n-2)$, and volume bounded below by $\sigma_0$. Then \eqref{eq:SobolevSuConvessi} follows from classical Sobolev embeddings keeping track of the dependence of constants (see, e.g., proofs in \cite{HebeyNonlinearAnalysis, ChavelBook}).
More generally, if $\partial\Sigma$ is not a smooth submanifold, it is well-known than $\Sigma$ is an Alexandrov space with positive curvature \cite{AlexanderKapovitchPetrunin}, also having diameter and volume bounded both from above and from below away from zero in terms of $\sigma_0$. Also, the volume measure on $\Sigma$ is doubling and Ahlfors regular, i.e., volumes of relative balls in $\Sigma$ are comparable to the $(n-1)$-th power of the radius, with constants depending on $n, \sigma_0$ only. Hence $\Sigma$ supports a $p$-Poincaré inequality in the sense of \cite[Eq. (5)]{HajlaszKoskela} with
\begin{equation*}
    p= \begin{cases}
        k_0' & \text{ if } n=3,\\
        2 & \text{ if } n\ge 4,
    \end{cases}
\end{equation*}
where $\tfrac{1}{k_0'}=1-\tfrac{1}{k_0}$, and with a Poincaré constant depending only on the dimension \cite{RajalaPoincare}. Hence \eqref{eq:SobolevSuConvessi} follows from \cite[Theorem 5.1, Eq. (22)]{HajlaszKoskela} applied with $s=n$ and $r=\pi$.

Applying the H\"{o}lder inequality with exponents $p=\tfrac{k}{4} + \tfrac12 >1$ and $p'=1+\tfrac{4}{k-2}$, we estimate
\begin{equation*}
\begin{split}
\left(\int_\Sigma u^2 \right)^{1+\frac2k} 
&= \left(\int_\Sigma |u|^{\frac1p} |u|^{2-\frac1p} \right)^{1+\frac2k} 
\le
\left( \int_\Sigma |u| \right)^{\frac4k}
\left( \int_\Sigma |u|^{\frac{2k}{k-2}} \right)^{\frac{k-2}{k}} 
\\
&\overset{\eqref{eq:SobolevSuConvessi}}{\le} 
C_{\rm Sob}  \left( \int_\Sigma |u| \right)^{\frac4k}  \left(\int_\Sigma u^2 + \int_\Sigma |\nabla u|^2 \right) .
\end{split}
\end{equation*}
Applying the Young inequality with exponents $q=1+\tfrac2k$ and $q'=1+\tfrac{k}{2}$, we find
\begin{equation*}
\begin{split}
\left(\int_\Sigma u^2 \right)^{1+\frac2k} 
&\leq \eta^{1+\frac2k} \left(\int_\Sigma |\nabla u|^2 \right)^{1+\frac2k} + \frac{C(n,\sigma_0)}{\eta^{1+\tfrac{k}{2}}}  \left(\int_\Sigma |u| \right)^{2+\frac4k} + \eta^{1+\tfrac{2}{k}}\left(\int_\Sigma u^2 \right)^{1+\frac2k} ,
\end{split}
\end{equation*}
for any $\eta>0$, for some constant $C(n,\sigma_0)>0$.  Choosing $\eta=\eta(n,\sigma_0,\delta)$ sufficiently small, we can absorb the summand $\eta^{1+\tfrac{2}{k}}\Big(\int_\Sigma u^2 \Big)^{1+\frac2k} $ on the left hand side. Eventually raising the inequality to the power $1/(1+\tfrac2k) <1$ the statement follows with $\alpha_n:= \tfrac{k}{2}$.
\end{proof}

\bigskip
\begin{center}
\textsc{Acknowledgements} 
\end{center}

This project was partially developed while D.C. was in Napoli, and he whishes to thank the Mathematics department ``Renato Caccioppoli'' at Università degli Studi di Napoli ``Federico II'' for the hospitality.
M.P. is grateful to Kenneth DeMason for many useful conversations around the topic of the work, to Guido De Philippis for suggesting reference \cite{EdelenLi22}, and to Nicola Fusco for pointing out useful references. 

The authors are members of INdAM - GNAMPA.
The second and third authors have been partially supported by INdAM - GNAMPA Project 2025 \emph{Flussi geometrici su variet\`a: criteri di esistenza e applicazioni geometriche} (Code CUP \#E5324001950001\#).

This research was funded in part by the Austrian Science Fund (FWF) [grant DOI \href{https://www.fwf.ac.at/en/research-radar/10.55776/EFP6}{10.55776/EFP6}]. For open access purposes, the authors have applied a CC BY public copyright license to any author-accepted manuscript version arising from this submission.

\bigskip

\noindent\textbf{Conflicts of interest.} The authors declare no conflicts of interest.

\noindent\textbf{Data availability statement.} Not applicable.

\sloppy
\printbibliography[title={References}]
\end{document}